\documentclass{amsart}
\usepackage{amssymb, amsmath,amsthm }
\usepackage{hyperref}
\usepackage{stmaryrd}
\usepackage[all]{xy}
\newtheorem{prop}{Proposition}[section]

\newtheorem{lem}[prop]{Lemma}
\newtheorem{thm}[prop]{Theorem}

\newtheorem{cor}[prop]{Corollary}

\theoremstyle{remark}
\newtheorem{remar}[prop]{Remark}
\theoremstyle{definition}
\newtheorem{defi}[prop]{Definition}

\DeclareMathAlphabet{\mathpzc}{OT1}{pzc}{m}{it}

\DeclareMathOperator{\End}{End}
\DeclareMathOperator{\Hom}{Hom}

\DeclareMathOperator{\Ind}{Ind}
\DeclareMathOperator{\cInd}{c-Ind}

\DeclareMathOperator{\Sym}{Sym}
\DeclareMathOperator{\Sm}{\underline{Sym}}
\DeclareMathOperator{\Ass}{Ass}

\DeclareMathOperator{\GL}{GL}
\DeclareMathOperator{\SL}{SL}
\DeclareMathOperator{\Ker}{Ker}
\DeclareMathOperator{\Coker}{Coker}
\DeclareMathOperator{\WD}{WD}
\DeclareMathOperator{\Gal}{Gal}

\DeclareMathOperator{\soc}{soc}

\DeclareMathOperator{\Irr}{Irr}

\DeclareMathOperator{\Spec}{Spec}
\DeclareMathOperator{\MaxSpec}{m-Spec}
\DeclareMathOperator{\supp}{Supp}

\DeclareMathOperator{\Mod}{Mod}

\DeclareMathOperator{\Sp}{Sp}
\DeclareMathOperator{\Rep}{Rep}

\DeclareMathOperator{\St}{St}

\DeclareMathOperator{\Ext}{Ext}

\DeclareMathOperator{\Ban}{Ban}

\DeclareMathOperator{\dualcat}{\mathfrak C}

\DeclareMathOperator{\ann}{ann}
\DeclareMathOperator{\wTor}{\mathcal{T}\mathit{or}}
\newcommand{\cIndu}[3]{\cInd_{#1}^{#2}{#3}}

\newcommand{\Indu}[3]{\Ind_{#1}^{#2}{#3}}

\newcommand{\Qp}{\mathbb {Q}_p}
\newcommand{\Zp}{\mathbb{Z}_p}
\newcommand{\Qpbar}{\overline{\mathbb{Q}}_p}

\newcommand{\HH}{\mathcal H}

\newcommand{\Eins}{\mathbf 1}

\newcommand{\ZZ}{\mathbb Z}
\newcommand{\VV}{\mathbf V}

\newcommand{\Fp}{\mathbb F_p}

\newcommand{\II}{\mathcal I}

\newcommand{\mm}{\mathfrak m}

\newcommand{\st}{\mathrm{st}}

\newcommand{\pr}{\mathrm{pr}}

\newcommand{\wP}{\widetilde{P}}
\newcommand{\wE}{\widetilde{E}}

\newcommand{\OO}{\mathcal O}

\DeclareMathOperator{\wtimes}{\widehat{\otimes}}

\newcommand{\cV}{\check{\mathbf{V}}}

\newcommand{\nn}{\mathfrak n}

\newcommand{\md}{\mathrm m}

\newcommand{\pp}{\mathfrak p}

\newcommand{\br}[1]{\llbracket #1\rrbracket}
\newcommand{\qq}{\mathfrak{q}}

\newcommand{\ana}{\mathrm{an}}
\newcommand{\sm}{\mathrm{sm}}
\newcommand{\pro}{\mathrm{pro}}
\newcommand{\lc}{\mathrm{lc}}
\newcommand{\adm}{\mathrm{adm}}
\newcommand{\alg}{\mathrm{alg}}
\newcommand{\cont}{\mathrm{cont}}

\newcommand{\mR}{A}
\newcommand{\mM}{\mathrm M}
\newcommand{\un}{\mathrm{un}}

\newcommand{\lfin}{\mathrm{l.fin}}
\DeclareMathOperator{\LLL}{LL}
\DeclareMathOperator{\pdim}{projdim}

\title{On the Breuil--M\'{e}zard conjecture}
\author{Vytautas Pa\v{s}k\={u}nas}
\date{\today.}
\begin{document} 
\begin{abstract} We give a new local proof of the Breuil-M\'ezard conjecture for two dimensional representations of the absolute Galois group of $\Qp$, 
when $p\ge 5$ and the representation has scalar endomorphisms.
\end{abstract}
\maketitle
\tableofcontents

\section{Introduction}\label{introduction}
Let $p$ be a prime number, $\Qp$ the field of $p$-adic numbers, $L$ a finite extension of $\Qp$ with the ring of integers 
$\OO$, uniformizer $\varpi$ and residue field $k$. Let $\rho:G_{\Qp}\rightarrow \GL_2(k)$ be a continuous representation 
of the absolute Galois group of $\Qp$ such that $\End_{G_{\Qp}}(\rho)=k$.  Let $R_{\rho}$ be the universal deformation ring of $\rho$ and $\rho^{\mathrm{un}}:G_{\Qp}\rightarrow \GL_2(R_{\rho})$ the universal deformation.  If $\nn\in \MaxSpec R_{\rho}[1/p]$, where  $\MaxSpec$ denotes the set of maximal ideals, then the residue field $\kappa(\nn)$ is a finite extension of $L$. Let $\OO_{\kappa(\nn)}$ be the ring of integers in $\kappa(\nn)$. By 
specializing the universal deformation at $\nn$, we obtain a continuous representation $\rho^{\mathrm{un}}_{\nn}: G_{\Qp}\rightarrow \GL_2(\OO_{\kappa(\nn)})$, which reduces to $\rho$ modulo 
the maximal ideal of $\OO_{\kappa(\nn)}$. Since $\kappa(\nn)$ is a finite extension of $L$, $\rho^{\mathrm{un}}_{\nn}$ lives in Fontaine's world of $p$-adic Hodge theory. We  fix a $p$-adic Hodge type $(\mathbf w, \tau, \psi)$ and investigate the locus of points in $\MaxSpec R_{\rho}[1/p]$ for which $\rho^{\mathrm{un}}_{\nn}$ is of type   $(\mathbf w, \tau, \psi)$,  where $\mathbf w=(a,b)$ is a pair of 
integers with $b>a$, $\tau: I_{\Qp}\rightarrow \GL_2(L)$ is a representation of the inertia subgroup with an open kernel and $\psi:G_{\Qp}\rightarrow \OO^{\times}$ a continuous character, such 
that $\psi\equiv \det \rho \pmod{\varpi}$, $\psi|_{I_{\Qp}}= \varepsilon^{a+b}\det\tau$, where $\varepsilon$ is the $p$-adic cyclotomic character. If $\rho^{\mathrm{un}}_{\nn}$ is potentially semi-stable
we let $\WD(\rho^{\mathrm{un}}_{\nn})$ be the Weil-Deligne representation associated to $\rho^{\mathrm{un}}_{\nn}$ by Fontaine in \cite{fontaine}.
In this case we say that $\rho^{\mathrm{un}}_{\nn}$ is of type $(\mathbf w, \tau, \psi)$ if its Hodge-Tate weights are equal to $\mathbf w$, the determinant equal to $\psi$ and $\WD(\rho^{\mathrm{un}}_{\nn})|_{I_{\Qp}}\cong \tau$.

In \cite{henniart}, Henniart has shown the existence and uniqueness of a smooth irreducible  representation $\sigma(\tau)$ (resp. $\sigma^{\mathrm{cr}}(\tau)$) of $K:=\GL_2(\Zp)$ on an $L$-vector space, such that if $\pi$ is a smooth absolutely irreducible infinite dimensional  representation of $G:=\GL_2(\Qp)$  and  $\LLL(\pi)$ is the Weil-Deligne representation attached to $\pi$ by the classical local Langlands  correspondence then $\Hom_K(\sigma(\tau), \pi)\neq 0$ (resp. $\Hom_K(\sigma^{\mathrm{cr}}(\tau), \pi)\neq 0$) if and only if $\LLL(\pi)|_{I_{\Qp}}\cong \tau$ (resp. $\LLL(\pi)|_{I_{\Qp}}\cong \tau$ and the monodromy operator $N=0$). We have $\sigma(\tau)\cong \sigma^{\mathrm{cr}}(\tau)$ in all cases, except if $\tau\cong \chi\oplus \chi$, then $\sigma(\tau)\cong \tilde{\st}\otimes \chi\circ \det$ and $\sigma^{\mathrm{cr}}(\tau)\cong \chi\circ\det$, where $\tilde{\st}$ is the Steinberg representation of $\GL_2(\Fp)$, and we view $\chi$ as a character of $\Zp^{\times}$ via the local class field theory.

We let $\sigma(\mathbf w, \tau):=\sigma(\tau)\otimes \Sym^{b-a-1} L^2\otimes \det^a$. Then $\sigma(\mathbf w, \tau)$ is a finite dimensional $L$-vector space. Since $K$ is compact 
and the action of $K$ on $\sigma(\mathbf w, \tau)$ is continuous, there is a $K$-invariant $\OO$-lattice $\Theta$ in $\sigma(\mathbf w, \tau)$. Then $\Theta/(\varpi)$ is a smooth 
finite length $k$-representation of $K$, and we let $\overline{\sigma(\mathbf w, \tau)}$ be its semi-simplification. One may show that $\overline{\sigma(\mathbf w, \tau)}$ does not depend on the choice of a lattice.  For each smooth irreducible  $k$-representation $\sigma$ of $K$ we let $m_{\sigma}(\mathbf w, \tau)$ be the multiplicity with which $\sigma$ occurs in $\overline{\sigma(\mathbf w, \tau)}$. We let $\sigma^{\mathrm{cr}}(\mathbf w, \tau):=\sigma^{\mathrm{cr}}(\tau)\otimes  \Sym^{b-a-1} L^2\otimes \det^a$ and let $m^{\mathrm{cr}}_{\sigma}(\mathbf w, \tau)$ be the multiplicity of $\sigma$ in $\overline{\sigma^{\mathrm{cr}}(\mathbf w, \tau)}$.

Recall, see for example \cite[\S V.A]{mult},  that the group of $d$-dimensional cycles $\mathcal Z_d(R)$ of  a noetherian ring $R$  is a free abelian group generated by $\pp\in \Spec R$ with $\dim R/\pp = d$. If 
$M$ is a finitely generated $R$-module of dimension at  most $d$ then $M_{\pp}$ is an $R_{\pp}$-module of finite length, which we denote by $\ell_{R_{\pp}}(M_{\pp})$, for all $\pp$ 
with $\dim R/\pp=d$. We let $z_d(M):=\sum_{\pp} \ell_{R_{\pp}}(M_{\pp}) \pp$, where the sum is taken over all $\pp\in \Spec R$ such that $\dim R/\pp=d$.

\begin{thm}\label{main_intro} Assume that $p\ge 5$. There exists a reduced,  two dimensional, $\OO$-torsion free quotient $R^{\psi}_{\rho}(\mathbf w, \tau)$ (resp. $R^{\psi, \mathrm{cr}}_{\rho}(\mathbf w, \tau)$) of $R_{\rho}$ such that for all $\nn\in \MaxSpec R_{\rho}[1/p]$, $\nn$ lies in $\MaxSpec R^{\psi}_{\rho}(\mathbf w, \tau)[1/p]$ (resp. $R^{\psi, \mathrm{cr}}_{\rho}(\mathbf w, \tau)[1/p]$) if and only if $\rho^{\mathrm{un}}_{\nn}$ is potentially semi-stable (resp. potentially crystalline) of $p$-adic Hodge type $(\mathbf w, \tau, \psi)$.

Moreover, for each irreducible $k$-representation $\sigma$ of $K$ there exists a one dimensional cycle $z(\sigma, \rho)$ of $R_{\rho}$, independent of $\mathbf w$ and $\tau$,
such that for all  $p$-adic Hodge types $(\mathbf w, \tau, \psi)$ we have an equality of one dimensional cycles:
\begin{equation}\label{one}
z_1(R^{\psi}_{\rho}(\mathbf w, \tau)/(\varpi))= \sum_{\sigma} m_{\sigma}(\mathbf w, \tau) z(\sigma, \rho),
\end{equation}
\begin{equation}\label{two}
z_1(R^{\psi, \mathrm{cr}}_{\rho}(\mathbf w, \tau)/(\varpi))= \sum_{\sigma} m^{\mathrm{cr}}_{\sigma}(\mathbf w, \tau) z(\sigma, \rho),
\end{equation}
where the sum is taken over the set of isomorphism classes of smooth irreducible $k$-representations of $K$. 

Further, $z(\sigma, \rho)$ is non-zero if and only if $\Hom_K(\sigma, \beta)\neq 0$, where $\beta$ is a smooth $k$-representation 
of $\GL_2(\Qp)$ associated to $\rho$ by Colmez in \cite{colmez} (the so called \textit{atome automorphe}). In which case 
the Hilbert-Samuel multiplicity of $z(\sigma, \rho)$ is one, with one exception when $\rho\cong \bigl(\begin{smallmatrix} \omega &  \ast \\ 0 & 1\end{smallmatrix}\bigr )\otimes \chi$, where $\omega$ is the mod $p$ cyclotomic character,
and $\ast$ is \textit{peu ramifi\'e} in the sense of \cite{serre_duke}, in which case the multiplicity is $2$.
\end{thm}
 
 The equality of cycles implies the equality of Hilbert-Samuel multiplicities, which proves the conjecture posed in \cite{bm}. The result is new only in the case when $\rho\cong \bigl(\begin{smallmatrix} \omega &  \ast \\ 0 & 1\end{smallmatrix}\bigr )\otimes \chi$. The other cases have been treated by Kisin in \cite{kisinfm}. Moreover, the existence of potentially semi-stable deformation rings has been established by Kisin in \cite{kisin_pst} in much greater generality, however our arguments
 give a new proof of this for $\rho$ as above.  Different parts of Kisin's arguments have been improved upon by Breuil-M\'ezard \cite{bm2} and Emerton-Gee \cite{emertongee}. 
 In particular, the formulation of the problem in terms of the language of cycles is due to Emerton-Gee. In all these papers, there is a local part of the argument using the $p$-adic local Langlands for $\GL_2(\Qp)$, see \cite{colmez}, \cite{colmez1}, \cite{bb} and \cite{berbour} for an overview, and 
 a global part of the argument, which uses Taylor-Wiles-Kisin patching argument. Our proof is purely local, except that we need the description of locally algebraic vectors in the $p$-adic Langlands 
 correspondence for $\GL_2(\Qp)$, and in the supercuspidal case, i.e. when $\tau$ is a restriction to $I_{\Qp}$ of an irreducible representation of the Weil group,  this description uses a global argument of Emerton \cite{emfm}. We use the results of \cite{cmf} instead of the global input. We also require more refined information about locally algebraic vectors than  other proofs, see Theorem \ref{RED_is_ok}. In the case when $\tau$ is the restriction to $I_{\Qp}$ of an irreducible representation of the Weil group, this is proved by  Dospinescu  in \cite{dospinescu}. Although we deal only with $\rho$, which have scalar endomorphisms, Yongquan Hu and Fucheng Tan in \cite{hu} have extended our method to non-scalar split residual representations.
 
The paper is organized as follows. In \S\ref{general} we develop a general formalism, which generates Breuil-M\'ezard conjecture like statements. The formalism is summed up in \S\ref{summary}. 
Since the Breuil-M\'ezard conjecture is expected and sometimes even known to hold in other contexts, see \cite{Kisin_icm}, \cite{Kisin_Gee}, we hope that our formalism  will be useful. 
We then spend the rest of the paper checking the conditions that make the formalism work.  For a fast read, we suggest to look at \S\ref{summary} and then at \S \ref{proof_of_conj}.

\textbf{Acknowledgements}. I thank Toby Gee,  Eike Lau and Thomas Zink  for a number of stimulating discussions.  Large parts of the paper were written, while visiting Toby Gee at Imperial 
College London.  I thank especially Matthew Emerton, who explained to me during the LMS Symposium in Durham in 2011, 
that he had global reasons to believe that the restriction to $K$ of a representation 
denoted by $N$ in \S\ref{non_generic_case} is projective. This was a major source of motivation for me to prove it locally, which resulted in Theorem \ref{B}. I thank Gabriel Dospinescu, Yongquan Hu and  the anonymous referees for their comments on the earlier draft.

\section{Generalities}\label{general}

Let $G$ be a $p$-adic analytic group and let $K$ be a compact open subgroup of $G$. Let $\Mod^{\sm}_G(\OO)$ be the category 
of smooth $G$-representation on $\OO$-torsion modules. Let $\Mod^{\pro}_G(\OO)$ be the category of compact $\OO\br{K}$-modules 
with an action of $\OO[G]$, such that the two actions coincide, when restricted to $\OO[K]$. The superscript $\pro$ stands for 
profinite-augmented. We equip objects of $\Mod^{\sm}_G(\OO)$ with the discrete topology.  Sending $\tau$ to its Pontryagin dual 
$\tau^{\vee}:=\Hom_{\OO}(\tau, L/\OO)$ equipped with the compact-open topology induces an anti-equivalence of categories 
between $\Mod^{\sm}_G(\OO)$ and $\Mod^{\pro}_G(\OO)$, with the inverse functor given by $M\mapsto M^{\vee}:=\Hom^{\cont}_{\OO}(M, L/\OO)$. 

Let $(R, \mm)$ be a complete local commutative noetherian $\OO$-algebra with residue field $k$. Suppose that we are given $N$ in 
$\Mod^{\pro}_G(\OO)$, and an $\OO$-algebra homomorphism from $R$ into the ring of endomorphisms of $N$, such that there exists a basis 
of open neighbourhoods $\{N_i\}_{i\in I}$ of  $0$ in $N$, consisting of $R$-submodules of $N$. If $\md$ is a compact $R$ module, we let 
$\md\wtimes_R N:=\varprojlim\md/\md_j \otimes_{R} N/N_i$, where $\{\md_j\}_{j\in J}$ is a basis of open neighborhoods of $0$ in $\md$ consisting of $R$-submodules. 
For every finitely presented $R$-module $\md$ the completed and the usual tensor products coincide and, since $R$ is noetherian, every finitely generated module is finitely presented.

\begin{lem}\label{dis} If $\lambda\in \Mod^{\sm}_K(\OO)$ is of finite length, then the natural action of $R$ on
$\Hom^{\cont}_{\OO{\br{K}}}(N, \lambda^{\vee})$ is continuous for the discrete topology on the module. 
\end{lem}
\begin{proof} We have the following isomorphisms of $R$-modules: 
$$ \Hom^{\cont}_{\OO\br{K}}(N, \lambda^{\vee})\cong \Hom_K(\lambda, N^{\vee})\cong \Hom_K(\lambda, \varinjlim N^{\vee}[\mm^n]).$$
Since $\lambda$ is a finitely generated $K$-representation, we obtain  
$$\Hom_K(\lambda, \varinjlim N^{\vee}[\mm^n])\cong \varinjlim \Hom_K(\lambda,N^{\vee}[\mm^n]).$$
The action of $R$ on  $\Hom_K(\lambda,N^{\vee}[\mm^n])$ is continuous for the discrete topology on the module, since it factors
through the action of $R/\mm^n$. The isomorphism 
\begin{equation}\label{Nlambda}
\Hom^{\cont}_{\OO\br{K}}(N, \lambda^{\vee})\cong  \varinjlim \Hom_K(\lambda,N^{\vee}[\mm^n])
\end{equation}
implies the claim. 
\end{proof}

\begin{defi}\label{M_sigma} If $\lambda\in \Mod^{\sm}_K(\OO)$ is of finite length, then we let 
$$M(\lambda):=(\Hom_{\OO\br{K}}^{\cont}(N, \lambda^{\vee}))^{\vee}.$$
\end{defi}

\begin{lem}\label{Mfunct} The map $\lambda\mapsto M(\lambda)$ defines a covariant, right exact functor from the category of smooth, finite length 
$K$-representations on $\OO$-torsion modules, to the category of compact $R$-modules. If $N$ is projective in $\Mod^{\pro}_K(\OO)$ 
then the functor is exact. 
\end{lem}
\begin{proof} Since Pontryagin dual of a discrete module is compact, Lemma \ref{dis} implies that $M(\lambda)$ is a compact 
$R$-module for $\lambda\in \Mod^{\sm}_K(\OO)$ of finite length. Since $\lambda\mapsto M(\lambda)$ is defined as a
composition of three functors, the first contravariant and exact, the second covariant and left exact, and the third contravariant 
and exact, it is a covariant, right exact  functor. If $N$ is 
projective in $\Mod^{\pro}_K(\OO)$ then the functor $\Hom^{\cont}_{\OO\br{K}}(N, \ast)$ is exact. Since taking Pontryagin duals 
is an exact functor we deduce that $\lambda\mapsto M(\lambda)$ is an exact functor.
\end{proof}

\begin{prop}\label{disdual} Let $\lambda\in \Mod^{\sm}_K(\OO)$ be of finite length and let $\md$ be a 
finitely generated $R$-module. Then we have a natural isomorphism of compact $R$-modules:
\begin{equation}\label{compact}
\md\wtimes_R (\Hom_{\OO\br{K}}^{\cont}(N, \lambda^{\vee}))^{\vee}\cong (\Hom_{\OO\br{K}}^{\cont}(\md \wtimes_R N, \lambda^{\vee}))^{\vee}.
\end{equation}
\end{prop}
\begin{proof} Let $F$ be any  right exact covariant  $R$-linear functor from the category of compact $R\br{K}$-modules to $R$-modules. 
For each finitely generated $R$-module $\md$ we  have a natural map
$\md \otimes_R F(N)\rightarrow F(\md\otimes_R N),$
constructed as follows. For each $x\in \md$ we  have a map $\phi_x\in \Hom^{\cont}_{R\br{K}}(N, \md\otimes_R N)$ given by $\phi_x(v)= x\otimes v$.
This induces a map $\md \rightarrow  \Hom_R( F(N), F(\md\otimes_R N))$, $x\mapsto F(\phi_x)$. This map is $R$-linear, since $F(\phi_{rx})= F(r \phi_x)= r F(\phi_x)$, for all $r\in R$.
By the universal property of tensor product we obtain a map 
\begin{equation}\label{construct_map}
\md \otimes_R F(N) \rightarrow F(\md\otimes_R N), \quad x \otimes v \mapsto F(\phi_x) (v).
\end{equation}
If $\md=R$ then the above map is an isomorphism. If we assume additionally that $F$ respects direct sums, so that the map 
$F(\pr_1)\oplus F(\pr_2):  F(N_1\oplus N_2) \rightarrow F(N_1)\oplus F(N_2)$ is an isomorphism for all $N_1$ and $N_2$, where $\pr_1$ and $\pr_2$ are the two projections,
then we obtain  that \eqref{construct_map} is an isomorphism for all $\md=R^n$. Finally by  choosing a presentation of $\md$, and using  that the functor is right exact and covariant we  deduce that  \eqref{construct_map} is an isomorphism for all finitely generated modules $\md$.   A special case of this argument with $F(N)= \Hom_{\OO\br{K}}^{\cont}(N, \lambda^{\vee})^{\vee}$ implies the proposition. Note that the functors $\md\otimes_R$ and $\md\wtimes_R$ are naturally 
isomorphic, as  $\md$ is finitely presented over $R$.
\end{proof}

\begin{cor}\label{fg} Let $\lambda\in \Mod^{\sm}_K(\OO)$ be of finite length. If $k\wtimes_R N$ is a finitely generated $\OO\br{K}$-module, 
then $M(\lambda)$ is a finitely generated $R$-module. Moreover, $\dim M(\lambda)\le \max_{\sigma} \{\dim M(\sigma)\}$, where  the 
maximum is taken over all the irreducible subquotients of $\lambda$. 
\end{cor}
\begin{proof} By Nakayama's lemma for compact $R$-modules, \cite[Cor.1.5]{bru}, it is enough to show that $k\wtimes_R M(\lambda)$ is 
a finite dimensional $k$-vector space. Since  $k\wtimes_R N$ is a finitely generated $\OO\br{K}$-module, 
its Pontryagin dual is an admissible $K$-representation. Hence, the subspace of $H$-invariants $((k\wtimes_R N)^{\vee})^H$ 
is finite dimensional for every open subgroup 
$H$ of $K$. Taking $H$ to be the kernel of $\lambda$, we deduce that   $\Hom_K (\lambda, (k\wtimes_R N)^{\vee})$ is finite dimensional. 
Hence, $\Hom_{\OO\br{K}}^{\cont}(k\wtimes_R N, \lambda^{\vee})$ is a finite dimensional $k$-vector space. It follows from 
\eqref{compact} that  $k\wtimes_R M(\lambda)$ is finite dimensional. The last assertion follows from the right exactness of the 
functor $\lambda\mapsto M(\lambda)$ and the behaviour of dimension in short exact sequences. 
\end{proof}

\begin{cor}\label{zero_lambda}  Let $\lambda\in \Mod^{\sm}_K(\OO)$ be of finite length. If 
$\Hom^{\cont}_{\OO\br{K}}(k\wtimes_R N, \lambda^{\vee})=0$, then $M(\lambda)=0$.
\end{cor}
\begin{proof} It follows from Proposition \ref{disdual} that $k\wtimes_R M(\lambda)=0$. Topological Nakayama's lemma 
implies that $M(\lambda)=0$.
\end{proof} 

\begin{cor}\label{fgG} Let $\kappa$ be a smooth, finitely generated $k$-representation of $G$.
 If $k\wtimes_R N$ is a finitely generated $\OO\br{K}$-module, 
then $\Hom_G(\kappa, N^{\vee})^{\vee}$ is a finitely generated $R$-module. 
\end{cor}
\begin{proof} Since $\kappa$ is smooth and is finitely generated as a $G$-representation, there exists a finite 
length $K$-subrepresentation $\lambda$ of $\kappa$, such that $\lambda$ generates $\kappa$ as a $G$-representation. 
In other words, we have a surjection $\cIndu{K}{G}{\lambda}\twoheadrightarrow \kappa$. Since 
\begin{equation}\label{dim_compact_ind}
\Hom_G(\cIndu{K}{G}{\lambda}, N^{\vee})\cong \Hom_K(\lambda, N^{\vee})\cong \Hom^{\cont}_{\OO\br{K}}(N, \lambda^{\vee}),
\end{equation} 
we deduce that $\Hom_G(\kappa, N^{\vee})^{\vee}$ is a quotient of $M(\lambda)$. As $M(\lambda)$ is a finitely generated $R$-module 
by Corollary \ref{fg}, we deduce that $\Hom_K(\kappa, N^{\vee})^{\vee}$ is a finitely generated $R$-module.
\end{proof}

\begin{cor}\label{fgG1} Let $\kappa$ be a smooth, finitely generated $k$-representation of $G$, let $\mathfrak a$ be the $R$-annihilator of 
$\Hom_G(\kappa, N^{\vee})^{\vee}$ and assume that 
$k\wtimes_R N$ is a finitely generated $\OO\br{K}$-module. If $\dim_k \Hom_G(\kappa, (k\wtimes_R N)^{\vee})=1$ then $\Hom_G(\kappa, N^{\vee})^{\vee}\cong R/\mathfrak a$ as an $R$-module. Moreover, the Krull dimension of 
$\Hom_G(\kappa, N^{\vee})^{\vee}$ is at most $\dim_k \Hom_G(\kappa, (N /\mm^2 N)^{\vee})- 1$.
\end{cor}
\begin{proof} The proof of Proposition \ref{disdual} shows that for all finitely generated $R$-modules  $\md$, we have an isomorphism of $R$-modules: 
\begin{equation}\label{same_trick}
\md \otimes_R \Hom_G(\kappa, N^{\vee})^{\vee}\cong \Hom_G(\kappa, (\md\wtimes_R N)^{\vee})^{\vee}.
\end{equation}
 Since $k\wtimes_R N$ is a finitely generated $\OO\br{K}$-module, Corollary \ref{fgG} implies that 
 $M:=\Hom_G(\kappa, N^{\vee})^{\vee}$ is a finitely generated $R$-module. 
Since $\dim_k \Hom_G(\kappa, (k\wtimes_R N)^{\vee})=1$, we deduce from \eqref{same_trick} that $k\otimes_R M$ is a one dimensional $k$-vector space. Nakayama's lemma  implies that $M$ can be generated by one element over $R$, and hence $M\cong R/\mathfrak a$.  It follows from \eqref{same_trick} 
that $\dim_k \Hom_G(\kappa, (N/\mm^2)^{\vee})= \dim_k M/\mm^2 M$.  The isomorphism $M\cong R/\mathfrak a$ implies that 
$\dim_k \mm/(\mathfrak a +\mm^2)= \dim_k \Hom_G(\kappa, (N/\mm^2 N)) -1$. Since the dimension of $R/\mathfrak a$ is bounded by its embedding dimension, 
\cite[\S14]{matsumura}, we obtain the result. 
\end{proof}

If $X$ is an $\OO$-torsion free, $p$-adically complete and separated $\OO$-module, then we let
$X^d:=\Hom_{\OO}(X, \OO)$ equipped with the topology of pointwise convergence, that is the coarsest locally convex topology 
such that for each $x\in X$ the evaluation map at $x$, $X^d\rightarrow \OO$, $\phi\mapsto \phi(x)$ is continuous. We will refer
to $X^d$ as the Schikhof dual of $X$. The following is shown in the proof of \cite[Lem.5.4]{comp}:

\begin{prop}\label{schdual} Let $X$ be an $\OO$-torsion free, $p$-adically complete and separated $\OO$-module. 
Then for each $n \ge 1$ we have a natural topological isomorphism 
$$X^d/\varpi^n X^d\cong (X/\varpi^n X)^{\vee}.$$
\end{prop}
It follows from the proof of \cite[Thm.1.2]{iw} that the evaluation map induces a topological isomorphism 
$X\overset{\cong}{\rightarrow} \Hom_{\OO}^{\cont}(X^d, \OO)$, for the $p$-adic topology on the target. 

Let $V$ be a continuous representation of $K$ on a finite dimensional $L$-vector space, let 
$\Theta$ be a $K$-invariant $\OO$-lattice in $V$ and let $\Theta^d:=\Hom_{\OO}(\Theta, \OO)$ and $V^d:=\Hom_{\OO}(V, L)$. 

\begin{lem}\label{act1n} $\Hom_{\OO\br{K}}^{\cont}(N, \Theta^d)$ is $p$-adically complete, separated and $\OO$-torsion free. 
\end{lem}
\begin{proof}  Since $\Theta^d$ is a 
free $\OO$-module of finite rank, we have 
\begin{equation}\label{limn}
\begin{split}
\Hom_{\OO\br{K}}^{\cont}(N, \Theta^d)&\cong \Hom_{\OO\br{K}}^{\cont}(N, \varprojlim \Theta^d/\varpi^n \Theta^d)\\&\cong 
\varprojlim \Hom_{\OO\br{K}}^{\cont}(N, \Theta^d/\varpi^n \Theta^d).
\end{split}
\end{equation}
 By applying $\Hom^{\cont}_{\OO\br{K}}(N, \ast)$ to the exact sequence 
$0\rightarrow \Theta^d\overset{\varpi^n}{\rightarrow}\Theta^d \rightarrow \Theta^d/\varpi^n \Theta^d\rightarrow 0$
we obtain an injection:
\begin{equation}\label{picomn}
\Hom_{\OO\br{K}}^{\cont}(N, \Theta^d)/\varpi^n \Hom_{\OO\br{K}}^{\cont}(N, \Theta^d)\hookrightarrow 
 \Hom_{\OO\br{K}}^{\cont}(N,\Theta^d/\varpi^n \Theta^d)
\end{equation}
for all $n\ge 1$. It follows from \eqref{limn} that the composition
$$  \Hom_{\OO\br{K}}^{\cont}(N, \Theta^d)\rightarrow \varprojlim \Hom_{\OO\br{K}}^{\cont}(N, \Theta^d)/(\varpi^n)
\rightarrow \varprojlim \Hom_{\OO\br{K}}^{\cont}(N, \Theta^d/\varpi^n \Theta^d)$$
is an isomorphism, and from \eqref{picomn} that the second arrow is an injection, as $\varprojlim$ is left exact,  
thus the first arrow in the above equation is 
an isomorphism. Hence, $\Hom_{\OO\br{K}}^{\cont}(N, \Theta^d)$ is $p$-adically complete and separated. Since $\Theta^d$ is $\OO$-torsion free, 
so is  $\Hom_{\OO\br{K}}^{\cont}(N, \Theta^d)$. 
\end{proof}

We equip $\Hom_{\OO\br{K}}^{\cont}(N, \Theta^d)$ with the $p$-adic topology.
\begin{lem}\label{act2n} The $R$-action on $\Hom_{\OO\br{K}}^{\cont}(N, \Theta^d)$ is continuous for the $p$-adic topology on the module.
\end{lem}
\begin{proof} It is enough to show that the action of $R$ on $\Hom_{\OO\br{K}}^{\cont}(N, \Theta^d)/(\varpi^n)$ is continuous 
for the discrete topology on the module. This follows from \eqref{picomn} and Lemma \ref{dis}.
\end{proof}

\begin{defi}\label{M_Th} Let $M=M(\Theta)=\Hom_{\OO}(\Hom^{\cont}_{\OO\br{K}}(N, \Theta^d), \OO)$ equipped with the topology of pointwise convergence.
\end{defi}

\begin{lem}\label{act3n} $M$ is a compact $R$-module.
\end{lem} 
\begin{proof} Since $R$ is a compact ring, Pontryagin duality induces an anti-equi\-va\-lence of 
categories between the discrete and compact $R$-modules, \cite{gabriel}. This observation combined with Lemma \ref{dis} and 
\eqref{picomn} implies that the Pontryagin dual of $\Hom_{\OO\br{K}}^{\cont}(N, \Theta^d)/(\varpi^n)$ is a compact 
$R$-module.  It follows from Proposition \ref{schdual} that 
we have a topological isomorphism 
$$M \cong \varprojlim (\Hom_{\OO\br{K}}^{\cont}(N, \Theta^d)/(\varpi^n))^{\vee},$$
which implies the claim. 
\end{proof}
Since $N$ is compact, $\Hom^{\cont}_{\OO\br{K}}(N, V^d)$ with the topology defined by the supremum norm is 
an $L$-Banach space (we choose some norm on $V$). Moreover, it contains $\Hom_{\OO\br{K}}^{\cont}(N, \Theta^d)$ as an open bounded
$\OO$-lattice. Since $M$ is compact $\Hom_{\OO}^{\cont}(M, L)$ with a supremum norm is an $L$-Banach space and 
$\Hom^{\cont}_{\OO}(M, \OO)$ is a unit ball inside it. It follows from the proof of \cite[Thm.1.2]{iw} that the evaluation map 
induces a topological  isomorphism:
\begin{equation}\label{top_0} 
 \Hom_{\OO\br{K}}^{\cont}(N, \Theta^d)\cong \Hom_{\OO}^{\cont}(M, \OO).
\end{equation}
This in turn induces a functorial isomorphism of $L$-Banach spaces :
\begin{equation}\label{top_1}
  \Hom_{\OO\br{K}}^{\cont}(N, V^d)\cong \Hom^{\cont}_{\OO}(M, L).
\end{equation}

\begin{lem}\label{reduce_M} If $N$ is projective in $\Mod^{\pro}_K(\OO)$ then $M(\Theta)/(\varpi)\cong M(\Theta/(\varpi))$.
\end{lem}
\begin{proof} The projectivity of $N$ implies that \eqref{picomn} is an isomorphism. The assertion follows from  Proposition \ref{schdual}.
\end{proof}

\begin{prop}\label{fgM} If $k\wtimes_R N$ is a finitely generated $\OO\br{K}$-module then $M$ is a finitely generated $R$-module.
\end{prop}
\begin{proof} Nakayama's lemma for compact $R$-modules implies that $M$ is a finitely generated $R$-module if and only if 
 $k\wtimes_R M$ is a finite dimensional $k$-vector space. In particular, it is enough to show that $M/\varpi M$ is finitely generated. 
It follows from \eqref{picomn} and Proposition \ref{schdual} that $M/\varpi M$ is a quotient of 
$(\Hom_{\OO\br{K}}^{\cont}(N, \Theta^d/(\varpi)))^{\vee}$, which is finitely generated by Corollary \ref{fg}.
\end{proof}

\subsection{Banach space representations}\label{Banachspace}

Let $R$, $N$, $V$, $\Theta$, and $M=M(\Theta)$ be as in the previous subsection. From now on we assume that $k\wtimes_R N$ is a finitely
generated $\OO\br{K}$-module. Nakayama's lemma for compact $R\br{K}$-modules implies that this assumption is equivalent to $N$ being finitely generated over $R\br{K}$. Proposition \ref{fgM} implies that $M$ is a finitely generated $R$-module. 

Since $R$ is a commutative complete  local 
noetherian $\OO$-algebra with residue field $k$ there is a surjection $\varphi: \OO\br{x_1, \ldots, x_s}\twoheadrightarrow R$. We identify  
$\OO\br{x_1, \ldots, x_s}$ with the completed group algebra of $\mathcal G:=\Zp^{s}$, and let $\mathcal G$ act on $N$ via $\varphi$. 

\begin{lem}\label{admissible_G}  $\Hom_{\OO\br{K}}^{\cont}(N, \Theta^d)_L$ is an admissible unitary $L$-Banach space representation of $\mathcal G$.
\end{lem}
\begin{proof} Since $M$ is finitely generated over $R$, it is also finitely generated over $\OO\br{\mathcal G}$, and the assertion follows from 
\eqref{top_0}.
\end{proof}

\begin{lem}\label{r_e_co} The map $N \mapsto \Hom_{\OO}^{\cont}( \Hom_{\OO\br{K}}^{\cont}(N, \Theta^d), L)$ induces a covariant, right exact functor
from the category of finitely generated $R\br{K}$-modules to the category of finitely generated $R[1/p]$-modules.
\end{lem}
\begin{proof} Since $R$ is noetherian and $K$ is a $p$-adic analytic group, the ring $R\br{K}$ is noetherian. Thus 
the forgetful functor induces an equivalence of categories between the category of finitely generated compact $R\br{K}$-modules
with continuous $R\br{K}$-linear homomorphism as morphisms, and the category of finitely generated $R\br{K}$-modules with 
$R\br{K}$-linear homomorphism as morphisms, see for example \cite[(5.2.23)]{nsw}. Since $R$ is an $\OO$-algebra, we have a 
natural map $\iota: \OO\br{K}\rightarrow R\br{K}$. Since $R$ and $\OO$ have the same residue field, any 
$R\br{K}$-module of finite length will also be an $\OO\br{K}$-module of finite length. Hence $\iota$ is a continuous ring homomorphism, 
which induces an exact functor from the category of compact $R\br{K}$-modules to the category of $\OO\br{K}$-modules. Hence, if we start with 
an exact sequence of finitely generated $R\br{K}$-modules:
\begin{equation}\label{fg_rk}
N_1\rightarrow N_2\rightarrow N_3\rightarrow 0,
\end{equation}
then we may view it as an exact sequence in the category of compact $\OO\br{K}$-modules, which is abelian by \cite[Th. IV.3.3]{gabriel}. Since $\Theta^d$ is 
also an object in this category, by applying $\Hom_{\OO\br{K}}^{\cont}(\ast, \Theta^d)$ to \eqref{fg_rk} we obtain an exact sequence of $\OO$-modules:
\begin{equation}\label{fg_rk2}
0\rightarrow \Hom_{\OO\br{K}}^{\cont}(N_3, \Theta^d)\rightarrow \Hom_{\OO\br{K}}^{\cont}(N_2, \Theta^d)\rightarrow \Hom_{\OO\br{K}}^{\cont}(N_1, \Theta^d).
\end{equation}
If we equip the $\OO$-modules in \eqref{fg_rk2} with the $p$-adic topology using Lemma \ref{act1n}, then the maps between them become continuous. If we further invert $p$ 
then Lemma \ref{admissible_G} implies that  we get an exact sequence of admissible unitary $L$-Banach space representations of  $\mathcal G$. The functor 
$\Hom_{\OO}^{\cont}(\ast, L)$ induces an anti-equivalence of categories between the category of admissible unitary $L$-Banach space representations of $\mathcal G$ 
and the category of finitely generated $\OO\br{\mathcal G}[1/p]$-modules by \cite[Thm.3.5]{iw}. Since the action of $\OO\br{\mathcal G}$ on $N_1$, $N_2$ and $N_2$ factors through $R$ 
by assumption, we deduce that by applying $\Hom_{\OO}^{\cont}(\ast, L)$ to \eqref{fg_rk2} we obtain an exact sequence of finitely generated $R[1/p]$-modules.
\end{proof}

\begin{remar} It follows from \cite[Prop.3.1]{nfa} that 
$$\Hom_{\OO}^{\cont}( \Hom_{\OO\br{K}}^{\cont}(N, \Theta^d), L)\cong \Hom_{\OO}( \Hom_{\OO\br{K}}^{\cont}(N, \Theta^d), \OO)_L= M(\Theta)_L$$
\end{remar}

\begin{prop}\label{good_functor} If $\md^0$ is a finitely generated $R$-module then there is a natural isomorphism of finitely generated $R[1/p]$-modules:
$$ \md^0\otimes_R M(\Theta)_L \cong \Hom_{\OO}^{\cont}( \Hom_{\OO\br{K}}^{\cont}(\md^0\wtimes_R N, \Theta^d), L).$$
\end{prop}
\begin{proof} Since $\md^0$ is a finitely generated $R$-module, it is a quotient of $R^{\oplus n}$ for some $n$. Then $\md^0\wtimes_R N$ is 
a quotient of $N^{\oplus n}$, and hence is a finitely generated $R\br{K}$-module module. Since $R$ is noetherian, $\md^0$ is finitely presented and 
thus $\md^0\otimes_R N \cong \md^0\wtimes_R N$. The Proposition asserts the existence of the natural isomorphism  between $\md^0\otimes_R F(N)$ 
and $F(\md^0\otimes_R N)$, where $F$ is the functor $N\mapsto  \Hom_{\OO}^{\cont}( \Hom_{\OO\br{K}}^{\cont}(N, \Theta^d), L)$. The functor $F$
is right exact and covariant by Lemma \ref{r_e_co} and respects finite direct sums. The same proof as in Proposition \ref{disdual} gives us the assertion. 
\end{proof}

Let $\md$ be a finitely generated $R[1/p]$-module. We choose a finitely generated $R$-submodule $\md^0$ of $\md$, such that 
$\md\cong \md^0_L$. Then $\md^0\wtimes_R N$ is a finitely generated $R\br{K}$-module. The surjection $\varphi:\OO\br{\mathcal G}\twoheadrightarrow 
R$ induces the surjection $\OO\br{\mathcal G\times K}\cong \OO\br{\mathcal G}\wtimes_{\OO} \OO\br {K}\twoheadrightarrow R\br{K}$. This allows to consider
$\md^0\wtimes_R N$ as a finitely generated $\OO\br{\mathcal G\times K}$-module. Since the ring $\OO\br{\mathcal G\times K}$ is noetherian, the submodule 
of  $\md^0\wtimes_R N$ consisting of elements killed by a power of $\varpi$ is finitely generated and hence closed. Thus the maximal $\OO$-torsion free quotient of 
$\md^0\wtimes_{R} N$ is Hausdorff and 
\begin{equation}\label{def_Pi_md}
\Pi(\md):=\Hom^{\cont}_{\OO}(\md^0\wtimes_{R} N, L)
\end{equation} 
equipped with the supremum norm is an admissible unitary $L$-Banach space representation of $\mathcal G\times G$, see \cite{iw}. 
The definition of $\Pi(\md)$ does not depend on the choice of $\md^0$, since any two are commensurable.

\begin{lem}\label{m_Pi1}  The map $\md\mapsto \Pi(\md)$ defines a  left exact, contravariant functor from the category 
of finitely generated $R[1/p]$-modules to $\Ban^{\adm}_{\mathcal G\times G}(L)$, the category of admissible $L$-Banach
space representations of $\mathcal G\times G$.
\end{lem}
\begin{proof} If $0\rightarrow \md_1\rightarrow \md_2\rightarrow \md_3\rightarrow 0$ is an exact sequence of finitely generated $R[1/p]$-modules
then let $\md_2^0$ be any finitely generated $R$-submodule of $\md_2$ such that $\md^0_2\otimes_{\OO} L\cong \md_2$. Let 
$\md_3^0$ be the image of $\md_2^0$ in $\md_3$ and let $\md_1^0= \md_2^0\cap \md_1$. Then we obtain an exact sequence 
of finitely generated $R$-modules $0\rightarrow \md_1^0\rightarrow \md_2^0\rightarrow \md_3^0\rightarrow 0$, which recovers the
sequence we started with, once we invert $p$. We apply the right exact functor $\wtimes_R N$ and the left exact functor $\Hom^{\cont}_{\OO}(\ast, L)$ 
to obtain  an exact sequence $0\rightarrow \Pi(\md_3)\rightarrow \Pi(\md_2)\rightarrow \Pi(\md_1)$ in $\Ban^{\adm}_{\mathcal G\times G}(L)$.
\end{proof}

\begin{lem}\label{m_Pi} If $\md$ is  an $R[1/p]$-module of finite length then $\Pi(\md)$ is an admissible $L$-Banach space representation of $G$. In particular, 
$\md\mapsto \Pi(\md)$ defines a  left exact, contravariant functor from the category 
of $R[1/p]$-modules of finite length  to $\Ban^{\adm}_{G}(L)$, the category of admissible $L$-Banach
space representations of $G$.
\end{lem}
\begin{proof} Since $R$ is a complete noetherian $\OO$-algebra with residue field $k$, the residue fields of maximal ideals of $R[1/p]$
are finite extensions of $L$.  Thus $\md$ is a finite dimensional $L$-vector space and $\md^0$ is a free $\OO$-module of finite rank.
Then $\md^0/(\varpi)$ is a finite dimensional $k$-vector space, and in particular an $R$-module of finite length. Since $k\wtimes_R N$ is a finitely generated $\OO\br{K}$-module, 
we deduce that $\md^0/(\varpi) \wtimes_R N$ is a finitely generated $\OO\br{K}$-module. Nakayama's lemma implies that 
$\md^0\wtimes_R N$ is a finitely generated $\OO\br{K}$-module. Hence, $\Pi(\md)$ is an admissible unitary $L$-Banach space representation of $G$. The last part follows from 
Lemma \ref{m_Pi1}.
\end{proof}

\begin{prop}\label{dim_eq} Let $\Theta$ be a $K$-invariant lattice in a continuous representation of $K$ on a finite dimensional
$L$-vector space $V$ and let $M=M(\Theta)$ be the module defined in Definition \ref{M_Th}. If $\md$ is a finitely generated $R[1/p]$-module then there is 
a natural isomorphism of finitely generated $R[1/p]$-modules: 
\begin{equation}\label{first_iso}
\md \otimes_R M(\Theta) \cong \Hom^{\cont}_L( \Hom_K(V, \Pi(\md)), L).
\end{equation}
In particular, if $\md$ is an $R[1/p]$-module of finite length then 
 \begin{equation}\label{second_iso}
 \dim_L \Hom_K( V, \Pi(\md))= \dim_L \md \otimes_{R} M.
 \end{equation}
\end{prop}
\begin{proof} Let $\md^0$ be a finitely generated $R$-submodule of $\md$, such that $\md^0_L\cong \md$. Then
\begin{equation}\label{dim_eq1}
 \md\otimes_R M(\Theta)\cong \md^0\otimes_R M(\Theta)_L\cong \Hom_{\OO}^{\cont}( \Hom_{\OO\br{K}}^{\cont}(\md^0\wtimes_R N, \Theta^d), L),
\end{equation}
where the last isomorphism follows from Proposition \ref{good_functor}. Let $(\md^0\wtimes_R N)_{\mathrm{tf}}$ be the maximal $\OO$-torsion free quotient 
of $\md^0\wtimes_R N$. Since $\md^0\wtimes_R N$ is a finitely generated $R\br{K}$-module, $(\md^0\wtimes_R N)_{\mathrm{tf}}$ is a compact linear-topological
$\OO$-module. Then 
$$\Hom^{\cont}_{\OO\br{K}}(\md^0\wtimes_R N, \Theta^d)_L\cong  \Hom^{\cont}_{\OO\br{K}}((\md^0\wtimes_R N)_{\mathrm{tf}}, \Theta^d)_L\cong 
\Hom^{\cont}_K(V, \Pi(\md))$$
where the first map is an isomorphism as $\Theta^d$ is $\OO$-torsion free, and the second map is induced by Schikhof duality and is isomorphism by \cite[Thm.2.3]{iw}.
Since $V$ is finite dimensional any $L$-linear map from $V$ to $\Pi(\md)$ is continuous, and we obtain a natural isomorphism  
\begin{equation}\label{dim_eq2}
\Hom^{\cont}_{\OO\br{K}}(\md^0\wtimes_R N, \Theta^d)_L\cong \Hom_K(V, \Pi(\md)).
\end{equation}
We use \eqref{dim_eq2}  to put the Banach space topology on $\Hom_K(V, \Pi(\md))$. Combining \eqref{dim_eq1} and \eqref{dim_eq2} we obtain the first assertion of the 
proposition. 

If $\md$ is an $R[1/p]$-module of finite length, then $\md$ is a finite dimensional $L$-vector space. Since $M$ is a finitely 
generated $R$-module,  $\md\otimes_R M$ is a finite dimensional $L$-vector space. We obtain \eqref{second_iso} by comparing the dimensions 
of $L$-vector spaces in \eqref{first_iso}. 
\end{proof}
 
\subsection{Cycles}

We keep the notation of the previous subsection, and we continue to assume that $k\wtimes_R N$ is a finitely generated 
$\OO\br{K}$-module. 

\begin{lem}\label{cycindep} Let $\Theta_1$ and $\Theta_2$ be open $K$-invariant lattices in $V$, and let $M_1=M(\Theta_1)$, $M_2=M(\Theta_2)$. Then 
$\ann_R M_1=\ann_R M_2$. In particular, $M_1$ and $M_2$ have the same dimension $d$ (say). Moreover, 
we have an equality of $d$-dimensional cycles $z_d(M_1)=z_d(M_2)$.
\end{lem}
\begin{proof} It is enough to consider the case $\varpi^n \Theta_2\subset \Theta_1\subset \Theta_2$ for some integer $n\ge 1$, as
all the lattices in $V$ are commensurable, since $V$ is a finite dimensional $L$-vector space.
Let $C$ be the cokernel of $\Hom^{\cont}_{\OO\br{K}}(N, \Theta_1)\hookrightarrow \Hom^{\cont}_{\OO\br{K}}(N, \Theta_2)$. Then $C$ is killed 
by $\varpi^n$, and 
so $\Hom_{\OO}(C, \OO)=0$ and $\Ext^1_{\OO}(C, \OO)$ is killed by $\varpi^n$. 
By applying $\Hom_{\OO}(\ast, \OO)$ we deduce that $M_2$ is an $R$-submodule of $M_1$ and the quotient is killed by $\varpi^n$, so that 
$\varpi^n M_1\subset M_2 \subset M_1$. Since $M_1$ is $\OO$-torsion free, we deduce that $\ann_R M_1=\ann_R M_2$. Hence, the modules have the 
same dimension $d$. Moreover, since $\varpi^n$ is regular on $M_1$, the dimension of $M_1/\varpi^n M_1$ is less than the dimension of $M_1$, and 
hence the dimension of $M_1/M_2$ is less than dimension of $M_1$. This implies that $z_d(M_1/M_2)=0$, and since cycles are additive in short exact sequences, we get $z_d(M_1)=z_d(M_2)$. 
\end{proof}

\begin{prop}\label{cycles_modp}  Let $d$ be the dimension of $M(\Theta)$. If $N$ is projective in $\Mod^{\pro}_K(\OO)$ then 
we have an equality of $(d-1)$-dimensional cycles
$$ z_{d-1}(M(\Theta)/(\varpi))= \sum_{\sigma} m_{\sigma} z_{d-1}(M(\sigma)),$$
where the sum is taken over the set of isomorphism classes of smooth irreducible $k$-representations of $K$, 
$m_{\sigma}$ is the multiplicity with which $\sigma$ occurs as a subquotient of $\Theta/(\varpi)$ and $M(\sigma)$ 
is an $R$-module defined in Definition \ref{M_sigma}.  
\end{prop}
\begin{remar} Recall that Corollary \ref{zero_lambda} implies that if $\Hom_K(\sigma, (k\wtimes_R N)^{\vee})=0$, then $M(\sigma)=0$.
\end{remar}
\begin{proof} Since $M(\Theta)$ is $\OO$-torsion free, $\varpi$ is $M(\Theta)$-regular, and so $\dim M(\Theta)/(\varpi)=d-1$.
Since $N$ is projective, $M(\Theta)/(\varpi)\cong M(\Theta/(\varpi))$ by Corollary \ref{reduce_M}. The assertion follows from 
the exactness of the functor $\lambda\mapsto M(\lambda)$, established in Lemma \ref{Mfunct}, and the fact that cycles are additive 
in short exact  sequences. 
\end{proof}

\subsection{Generic freeness}

\begin{lem}\label{uppersemi} Let $A$ be a noetherian ring, $\Sigma$ a Zariski dense subset of $X=\Spec A$ and 
let $\mathrm  M$ be a finitely generated $A$-module. Suppose that $\dim_{\kappa(\pp)} \mathrm M\otimes_{A} \kappa(\pp)=d$ for all 
$\pp\in \Sigma$, then $\dim_{\kappa(\pp)} \mathrm M\otimes_{A} \kappa(\pp)=d$ for all minimal primes of $A$.
\end{lem}
\begin{proof} The function $\varphi: X\rightarrow \ZZ$, $\pp\mapsto  \dim_{\kappa(\pp)}\mathrm M\otimes_{A} \kappa(\pp)$ is 
upper semicontinuous by \cite[Ex.12.7.2]{Hart}, so that $\{x\in X: \varphi(x)\ge n\}$ is a closed subset of $X$ for all 
$n\in \ZZ$. Since $\Sigma$ is dense in $X$, we deduce that $\varphi(\pp)\ge d$ 
for all $\pp\in X$. Let $\pp$ be a minimal prime of $R$. Since $\Sigma$ is dense there is some 
$\qq\in \Sigma$ containing 
$\pp$. Now $\qq$ lies in the closure of $\{\pp\}$ and so again the upper semicontinuity of $\varphi$ implies that 
$d=\varphi(\qq)\ge \varphi(\pp)$.
\end{proof} 

\begin{lem}\label{generic} Let $A$ be a noetherian ring of dimension $d$ and let $\mathrm M$ be a faithful, finitely generated $A$-module.
Let $\Sigma$ be a dense subset of $\Spec A$. If 
$$\dim_{\kappa(\mathfrak p)} \kappa(\mathfrak p)\otimes_{A} \mathrm M=1, \quad \forall \mathfrak p\in \Sigma,$$
 then $z_d(A)=z_d(\mathrm M)$.
\end{lem}  
\begin{proof} Let $\qq$ be a minimal prime of $A$. Lemma \ref{uppersemi} implies that $\kappa(\qq)\otimes_{A} \mathrm M$ is a one dimensional
$\kappa(\qq)$-vector space. Nakayama's lemma implies that there is a surjection $A_{\qq}\twoheadrightarrow \mathrm M_{\qq}$. 
Since $\mathrm M$ is a faithful $A$-module, $\mathrm M_{\qq}$ is a faithful $A_{\qq}$-module. This implies that the surjection is an isomorphism. In particular, $\ell_{A_{\qq}}(A_{\qq})= \ell_{A_{\qq}}(\mathrm M_{\qq})$, which implies the equality of cycles. 
\end{proof}

Now let $R$, $M$ and $V$ be as in the previous subsection, and let $d$ be the dimension of $M$.  We will denote by $\MaxSpec R[1/p]$ the set of maximal ideals of $R[1/p]$.

\begin{prop}\label{get_to_R} If there exists a dense subset $\Sigma$ of $\supp M$, which is contained in $\MaxSpec R[1/p]$, such that 
$\dim_{\kappa(\nn)} \Hom_K(V, \Pi(\kappa(\nn)))=1$ for all $\nn\in \Sigma$ then we have an equality of $d$-dimensional cycles
$z_d(R/\ann_R M)=z_d(M)$.
\end{prop}
\begin{proof} This follows from Lemma \ref{generic} and Proposition \ref{dim_eq}. 
\end{proof}   

\subsection{Reduceness} Let $R$ and $M$ as in the previous section. We will devise a criterion for the ring $R/\ann_R M$ to be reduced.

\begin{prop}\label{HS3} If  a local noetherian ring $\mR$ affords a faithful, finitely generated, Cohen-Macaulay module $\mM$ then $\mR$ is equidimensional 
of dimension equal to $\dim \mM$. Moreover, every associated prime ideal of $\mR$ is minimal.
\end{prop} 
\begin{proof} Since $\mM$ is Cohen-Macaulay,  $\dim \mR/\mathfrak p=\dim \mM$ for every $\mathfrak p\in \Ass \mM$,  \cite[Thm.17.3(i)]{matsumura}.  
Let $m_1, \ldots, m_s$ be a set of generators of $\mM$. Since $\mR$ acts faithfully, the map $\mR\rightarrow \mM^{\oplus s}$, 
$a\mapsto (a m_1, \ldots, a m_s)$ 
is injective, and hence $\Ass \mR\subseteq \Ass \mM^{\oplus s}=\Ass \mM$, where the last equality follows from \cite[Thm.6.3]{matsumura}. We deduce that 
$\dim \mR/\mathfrak p=\dim \mM$ for all associated prime ideals  $\mathfrak p$ of $\mR$. Since the minimal 
elements  of $\Spec \mR$ and $\Ass \mR$ coincide by \cite[Thm.6.5(iii)]{matsumura}, we deduce that every associated prime $\mathfrak p$ of $\mR$ is minimal
and $\dim \mR/\mathfrak p=\dim \mR=\dim \mM$. 
\end{proof}

\begin{prop}\label{HS4} Let $\mR$ be a complete local noetherian $\OO$-algebra with residue field $k$, which is equidimensional,
and every associated prime ideal of $\mR$ is minimal. Let $\mM$ be  a finitely generated,
faithful $\mR$-module.
If there exists a subset $\Sigma$ of $\MaxSpec \mR[1/p]$, such that $\Sigma$  is dense in $\Spec \mR$ and the following hold
for all $\nn\in \Sigma$:
\begin{itemize}
\item[(i)] $\dim_{\kappa(\nn)} \kappa(\nn)\otimes_{\mR} \mM =1$;
\item[(ii)] $\dim_{\kappa(\nn)}(\mM\otimes_\mR \mR_{\nn}/\nn^2) \le \dim \mR$; 
\end{itemize}
then $\mR_{\nn}$ is regular for all $\nn\in \Sigma$ and $\mR$ is reduced.
\end{prop}
\begin{remar} The $\nn$-adic completion  $\widehat{A}_{\nn}$ of $A$ is an $L$-algebra with residue field $\kappa(\nn)$. 
Since $\kappa(\nn)$ is a finite separable extension of $L$, \cite[Thm. 7.8]{Eis} implies that there is a unique homomorphism of $L$-algebras, 
$\kappa(\nn)\rightarrow \widehat{A}_{\nn}$, such that the composition $\kappa(\nn)\rightarrow \widehat{A}_{\nn}\rightarrow \kappa(\nn)$ is the identity map. 
Thus $\mR_{\nn}/\nn^2\cong \widehat{A}_{\nn}/ \nn^2 \widehat{A}_{\nn}$ is naturally a $\kappa(\nn)$-algebra. 
\end{remar}
\begin{proof} Let $d$ be the dimension of $\mR$ and let $\nn$ be a prime in $\Sigma$. Since $\dim_{\kappa(\nn)} \kappa(\nn)\otimes_{\mR} \mM =1$ 
and $\mM$ is a faithful $\mR$-module, it follows from the proof of Lemma \ref{generic} that $\mM_{\nn}$ is a free $\mR_{\nn}$-module of rank $1$.
Since $\dim_{\kappa(\nn)}(\mM\otimes_\mR \mR_{\nn}/\nn^2) \le d$, we obtain that $\dim_{\kappa(\nn)} \nn/\nn^2\le d-1$.
Since $\mR$ is a complete local noetherian $\OO$-algebra with residue field $k$, $\kappa(\nn)$ is a finite extension of $L$. 
Thus the image of $\mR$ in $\kappa(\nn)$ is an $\OO$-order in $\kappa(\nn)$, and thus has Krull dimensional one. 
 
 Complete local noetherian equidimensional  rings are catenary, \cite[Thm.31.4]{matsumura} implies that $\dim \mR_{\nn}=d-1$. 
Hence the dimension of $\mR_{\nn}$ is equal 
to its embedding dimension, and thus $\mR_{\nn}$ is a regular ring, and so $\mR_{\nn}$ is reduced. 

Since every associated prime is minimal and $\Sigma$ is dense in $\Spec \mR$, above every associated prime of $\mR$ we may find $\nn\in \Sigma$. 
Since localization is transitive, and preserves reduceness, the claim implies that $\mR_{\pp}$ is reduced 
for every associated prime ideal of $\mR$. The natural map $\varphi:\mR\rightarrow \prod_{\pp} \mR_{\pp}$, 
where the product is taken over all the associated primes of $\mR$, is injective, since the kernel is not supported on 
any minimal prime ideal of $\mR$, and thus $\Ass_\mR(\Ker \varphi)$ is empty. Hence, the claim implies that we may embed
$\mR$ into a product of reduced rings, which implies that $\mR$ is reduced. 
\end{proof}

Now let $R$, $M$ and $V$ be as in the previous subsection, and let $d$ be the dimension of $M$.

\begin{prop}\label{RHS4} Suppose that $R/\ann_R M$ is equidimensional and all the associated primes are minimal.
If there exists a dense subset $\Sigma$ of $\supp M$, contained in $\MaxSpec R[1/p]$, such that for all 
$\nn\in \Sigma$ the following hold:
\begin{itemize}
\item[(i)] $\dim_{\kappa(\nn)} \Hom_K(V, \Pi(\kappa(\nn)))=1$;
\item[(ii)] $\dim_{\kappa(\nn)} \Hom_K(V, \Pi(R_{\nn}/\nn^2))\le d$;
\end{itemize}
then $R/\ann_R M$ is reduced.
\end{prop}
\begin{proof} This follows from Lemma \ref{generic} and Proposition \ref{HS4}. 
\end{proof} 

\begin{lem}\label{MCM} Let $\underline{x}=(x_1, \ldots, x_r)$  be an $N$-regular sequence in $R$. If $N/\underline{x}N$
is a finitely generated projective $\OO\br{K}$-module  then $M(\Theta)$ is a Cohen-Macaulay module of dimension
$r+1$.  If $\lambda\in \Mod^{\sm}_K(\OO)$ is of finite length and $M(\lambda)\neq 0$
then $M(\lambda)$ is a Cohen-Macaulay module of dimension $r$.
\end{lem}
\begin{proof} The proof is by induction on $r$. If $N$ is  a finitely generated projective $\OO\br{K}$-module, 
(so that $r=0$), then $\Hom_{\OO\br{K}}^{\cont}(N, \Theta^d)$ is a free $\OO$-module of finite rank by \cite[Prop.4.8]{comp}.
The induction step is given by the following. Suppose that $x\in R$ is $N$-regular, and $N/x N$ is a projective 
$\OO\br{K}$-module. Then $N$ is a projective $\OO\br{K}$-module by Corollary \ref{proj}, Remark \ref{pk} proved below. Moreover, 
the surjection $N\twoheadrightarrow N/xN$ has an $\OO\br{K}$-equivariant splitting. Thus the sequence 
$$0\rightarrow \Hom^{\cont}_{\OO\br{K}}(N/xN, \Theta^d)\rightarrow \Hom^{\cont}_{\OO\br{K}}(N, \Theta^d)\overset{x}{\rightarrow}
\Hom^{\cont}_{\OO\br{K}}(N, \Theta^d)\rightarrow 0$$
is exact and the second arrow has a continuous $\OO$-equivariant section. Hence, when we apply $\Hom_{\OO}(\ast, \OO)$, 
we obtain an exact sequence 
$$0\rightarrow M\overset{x}{\rightarrow} M \rightarrow \Hom_{\OO}( \Hom_{\OO\br{K}}^{\cont}(N/xN, \Theta^d), \OO)\rightarrow 0.$$
The proof for $M(\lambda)$ is the same. 
\end{proof}

\subsection{Projectivity} We will devise a criterion to check, when a finitely generated $R\br{K}$-module $N$ is projective 
in $\Mod^{\pro}_K(\OO)$. Let $P$ be a  pro-$p$ Sylow subgroup of $K$. Then a compact $\OO\br{K}$-module is projective 
if and only if it is projective as an $\OO\br{P}$-module. Since 
$P$ is a pro-$p$ group,  $\OO\br{P}$ is a local ring.  Nakayama's lemma for compact modules implies  that a compact $\OO\br{P}$-module 
$N$ is zero if and only if $N\wtimes_{\OO\br{P}} k=0$. We let $\wTor_{\OO\br{P}}^i(\ast, k)$ be the $i$-th left derived functor of 
$\wtimes_{\OO\br{P}} k$ in $\Mod^{\pro}_P(\OO)$. Since every object in $\Mod^{\pro}_P(\OO)$ has a projective envelope, a
standard argument gives that a compact $\OO\br{P}$-module $N$ is projective if and only if $\wTor^1_{\OO\br{P}}(N, k)=0$. 
We will assume that $R$ is $\OO$-torsion free. 

\begin{prop}\label{torfg} If $N$ is a finitely generated $R\br{P}$-module, then $\wTor_{\OO\br{P}}^i(N, k)$ are finitely generated $R$-modules 
for all $i\ge 0$. 
\end{prop}
\begin{proof} Since $R$ is a compact and torsion free $\OO$-module there exists an isomorphism of compact $\OO$-modules:
$R\cong \prod_{i\in I} \OO$, for some set $I$. Since the completed tensor product commutes with projective limits we have:
$R\br{P}\cong R\wtimes_{\OO} \OO\br{P}\cong \prod_{i\in I} \OO\br{P}$. Hence, $R\br{P}$ is a topologically free $\OO\br{P}$ module
and hence is projective in $\Mod^{\pro}_P(\OO)$. Moreover, 
$R\br{P}\wtimes_{\OO\br{P}} k\cong R\wtimes_{\OO}(\OO\br{P}\wtimes_{\OO\br{P}} k)\cong R\wtimes_{\OO} k$ is a finitely generated 
$R$-module.  Since $N$ is a finitely generated $R\br{P}$-module and $R\br{P}$ is noetherian, there is a projective resolution 
of $M$ by free $R\br{P}$-modules of finite rank. Applying $\wtimes_{\OO\br{P}} k$ to the resolution we obtain a complex 
of finitely generated $R$-modules, the homology of which computes  $\wTor_{\OO\br{P}}^i(N, k)$. 
\end{proof}

\begin{cor}\label{proj} Let $N$ be a finitely generated $R\br{P}$-module, and let $x\in R$ be $N$-regular. If $N/xN$ is projective
in $\Mod^{\pro}_P(\OO)$ then $N$ is also projective in $\Mod^{\pro}_P(\OO)$.
\end{cor}
\begin{proof} We apply $\wtimes_{\OO\br{P}} k$ to the exact sequence 
$0\rightarrow N\overset{x}{\rightarrow} N\rightarrow N/xN\rightarrow 0$. Since $N/xN$ is projective by assumption, 
$\wTor^i_{\OO\br{P}}(N/xN, k)=0$ for $i\ge 1$. Hence, the multiplication by $x$, 
$\wTor^{i}_{\OO\br{P}}(N, k)\overset{x}\rightarrow  \wTor^{i}_{\OO\br{P}}(N, k)$, induces an isomorphism for all $i\ge 1$. Since these are finitely generated 
$R$-modules by Proposition \ref{torfg}, Nakayama's lemma forces them to vanish. Hence, $N$ is a projective $\OO\br{P}$-module.  
\end{proof}

\begin{remar}\label{pk} Since $P$ is a pro-$p$ Sylow of $K$ and is open in $K$, $N$ is projective in $\Mod^{\pro}_P(\OO)$ if and only if it is projective in
$\Mod^{\pro}_K(\OO)$.
\end{remar}

\begin{cor}\label{pdim} Let $N$ be an $R\br{P}$-module, which is finitely generated over $\OO\br{P}$. Let $x\in R$ be  $N$-regular.
If $N/xN$ has a finite projective dimension as an $\OO\br{P}$-module, then so does $N$ and  $1+\pdim_{\OO\br{P}} N = \pdim_{\OO\br{P}} N/xN$.  
\end{cor}
\begin{proof} Since Nakayama's lemma is available to us, standard arguments, see for example Lemma 1 in \cite[\S 19]{matsumura},  
show that the projective dimension of $N$ is equal to the length of minimal projective resolution of $N$, which is equal 
to $\sup \{i : \wTor^i_{\OO\br{P}}(N, k)\neq 0\}$. The assumption that $N$ is a finitely generated $\OO\br{P}$-module implies 
that $\wTor^i_{\OO\br{P}}(N, k)$ are finite dimensional $k$-vector spaces for all $i\ge 0$. If $\wTor^i_{\OO\br{P}}(N/xN, k)=0$
then the multiplication map by $x$, $\wTor^{i-1}_{\OO\br{P}}(N, k)\overset{x}{\rightarrow} \wTor^{i-1}_{\OO\br{P}}(N, k)$
is an injection and hence an isomorphism, as both the source and the target have the same dimension 
as $k$-vector spaces. Proposition \ref{torfg} and Nakayama's lemma imply that $\wTor^{i-1}_{\OO\br{P}}(N, k)=0$. 
The same argument implies that if $\wTor^i_{\OO\br{P}}(N, k)\neq 0$ then the multiplication map by $x$ cannot be injective 
and hence $\wTor^{i+1}_{\OO\br{P}}(N/xN, k)$ is non-zero. Hence, $1+\pdim_{\OO\br{P}} N = \pdim_{\OO\br{P}} N/xN$.
 \end{proof}

\begin{lem}\label{fgqq} Let $x_1, \ldots, x_r\in R$ be a system of parameters for $M(\Indu{P}{K}{\Eins})$, and let 
$\qq$ be the ideal of $R$ generated by it. Then $N/\qq N$ is a finitely generated $\OO\br{K}$-module. 
\end{lem} 
\begin{proof} Let $\lambda= \Indu{P}{K}{\Eins}$, and let $M(\lambda)$ be the module defined in Definition \ref{M_sigma}.
Proposition \ref{disdual} implies that $M(\lambda)/\qq M(\lambda)\cong (\Hom^{\cont}_{\OO\br{K}}(N/\qq N, \lambda^{\vee}))^{\vee}$. 
Moreover, we have 
\begin{displaymath}
\begin{split}
\Hom^{\cont}_{\OO\br{K}}(N/\qq N, \lambda^{\vee})&\cong \Hom_K(\Indu{P}{K}{\Eins}, (N/\qq N)^{\vee})
\cong\Hom_P(\Eins, (N/\qq N)^{\vee})\\ &\cong \Hom_{\OO\br{P}}^{\cont}(N/\qq N, k)\cong ((N/\qq N)\wtimes_{\OO\br{P}} k)^{\vee}.
\end{split}
\end{displaymath} 
Hence, $M(\lambda)/\qq M(\lambda)\cong (N/\qq N)\wtimes_{\OO\br{P}} k$ as $R$-modules. The left hand side of this isomorphism 
is an $R$-module of finite length as $\qq$ is generated by a system of parameters of $M$. Since the residue field of $R$ is $k$, 
we deduce that the right hand side is an $\OO$-module of finite length, and hence $N/\qq N$ is finitely generated over $\OO\br{P}$ by Nakayama's lemma.
\end{proof}

\begin{lem}\label{dimmax} $\dim M(\Indu{P}{K}{\Eins})=\max_{\sigma}\{\dim M(\sigma)\}$, where the maximum is taken over all 
smooth irreducible $k$-representations of $K$.
\end{lem} 
\begin{proof} Let $\sigma$ be an irreducible smooth $k$-representation of $K$. Since $P$ is a pro-$p$ group, $\sigma^P\neq 0$, and 
hence $\sigma$ is a quotient of $\Indu{P}{K}{\Eins}$. Since the functor $\lambda\mapsto M(\lambda)$ is right exact by 
Lemma \ref{Mfunct}, $M(\sigma)$ is a quotient of $M(\Indu{P}{K}{\Eins})$, and hence $\dim M(\Indu{P}{K}{\Eins})\ge \dim M(\sigma)$, 
for all smooth irreducible representations $\sigma$. The reverse inequality follows from the last part of Corollary \ref{fg}.   
\end{proof}

\begin{prop}\label{regproj} Let $x_1, \ldots, x_d$ be an $N$-regular sequence in $R$ and let $\qq$ be the ideal of $R$ generated by it.
If $x_1, \ldots, x_r$ is a system of parameters for $M(\Indu{P}{K}{\Eins})$ and $\pdim_{\OO\br{P}} N/\qq N\le d-r$, then
the following hold:
\begin{itemize}
\item[(i)] $\pdim_{\OO\br{P}} N/\qq N =d-r$;
\item[(ii)] $N$ is projective in $\Mod^{\pro}_K(\OO)$;
\item[(iii)] $M(\Theta)$ is a Cohen-Macaulay module of dimension $r+1$;
\item[(iv)] if $\lambda\in \Mod^{\sm}_K(\OO)$ is of finite length and $M(\lambda)\neq 0$
then $M(\lambda)$ is a Cohen-Macaulay module of dimension $r$.
\end{itemize}
\end{prop}
\begin{proof} Let $\qq'$ be the ideal of $R$ generated by $x_1, \ldots, x_r$. Lemma \ref{fgqq} implies that 
$N/\qq' N$ is a finitely generated $\OO\br{P}$-module. Corollary \ref{pdim} implies that 
$$\pdim_{\OO\br{P}}N/\qq' N= \pdim_{\OO\br{P}} N/\qq N -(d-r)\le 0.$$ 
Since $N/\qq' N$ is non-zero, the inequality must be an equality. Thus $\pdim N/\qq N=d-r$ and $N/\qq' N$ is a finitely generated projective 
$\OO\br{P}$-module. Since $P$ is a pro-$p$ Sylow of $K$, $N/\qq'N$ is also a projective $\OO\br{K}$-module.
Lemma \ref{MCM} implies that $M(\Theta)$ is a Cohen-Macaulay module of dimension $r+1$.  
Corollary \ref{proj} implies that $N$ is projective in $\Mod^{\pro}_K(\OO)$.
\end{proof} 

\begin{cor}\label{mainproj} Suppose that $R$ is Cohen-Macaulay and $N$ is $R$-flat. If 
\begin{equation}\label{ineq_dim}
\pdim_{\OO\br{P}} k\wtimes_R N + \dim_R M(\Indu{P}{K}{\Eins})\le \dim R,
\end{equation}
then \eqref{ineq_dim} is an equality, $N$ is  projective in $\Mod^{\pro}_K(\OO)$ and $M(\Theta)$ is a Cohen-Macaulay 
module of dimension $1+  \dim_R M(\Indu{P}{K}{\Eins})$. If $\lambda\in \Mod^{\sm}_K(\OO)$ is of finite length and $M(\lambda)\neq 0$
then $M(\lambda)$ is a Cohen-Macaulay module of dimension $\dim_R M(\Indu{P}{K}{\Eins})$.
\end{cor} 
\begin{proof} Let $r=\dim_R M(\Indu{P}{K}{\Eins})$ and let $x_1, \ldots, x_r$ be a system of parameters 
for  $M(\Indu{P}{K}{\Eins})$. We may extend it to a system of parameters of $R$, $\underline{x}:=(x_1,\ldots, x_d)$. 
Since $R$ is Cohen-Macaulay any system of parameters is $R$-regular. Since $N$ is $R$-flat, 
the sequence $\underline{x}$ is also $N$-regular. Let $\qq$ be the ideal of $R$ generated by 
$\underline{x}$. Since $\qq$ is generated by a system of parameters, $R/\qq$ is an $R$-module of finite length. 
Since $R$ is local the irreducible subquotients are isomorphic to $k$. Since $N$ is $R$-flat, 
we deduce that $N/\qq N$ has a filtration with graded pieces isomorphic to $k\wtimes_R N$. This implies that 
$\pdim_{\OO\br{P}} N/\qq N= \pdim_{\OO\br{P}} k\wtimes_R N$. The assertion now follows from Proposition \ref{regproj}.
\end{proof}

\subsection{Summary}\label{summary} We sum up the main results of this section. Let $R$ be a complete local noetherian commutative $\OO$-algebra
with residue field $k$. Let $K$ be a compact $p$-adic analytic group and let $P$ be its pro-$p$ Sylow subgroup. Let $N$ be a finitely generated 
$R\br{K}$-module, let $V$ be a continuous representation of $K$ on a finite dimensional $L$-vector space, and let $\Theta$ be a $K$-invariant 
$\OO$-lattice in $V$. In Definition \ref{M_Th} we have defined a compact $R$-module $M(\Theta)$, which is $\OO$-torsion free.
 Since $N$ is finitely generated over $R\br{K}$, Proposition \ref{fgM} implies that $M(\Theta)$ is a finitely generated $R$-module. Let $d$ be the Krull dimension of $M(\Theta)$. 
The $d$-dimensional cycle of $M(\Theta)$ is independent of the choice of a $K$-invariant lattice in $V$ by Lemma \ref{cycindep}.
Finally we recall that to each $R[1/p]$-module of finite length $\md$ we have associated an admissible unitary $L$-Banach space representation
of $K$, $\Pi(\md)$, see \S \ref{Banachspace}, \eqref{def_Pi_md}.

\begin{thm}\label{A} Let $\mathfrak a$ be the $R$-annihilator of $M(\Theta)$. If the following hold:
\begin{itemize} 
\item[(a)] $N$ is projective in $\Mod^{\pro}_K(\OO)$; 
\item[(b)] $R/\mathfrak a$ is equidimensional and all the associated primes are minimal;
\item[(c)] there exists a dense subset $\Sigma$ of $\supp M(\Theta)$, contained in $\MaxSpec R[1/p]$, such that for all 
$\nn\in \Sigma$ the following hold:
\begin{itemize}
\item[(i)] $\dim_{\kappa(\nn)} \Hom_K(V, \Pi(\kappa(\nn)))=1$;
\item[(ii)] $\dim_{\kappa(\nn)} \Hom_K(V, \Pi(R_{\nn}/\nn^2))\le d$;
\end{itemize}
\end{itemize}
then $R/\mathfrak a $ is reduced, of dimension $d$ and we have an equality of $(d-1)$-dimensional cycles
$$ z_{d-1}(R/(\varpi, \mathfrak a))= \sum_{\sigma} m_{\sigma} z_{d-1}(M(\sigma)),$$
where the sum is taken over the set of isomorphism classes of smooth irreducible $k$-representations of $K$, 
$m_{\sigma}$ is the multiplicity with which $\sigma$ occurs as a subquotient of $\Theta/(\varpi)$ and $M(\sigma)$ 
is an $R$-module defined in Definition \ref{M_sigma}.  
\end{thm}
\begin{proof} Let $M=M(\Theta)$. Proposition \ref{RHS4} and conditions (b) and (c) imply that $R/\mathfrak a$ is reduced. 
Since $M$ is a finitely generated $R$-module, the dimension of $R/\mathfrak a$ is equal to 
the dimension of $M$, which is $d$. Proposition \ref{get_to_R}
implies that $R/\mathfrak a$ and $M$ have the same $d$-dimensional cycles. Now $\varpi$ is $M$-regular, and hence 
$R/\mathfrak a$-regular. This implies that $M/(\varpi)$ and $R/(\varpi, \mathfrak a)$ have the same $(d-1)$-dimensional 
cycles, see for example \cite[2.2.10]{emertongee}. The last assertion of the Theorem follows from the Proposition \ref{cycles_modp}.
\end{proof}

\begin{remar} Let  $A$ be a complete local noetherian  $\OO$-algebra, which is $\OO$-torsion free. Then the minimal primes of 
$A$ and $A[1/p]$ coincide. Hence, any dense subset of $\Spec A[1/p]$ is also dense in $\Spec A$. 
Since $\Spec A[1/p]$ is Jacobson, $\MaxSpec A[1/p]$ is dense in $\Spec A[1/p]$. If $A$ is equidimensional  of 
dimension at least $2$ then $\MaxSpec A[1/p]$ minus finitely many points is dense in $\Spec A[1/p]$.
\end{remar}

The following will allow us to check the conditions (a) and (b) of Theorem \ref{A}.

\begin{thm}\label{B}  Suppose 
that $R$ is Cohen-Macaulay and $N$ is $R$-flat. If 
\begin{equation}\label{ineq_dim2}
\pdim_{\OO\br{P}} k\wtimes_R N + \max_{\sigma}\{\dim_R M(\sigma)\}\le \dim R,
\end{equation}
where the maximum is taken over all the irreducible smooth $k$-representations of $K$, and $M(\sigma)$ is an $R$-module defined in Definition 
\ref{M_sigma}, then the following hold:
\begin{itemize}
\item[(o)] \eqref{ineq_dim2} is an equality;
\item[(i)] $N$ is  projective in $\Mod^{\pro}_K(\OO)$;
\item[(ii)] $M(\Theta)$ is a Cohen-Macaulay module;
\item[(iii)] $R/\ann_R M(\Theta)$ is equidimensional, and all the associated prime ideals are  minimal.
\end{itemize} 
\end{thm}
\begin{proof} Parts (o), (i) and (ii) follow from Lemma \ref{dimmax} and Corollary \ref{mainproj}. Part (iii) follows from 
Proposition \ref{HS3}.
\end{proof}

\section{Preliminaries on the representations of $\GL_2(\Qp)$}
From now on we assume\footnote{unless it is stated otherwise.}   that  $p\ge 5$ and let $G:=\GL_2(\Qp)$, $K:=\GL_2(\Zp)$ and let $K_1\subset I_1\subset I$ be the
subgroups of $K$, consisting of matrices congruent  modulo $p$ to identity, unipotent upper triangular matrices and upper triangular matrices 
respectively. We let $\mathfrak K$ be the $G$-normalizer of $I$. Let $P$ be the subgroup of upper triangular matrices in $G$, $U$ the subgroup of unipotent 
upper triangular matrices in $G$, and $T$ the subgroup of diagonal matrices in $G$. If $\chi_1$ and $\chi_2$ are characters of $\Qp^{\times}$, then we 
write $\chi_1\otimes \chi_2$ for a character of $P$, which maps $\bigl (\begin{smallmatrix} a & b \\ 0 & d\end{smallmatrix} \bigr)$ to $\chi_1(a)\chi_2(d)$. 

Let $Z$ be the centre of $G$, so that $Z\cong \Qp^{\times}$. We fix a character $\zeta: Z\rightarrow \OO^{\times}$. We let $\Mod^{\sm}_{G, \zeta}(\OO)$ be the 
full subcategory of $\Mod^{\sm}_G(\OO)$ consisting of representations with central character $\zeta$, by which we mean that $Z$ acts by $\zeta$.  
We say that $\tau\in \Mod^{\sm}_{G, \zeta}(\OO)$
is \textit{locally finite} if for all $v\in \tau$ the $\OO[G]$-submodule generated by $v$ is of finite length. We let 
 $\Mod^{\lfin}_{G, \zeta}(\OO)$ be the full subcategory of $\Mod^{\sm}_{G, \zeta}(\OO)$ consisting of  all locally finite representations.
 We let $\dualcat(\OO)$ be the full subcategory of $\Mod^{\pro}_G(\OO)$ anti-equivalent to $\Mod^{\lfin}_{G, \zeta}(\OO)$ by Pontryagin duality.
 We will denote by $\Mod^{\lfin}_{G, \zeta}(k)$ and $\dualcat(k)$ the full subcategories of $\Mod^{\lfin}_{G, \zeta}(\OO)$ and $\dualcat(\OO)$ respectively, consisting of objects 
 which are killed by $\varpi$, so that  we work with $k$-vector spaces instead of $\OO$-torsion modules. Every object in $\Mod^{\sm}_{G, \zeta}(\OO)$ and  $\Mod^{\sm}_{G, \zeta}(k)$
 has an injective envelope, \cite[Cor.2.3]{cmf}. Dually, every object in $\dualcat(\OO)$ and $\dualcat(k)$ has a projective envelope.
 
  Let $\Ban^{\adm}_{G, \zeta}(L)$ be the category of admissible unitary 
$L$-Banach space representations of $G$ with central character $\zeta$. This category is abelian, see \cite{iw}.
A representation $\Pi$ in $\Ban^{\adm}_{G, \zeta}(L)$ is absolutely irreducible, if $\Pi_{L'}$ is irreducible in 
$\Ban^{\adm}_{G, \zeta}(L')$ for all finite extensions $L'$ of $L$. If $\Pi\in \Ban^{\adm}_{G, \zeta}(L)$ is irreducible then there is a finite 
extension $L'$ of $L$ such that $\Pi_{L'}$ is isomorphic to a finite direct sum of absolutely irreducible representations. If 
$\Pi\in \Ban^{\adm}_{G, \zeta}(L)$ is absolutely irreducible then $\End_G^{\cont}(\Pi)=L$. For $\GL_2(\Qp)$ 
these assertion can be deduced from Proposition 1.8, Corollary 1.9 and Theorems 1.10, 1.11 of \cite{cmf}. Dospinescu and Schraen have 
proved the assertions for any $p$-adic Lie group in \cite{DS}.

An absolutely irreducible Banach space representation 
is called non-ordinary if it is not a subquotient of a parabolic induction of a unitary character. If $\Theta$ is an open 
bounded $G$-invariant lattice in $\Pi\in \Ban^{\adm}_{G, \zeta}(L)$ then $\Theta^d:=\Hom_{\OO}(\Theta, \OO)$ equipped with the
weak topology is an object of $\dualcat(\OO)$, see \cite[Lem.4.4, 4.6]{cmf}.

Let $W_{\Qp}$ be the Weil group of $\Qp$ and $W^{\mathrm{ab}}_{\Qp}$ be its maximal abelian quotient. We normalize the isomorphism given by the local class field theory  $\Qp^{\times}\overset{\cong}{\rightarrow} W_{\Qp}^{\mathrm{ab}}$ so that uniformizers correspond to geometric Frobenii. This allows us to consider characters of $G_{\Qp}$ as characters of $\Qp^{\times}$, 
and unitary $p$-adic characters of $\Qp^{\times}$, or characters of $\Qp^{\times}$  which have finite image as  characters of $G_{\Qp}$.
 With this identification the $p$-adic cyclotomic character is identified with the character $\varepsilon: \Qp^{\times}\rightarrow L^{\times}$, $x\mapsto x|x|$, and we denote its reduction modulo 
 $\varpi$ by $\omega$.
 
 In \cite{colmez}, Colmez has defined an exact and covariant functor $\VV$ from the category of smooth, finite length representations of $G$ on $\OO$-torsion modules with 
 a central character to the category of continuous finite length representations of $G_{\Qp}$ on $\OO$-torsion modules. This functor enables us to make the connection 
 between the $\GL_2(\Qp)$ and $G_{\Qp}$ worlds. We modify Colmez's functor to obtain an exact covariant functor $\cV: \dualcat(\OO)\rightarrow \Mod^{\pro}_{G_{\Qp}}(\OO)$ as follows. 
 Let  $M$ be in $\dualcat(\OO)$, if it is of finite length then $\cV(M):=\VV(M^{\vee})^{\vee}(\varepsilon \zeta)$, 
where $\vee$ denotes the Pontryagin dual.  In general, we may write 
$M\cong \varprojlim M_i$, where the limit is taken over all quotients of finite length in $\dualcat(\OO)$ and  we define
$\cV(M):=\varprojlim \cV(M_i)$. Let  $\pi\in \Mod^{\lfin}_{G, \zeta}(k)$ be absolutely irreducible, then $\pi^{\vee}$ is an object of $\dualcat(\OO)$, and 
if $\pi$ is supersingular in the sense of \cite{bl}, then $\cV(\pi^{\vee})\cong \VV(\pi)$ is an absolutely irreducible continuous representations of $G_{\Qp}$ associated to $\pi$ by Breuil  in \cite{breuil1}. 
If $\pi\cong\Indu{P}{G}{\chi_1\otimes\chi_2 \omega^{-1}}$ then $\cV(\pi^{\vee})\cong \chi_1$. If $\pi\cong \chi\circ \det$ then $\cV(\pi^{\vee})=0$ and if $\pi\cong \Sp\otimes \chi\circ\det$, where 
$\Sp$ is the Steinberg representation, then $\cV(\pi^{\vee})\cong \chi$.  If $\Pi\in \Ban^{\adm}_{G, \zeta}(L)$ we let $\cV(\Pi):=\cV(\Theta^d)\otimes_{\OO} L$, where $\Theta$ is any open bounded
$G$-invariant lattice in $\Pi$. So that $\cV$ is exact and contravariant on $ \Ban^{\adm}_{G, \zeta}(L)$.

\section{The setup}\label{points}

Let $\rho:G_{\Qp}\rightarrow \GL_2(k)$ be  a continuous representation of the absolute Galois group of $\Qp$, such that $\End_{G_{\Qp}}(\rho)=k$. 
Let $\psi: G_{\Qp}\rightarrow \OO^{\times}$ be a continuous character, such that $\psi=\zeta \varepsilon$. We assume that $\zeta$ is such that 
$\psi \equiv \det \rho \pmod{\varpi}$. Let $R^{\psi}_{\rho}$ be the 
universal deformation ring, solving the deformation problem of $\rho$ with a fixed determinant equal to $\psi$, and let $\rho^{\un}$ be the universal deformation
of $\rho$ with determinant $\psi$.  For the rest of the section we assume that we are given $N$ in $\dualcat(\OO)$ and a continuous homomorphism of $\OO$-algebras
$R^{\psi}_{\rho}\rightarrow \End_{\dualcat(\OO)}(N)$ such that the following are satisfied: 
\begin{itemize}
\item[(N0)] $k\wtimes_{R^{\psi}_{\rho}} N$ is of finite length in $\dualcat(\OO)$, and is finitely generated  over $\OO\br{K}$;
\item[(N1)] $\Hom_{\SL_2(\Qp)}(\Eins, N^{\vee})=0$;
\item[(N2)] $\cV(N)$ and $\rho^{\un}$ are isomorphic as $R^{\psi}_{\rho}\br{G_{\Qp}}$-modules;
\end{itemize}

\begin{remar} We will verify the existence of such $N$ in \S \ref{generic_case} and \S\ref{non_generic_case}.
\end{remar} 

\begin{remar} If $\md$ is a compact   $R^{\psi}_{\rho}$-module, then \cite[Lem. 5.49]{cmf} implies that 
\begin{equation}\label{com_tensor}                               
\cV(\md\wtimes_{R^{\psi}_{\rho}} N)\cong \md \wtimes_{R^{\psi}_{\rho}} \cV(N).
\end{equation}
If $\md$ is finitely generated over $R^{\psi}_{\rho}$ then the completed tensor product coincides with the usual one. One can show using $k\wtimes_{R^{\psi}_{\rho}}\cV(N)\cong \cV(k\wtimes_{R^{\psi}_{\rho}} N)$ that (N1) and (N2) imply (N0),  but we add this condition for simplicity.
\end{remar}

For each finite length $R^{\psi}_{\rho}[1/p]$-module $\md$, we let $\Pi(\md)$ be the admissible unitary $L$-Banach space representation of $G:=\GL_2(\Qp)$ defined by 
\eqref{def_Pi_md}.

\begin{lem}\label{vpi} Let $\md$ be an $R^{\psi}_{\rho}[1/p]$-module of finite length, then $$\cV(\Pi(\md))\cong \md\otimes_{R^{\psi}_{\rho}} \cV(N).$$
\end{lem}
\begin{proof} Let $\md^0$ be a finitely generated $R^{\psi}_{\rho}$-submodule of $\md$, which is an $\OO$-lattice in $\md$.  It follows from 
(N0) and \cite[Lem. 4.25]{cmf} that the maximal torsion free quotient $(\md^0\wtimes_{R^{\psi}_{\rho}} N)_{\mathrm{tf}}$ of $\md^0\wtimes_{R^{\psi}_{\rho}} N$ is an object of $\dualcat(\OO)$. 
We have
$$\Pi(\md)=\Hom_{\OO}^{\cont}(\md^0\wtimes_{R^{\psi}_{\rho}} N, L) \cong \Hom_{\OO}^{\cont}((\md^0\wtimes_{R^{\psi}_{\rho}} N)_{\mathrm{tf}}, L)$$
with the  supremum norm. 
Then $\Xi:= \Hom_{\OO}^{\cont}((\md^0\wtimes_{R^{\psi}_{\rho} }N)_{\mathrm{tf}}, \OO)$ is a unit ball in $\Pi(\md)$. It follows from \cite{iw} that 
$\Xi^d\cong (\md^0\wtimes_{R^{\psi}_{\rho}} N)_{\mathrm{tf}}$. We thus have an exact sequence: 
\begin{equation}\label{starbucks}
0\rightarrow  (\md^0\wtimes_{R^{\psi}_{\rho} }N)_{\mathrm{tors}} \rightarrow \md^0\wtimes_{R^{\psi}_{\rho}} N \rightarrow \Xi^d \rightarrow 0
\end{equation}
Applying $\cV$  to \eqref{starbucks} and using \eqref{com_tensor} we get an exact sequence  of $G_{\Qp}$-representations:
\begin{equation}\label{starbucks1}
0\rightarrow  \cV((\md^0\wtimes_{R^{\psi}_{\rho} }N)_{\mathrm{tors}}) \rightarrow \md^0\wtimes_{R^{\psi}_{\rho}} \cV(N) \rightarrow \cV(\Xi^d) \rightarrow 0
\end{equation}
Since $\cV(N)$ is isomorphic to $\rho^{\un}$ as a $R^{\psi}_{\rho}\br{G_{\Qp}}$-module, it is a free $R^{\psi}_{\rho}$-module of rank $2$. In particular, 
$\md^0\wtimes_{R^{\psi}_{\rho}} \cV(N)$ is isomorphic to $\md^0\oplus \md^0$ as an $\OO$-module, which implies that it is $\OO$-torsion free. Since 
$ \cV((\md^0\wtimes_{R^{\psi}_{\rho} }N)_{\mathrm{tors}})$ is an $\OO$-torsion module, we deduce that 
$ \cV((\md^0\wtimes_{R^{\psi}_{\rho} }N)_{\mathrm{tors}})=0$, and hence $\md^0\wtimes_{R^{\psi}_{\rho}} \cV(N) \cong \cV(\Xi^d)$, and by inverting $p$ we obtain
$\md\otimes_{R^{\psi}_{\rho}} \cV(N)\cong \cV(\Pi)$. 
\end{proof}

\begin{lem}\label{derive} $\Pi(\kappa(\nn))^{\SL_2(\Qp)}=0$,  $\forall \nn\in \MaxSpec R^{\psi}_{\rho}[1/p]$.
\end{lem}
\begin{proof}  If $\Pi(\kappa(\nn))^{\SL_2(\Qp)}$ is non-zero, then dually we obtain a surjective, $\SL_2(\Qp)$-equivariant  homomorphism 
$\OO_{\kappa(\nn)}\wtimes_{R^{\psi}_{\rho}} N \twoheadrightarrow \OO_{\kappa(\nn)}$, where the $\SL_2(\Qp)$ acts trivially on the target. Reducing this map modulo 
$\varpi$ and taking Pontryagin duals gives us a non-zero, $\SL_2(\Qp)$-equivariant  homomorphism from $\Eins$ to $N^{\vee}$, which contradicts (N1).
\end{proof}

 The proof of the following Lemma is left as an exercise for the reader.
 
\begin{lem}\label{triv} Let $W$ be a continuous $2$-dimensional $L$-representation of $G_{\Qp}$ and let 
$M$ be a $G_{\Qp}$-stable $\OO$-lattice in $W$. Then the following hold:
\begin{itemize}
\item[(i)] if $M/\varpi M$ is an absolutely irreducible $k$-representation of $G_{\Qp}$ then 
$W$ is absolutely irreducible;
\item[(ii)] if $(M/\varpi M)^{ss}$ is a sum of distinct characters and $W$ is irreducible 
then $W$ is absolutely irreducible;
\item[(iii)] if $M/\varpi M$ is a non-split extension of distinct characters and $W$ is reducible
then $W$ is indecomposable.
\end{itemize}
\end{lem}

\begin{lem}\label{V_Pi} Let $\nn$ be a maximal ideal of $R^{\psi}_{\rho}[1/p]$ and let $\kappa(\nn)$ be its residue field. Then $\cV(\Pi(\kappa(\nn)))$ is a $2$-dimensional 
$\kappa(\nn)$-representation of $G_{\Qp}$ and one of the following holds:
\begin{itemize}
\item[(i)] $\cV(\Pi(\kappa(\nn)))$ is  absolutely irreducible;
\item[(ii)]  there is a non-split exact sequence $0\rightarrow \delta_2 \rightarrow \cV(\Pi(\kappa(\nn)))\rightarrow \delta_1\rightarrow 0$, where
 $\delta_1, \delta_2: G_{\Qp} \rightarrow \kappa(\nn)^{\times}$ are distinct, continuous characters with $\delta_1\delta_2=\zeta \varepsilon$.
 \end{itemize}
 \end{lem}
 \begin{proof}  Lemma \ref{vpi} implies that $\cV(\Pi(\kappa(\nn)))\cong \kappa(\nn)\otimes_{R^{\psi}_{\rho}} \cV(N)$. Since $\cV(N)$ is isomorphic to $\rho^{\un}$ as an $R^{\psi}_{\rho}\br{G_{\Qp}}$-module by assumption (N2), $\OO_{\kappa(\nn)}\otimes_{R^{\psi}_{\rho}} \cV(N)$ is a deformation of $\rho$ to $\OO_{\kappa(\nn)}$ with determinant equal to $\zeta\varepsilon$.
  Since $\End_{G_{\Qp}}(\rho)=k$ by assumption, 
 $\rho$ is either absolutely irreducible, or a non-split extension of two distinct one dimensional $k$-representations of $G_{\Qp}$. The assertion follows from Lemma \ref{triv}, as
 $\OO_{\kappa(\nn)}\otimes_{R^{\psi}_{\rho}} \cV(N)$ is a lattice in $\kappa(\nn)\otimes_{R^{\psi}_{\rho}} \cV(N)\cong \cV(\Pi(\kappa(\nn)))$. In the reducible case the 
 condition on the determinant forces $\delta_1\delta_2=\zeta \varepsilon$.
 \end{proof}
 
 Let $\widehat{\St}$ be the universal unitary completion of the smooth Steinberg representation of $G$. This Banach space representation is absolutely
 irreducible, \cite[5.3.3]{emcoates}.
 
\begin{lem}\label{unique} If $\Pi\in \Ban^{\adm}_{G, \zeta}(L)$ is  absolutely  irreducible then $\dim_L \cV(\Pi)\le 2$ and one of the following holds: 
\begin{itemize} 
\item[(i)] $\dim_L \cV(\Pi)=2$ if and only if $\Pi$ is non-ordinary. In this case $\cV(\Pi)$ is an absolutely irreducible representation of $G_{\Qp}$;
\item[(ii)] $\dim_L \cV(\Pi)=1$ if and only if either $\Pi\cong\bigl ( \Indu{P}{G}{\delta_1\otimes\delta_2\varepsilon^{-1}})_{\cont}$ or $\Pi\cong  \widehat{\St}\otimes \delta_1\circ \det$,  where $\delta_1, \delta_2:\Qp^{\times}\rightarrow L^{\times}$ are unitary 
characters satisfying $\delta_1\delta_2= \zeta \varepsilon$ and $\delta_1^2=\zeta$, respectively. In this case $\cV(\Pi)\cong \delta_1$;
\item[(iii)] $\dim_L \cV(\Pi)=0$ if and only if $\Pi\cong \delta\circ\det$, where $\delta: \Qp^{\times}\rightarrow L^{\times}$ is a unitary character.
\end{itemize} 
Moreover, in parts (i) and (ii) the $G_{\Qp}$-representation $\cV(\Pi)$ determines $\Pi$ up to isomorphism. 
\end{lem}
\begin{proof}  If $\Pi$ is ordinary then it is a subquotient of a parabolic induction of a unitary character, and hence either 
$\Pi\cong\bigl ( \Indu{P}{G}{\delta_1\otimes\delta_2\varepsilon^{-1}})_{\cont}$ with $\delta_1\neq \delta_2\varepsilon^{-1}$, 
$\Pi \cong \widehat{\St}\otimes \delta_1\circ \det$ or $\Pi\cong \delta_1\circ \det$, where $\delta_1, \delta_2: \Qp^{\times}\rightarrow L^{\times}$ are unitary characters, see \cite[5.3.4]{emcoates}.
Since we assume that $\Pi$ has central character $\zeta$, we have $\delta_1\delta_2\varepsilon^{-1}=\zeta$ and $\delta_1^2=\zeta$, respectively.  It follows from the definition of $\cV$ that in the first two cases 
we have $\cV(\Pi)\cong \delta_1$, and $\cV(\Pi)=0$ if $\Pi$ is a character.  Note that given $\delta_1$, $\delta_2= \zeta \varepsilon \delta_1^{-1}$ and hence  $\cV(\Pi)$ determines $\Pi$ up to isomorphism.
If $\Pi$ is non-ordinary then the assertion that $\cV(\Pi)$ is an absolutely irreducible  $2$-dimensional representation of $G_{\Qp}$, which determines $\Pi$ up to isomorphism, follows from Theorem \cite[Thm.11.4]{cmf}.
\end{proof}

\begin{lem}\label{pieces} Let $\Pi\in \Ban^{\adm}_{G, \zeta}(L)$ be of finite length with absolutely irreducible $G$-socle $\Pi_1$. Then one of the following holds: 
\begin{itemize}
\item[(i)] if $\Pi_1$ is non-ordinary then all the irreducible subquotients of $\Pi$ are isomorphic to $\Pi_1$;
\item[(ii)] if $\Pi_1\cong ( \Indu{P}{G}{\delta_1\otimes\delta_2\varepsilon^{-1}})_{\cont}$ with $\delta_1\delta_2^{-1}\neq \varepsilon^{\pm 1}$ then all the irreducible subquotients of $\Pi$ are isomorphic to either 
$\Pi_1$ or $( \Indu{P}{G}{\delta_2\otimes\delta_1\varepsilon^{-1}})_{\cont}$;
\item[(iii)]  if $\Pi_1$ is isomorphic to either $\delta\circ \det$, $\widehat{\St}\otimes\delta\circ \det$ or 
$( \Indu{P}{G}{\delta\varepsilon \otimes\delta\varepsilon^{-1}})_{\cont}$, where $\delta: \Qp^{\times}\rightarrow L^{\times}$ is a unitary character, then all irreducible subquotients of $\Pi$ lie in the set 
$\{\Eins, \widehat{\St}, ( \Indu{P}{G}{\varepsilon \otimes\varepsilon^{-1}})_{\cont}\}\otimes \delta\circ \det$.
\end{itemize}
\end{lem}
\begin{proof} In \cite[Prop.5.36]{cmf} we decomposed $\Ban^{\adm}_{G, \zeta}(L)$ into a direct sum of indecomposable subcategories called blocks. Then for each block we studied
the full subcategory of finite length representations lying in the block. We have decomposed this category into a direct sum of indecomposable subcategories and determined the irreducible 
objects in each of them, see Corollaries 6.8, 8.13, 8.14, 8.15, 9.35, 9.36, 9.37, 10.106 and Proposition 10.107 of \cite{cmf}. Since the $G$-socle of $\Pi$ is irreducible, $\Pi$ is indecomposable and 
hence must lie in one of these subcategories. The explicit description of irreducibles in the long list of Corollaries cited above implies the assertion. 
\end{proof}

\begin{prop}\label{specialize} Let $\nn$ be a maximal ideal of $R^{\psi}_{\rho}[1/p]$ and let $\kappa(\nn)$ be its residue field. Then one of the following holds:
\begin{itemize} 
\item[(i)] $\Pi(\kappa(\nn))$ is an absolutely irreducible non-ordinary $\kappa(\nn)$-Banach space representation  of $G$;
\item[(ii)] there is a non-split exact sequence of admissible $\kappa(\nn)$-Banach space representations  of $G$: 
\begin{equation}\label{reducible_case_1}
0\rightarrow \Pi_1\rightarrow \Pi(\kappa(\nn))\rightarrow \Pi_2\rightarrow 0
\end{equation}
such that $\delta_1:= \cV(\Pi_1)$ and $\delta_2:=\cV(\Pi_2)$ are distinct continuous characters of $G_{\Qp}$ satisfying $\delta_1\delta_2=\zeta\varepsilon$, and $\Pi_2^{\SL_2(\Qp)}=0$.
\end{itemize} 
Moreover, in (ii) one of the following holds: 
\begin{itemize}
\item[(a)] if $\delta_1\delta_2^{-1}\neq \varepsilon^{\pm 1}$ then 
$$ \Pi_1\cong ( \Indu{P}{G}{\delta_1\otimes\delta_2\varepsilon^{-1}})_{\cont}, \quad \Pi_2\cong ( \Indu{P}{G}{\delta_2\otimes\delta_1\varepsilon^{-1}})_{\cont}.$$
\item[(b)] if $\delta_1= \delta_2 \varepsilon^{-1}$ then $\Pi_2\cong ( \Indu{P}{G}{\varepsilon\otimes\varepsilon^{-1}})_{\cont}\otimes \delta_1\circ \det$ and there is a non-split exact sequence 
$$0\rightarrow \widehat{\St}\otimes \delta_1\circ \det \rightarrow \Pi_1\rightarrow (\delta_1\circ \det)^{\oplus m} \rightarrow 0$$
with $m\le 2$.
\item[(c)] if $\delta_1= \delta_2 \varepsilon$ then $\Pi_1\cong ( \Indu{P}{G}{\varepsilon\otimes\varepsilon^{-1}})_{\cont}\otimes \delta_2\circ \det$ and there is a non-split exact sequence 
$$0\rightarrow \widehat{\St}\otimes \delta_2\circ \det \rightarrow \Pi_2\rightarrow (\delta_2\circ \det)^{\oplus m} \rightarrow 0$$
with $m\le 2$.
\end{itemize}
\end{prop}

\begin{proof} Let $\Pi:=\Pi(\kappa(\nn))$ and let $\tilde{\rho}:= \cV(\Pi(\kappa(\nn)))$.   In the proof we will use repeatedly that a non-zero admissible unitary Banach space representation of $G$ contains an irreducible subrepresentation. This fact follows from \cite[Lem.5.8]{comp}. The reader should imagine that $\Pi$ is built out of Lego bricks. We will take $\Pi$ apart by starting removing the bricks at the bottom and identifying them using Lemmas \ref{unique} and \ref{pieces}. We will then try to put the pieces back together. For  $\Pi_1, \Pi_2\in \Ban^{\adm}_{G, \zeta}(\kappa(\nn))$ we let 
$e^1(\Pi_1, \Pi_2)=\dim_{\kappa(\nn)} \Ext^1_{G, \zeta}(\Pi_1, \Pi_2)$, where the $\Ext$-group is the Yoneda-$\Ext$ computed in $\Ban^{\adm}_{G, \zeta}(\kappa(\nn))$. The knowledge 
of $e^1(\Pi_1, \Pi_2)$ for irreducible $\Pi_1$ and $\Pi_2$ amounts to knowing, which bricks  and in how many ways fit together. 

Let $\Pi^1$ be an irreducible $\kappa(\nn)$-subrepresentation of $\Pi$. 
Since $\Pi^1$ is irreducible and admissible, $\End_G^{\cont}(\Pi^1)$ is a skew field. Since $\Pi^{\SL_2(\Qp)}= 0$ by Lemma \ref{derive}, we have $\cV(\Pi^1)\neq 0$. Since $\cV$ is contravariant, $\cV(\Pi^1)$ is a quotient of $\tilde{\rho}$. 
Since $\tilde{\rho}$ is either absolutely irreducible or a non-split extension of distinct characters by Lemma \ref{V_Pi}, we get that $\End_{G_{\Qp}}(\cV(\Pi^1))\cong \kappa(\nn)$. 
Thus the map $\End_G^{\cont}(\Pi^1)\overset{\cV}{\rightarrow} \End_{G_{\Qp}}(\cV(\Pi^1))^{\mathrm{op}}\cong\kappa(\nn)$ is non-zero, and hence an isomorphism. It follows from \cite[Lem.4.2]{cmf} that 
$\Pi^1$ is an absolutely irreducible $\kappa(\nn)$-Banach space representation of $G$. 

The $G$-socle of $\Pi$ is isomorphic to $\Pi^1$, since otherwise $\Pi$ would contain $\Pi^1\oplus \Pi^2$ as a subrepresentation, for some irreducible $\Pi^2$, and then 
$\tilde{\rho}$ would admit $\cV(\Pi_1)\oplus \cV(\Pi_2)$ as a quotient contradicting Lemma \ref{V_Pi}.  

If $\tilde{\rho}$ is absolutely irreducible then $\cV(\Pi^1)\cong \tilde{\rho}$. It follows from Lemma \ref{unique} (i) that $\Pi^1$ is non-ordinary. If $\Pi/\Pi^1$ is non-zero, then it would contain an irreducible subrepresentation $\Pi^2$. Lemma \ref{pieces} (i) implies that $\Pi^2\cong \Pi^1$. Since $\cV$ is exact, this would imply that $\dim_{\kappa(\nn)} \tilde{\rho}\ge 4$. Hence $\Pi\cong \Pi^1$.

If $\tilde{\rho}$ is not absolutely irreducible then we are in the situation of Lemma \ref{V_Pi} (ii), so that there is a non-split extension 
$0\rightarrow \delta_2\rightarrow \tilde{\rho} \rightarrow \delta_1\rightarrow 0$ with $\delta_1$, $\delta_2$ distinct characters with $\delta_1\delta_2=\zeta \varepsilon$.
Since $\cV(\Pi^1)$ is a quotient of $\tilde{\rho}$ we get that $\cV(\Pi_1)\cong \delta_1$. 

If $\delta_1\delta_2^{-1}\neq \varepsilon^{\pm 1}$ then Lemma \ref{unique} (ii) implies that $\Pi^1\cong ( \Indu{P}{G}{\delta_1\otimes\delta_2\varepsilon^{-1}})_{\cont}$. Let $\Pi^2$ be an irreducible subrepresentation 
of $\Pi/\Pi^1$. Since $\cV(\Pi/\Pi^1)\cong \delta_2$ Lemmas \ref{pieces} (ii) and \ref{unique} (ii) imply that  $\Pi^2\cong  ( \Indu{P}{G}{\delta_2\otimes\delta_1\varepsilon^{-1}})_{\cont}$. Let $\Pi'$ denote the quotient of 
$\Pi/\Pi^1$ by $\Pi^2$ then $\cV(\Pi')=0$, and the same argument gives that $\Pi'$ cannot contain an irreducible representation, and hence $\Pi'=0$.

If $\delta_1\delta_2^{-1}=\varepsilon^{\pm 1}$ then it follows from Lemmas \ref{unique}, \ref{pieces} that if $\Pi'$ is a subrepresentation of $\Pi$ of finite length then all the irreducible subquotients of 
$\Pi'$ lie in the set  $\{\Eins, \widehat{\St}, ( \Indu{P}{G}{\varepsilon \otimes\varepsilon^{-1}})_{\cont}\}\otimes \delta\circ \det$. Since $\cV$ respects twisting by characters, we may assume that $\delta$ is the trivial character. Then either ($\delta_1=\Eins$ and $\delta_2=\varepsilon$) or ($\delta_1=\varepsilon$ and $\delta_2=\Eins$). To ease the notation we let 
$\Pi_{\alpha}:= ( \Indu{P}{G}{\varepsilon \otimes\varepsilon^{-1}})_{\cont}$. Below we will  use the computations made in \S \ref{ext_banach}.

If $\delta_1=\Eins$ then $\Pi^1\cong \widehat{\St}$. Let $\Pi_1$ be the subrepresentation of $\Pi$ fitting into the exact sequence 
$0\rightarrow \widehat{\St}\rightarrow \Pi_1\rightarrow (\Pi/\widehat{\St})^G\rightarrow 0$. Since $\Pi_1^G=0$ and $e^1(\Eins, \widehat{\St})\le 2$ we get that 
$(\Pi/\widehat{\St})^G\cong \Eins^{\oplus m}$ with $m\le 2$. We let $\Pi_2:= \Pi/\Pi_1$. Since $e^1(\Eins, \Eins)=0$ we get $\Pi_2^G=0$. Since 
$\cV(\Pi_2)\cong \varepsilon$, we deduce that $\Pi_{\alpha}$ is a subrepresentation of $\Pi_2$. Since $e^1(\Eins, \Pi_{\alpha})=0$ and $\Pi_2^G=0$, 
$\Eins$ cannot be a subrepresentation of $\Pi_2/\Pi_{\alpha}$. Since $\cV(\Pi_2/\Pi_{\alpha})=0$ we deduce that $\Pi_2\cong \Pi_{\alpha}$.

If $\delta_1=\varepsilon$ then a very similar argument shows that $\Pi^1\cong \Pi_{\alpha}$ and we have a non-split exact sequence
$0\rightarrow \widehat{\St}\rightarrow \Pi/\Pi_{\alpha}\rightarrow \Eins^{\oplus m}\rightarrow 0$ with $m\le 2$.
\end{proof}

 \subsection{Locally algebraic vectors} We keep the same setup as in \S\ref{points}. We fix a $p$-adic Hodge type $(\mathbf w, \tau, \psi)$ as in \S\ref{introduction}.  We will show that part (c) (i) of Theorem \ref{A} holds when  $V=\sigma(\mathbf w, \tau)$ or $V=\sigma^{\mathrm{cr}}(\mathbf w, \tau)$  by using the description of locally algebraic vectors in the Banach space representations $\Pi(\kappa(\nn))$, due to Colmez \cite{colmez} and Emerton \cite{emfm}. This will allow us to reprove the existence of potentially semi-stable Galois deformation rings in our situation.  

\begin{lem}\label{homalg} Let $\pi_1$, $\pi_2$ be smooth representations of an open subgroup $H$ of $G$ and let $W_1$, $W_2$ be irreducible 
algebraic representations of $G$. If $\Hom_H(\pi_1 \otimes W_1, \pi_2\otimes W_2)\neq 0$ then $W_1\cong W_2\cong W$ (say) and 
$\Hom_H(\pi_1\otimes W, \pi_2\otimes W)\cong \Hom_H(\pi_1, \pi_2).$
\end{lem}
\begin{proof} If $\Hom_H(\pi_1 \otimes W_1, \pi_2\otimes W_2)\neq 0$, then since $\pi_1$ and $\pi_2$ are smooth, we may find an 
open subgroup $H'$ of $H$, such that $\Hom_{H'}(W_1, W_2)\neq 0$. This implies that $W_1\cong W_2$, since both representations 
are irreducible as representations of the Lie algebra of $G$. 

 Let us assume that $W_1\cong W_2\cong W$. When $\pi_1$ is the trivial representation, the statement is shown in \cite[Lem.1.4]{banal}. Let
 $\Hom_L(W, \pi_2\otimes W)^{\mathrm{sm}}$ denote the smooth vectors for the action of $H$ by conjugation on $\Hom_L(W, \pi_2\otimes W)$.
Then 
$\Hom_L(W, \pi_2\otimes W)^{\mathrm{sm}}= \bigcup_{H'} \Hom_{H'}(W, \pi_2\otimes W)\cong  \bigcup_{H'} \pi_2^{H'}\cong \pi_2,$
where the union is taken over all the open subgroups of $H$. Since $\pi_1$ is smooth, 
$\Hom_H(\pi_1\otimes W, \pi_2 \otimes W)\cong \Hom_H(\pi_1, \Hom_L(W, \pi_2 \otimes W)^{\mathrm {sm}})\cong 
\Hom_H(\pi_1, \pi_2).$
\end{proof}

\begin{lem}\label{extalg} Let $W$ be an irreducible algebraic representation of $G$, and let $\pi_1$, $\pi_2$ be smooth representations of 
$G$ over $L$, then tensoring with $W$ induces an isomorphism $\Ext^1_G(\pi_1, \pi_2)\overset{\cong}{\rightarrow} \Ext^1_G(\pi_1\otimes W, \pi_2\otimes W)$, where the source is computed in the category of smooth representations, and the target is computed in the category of locally algebraic representations. The same 
result holds, if we require the extensions to have central character.
\end{lem}
\begin{proof} We interpret both $\Ext$-groups as Yoneda-$\Ext$.  It follows from Lemma \ref{homalg} that the functor $\Hom_L^{\sm}(W, \ast)$ provides the inverse.
\end{proof}

\begin{lem}\label{exclude3} If $\nn$ is a maximal ideal of $R^{\psi}_{\rho}[1/p]$ such that  part (ii) (c) of Proposition \ref{specialize} holds then 
the space of locally algebraic vectors $\Pi(\kappa(\nn))^{\alg}$ in $\Pi(\kappa(\nn))$ is zero.
\end{lem}
\begin{remar}\label{exclude4} In this case $\rho^{\mathrm{un}}_{\nn}\cong \bigl ( \begin{smallmatrix} 1 & \ast\\ 0 & \varepsilon \end{smallmatrix}\bigr )\otimes \delta$. Since $\rho^{\mathrm{un}}_{\nn}$ indecomposable, 
\cite[Ex. 3.9]{BK} implies that $\rho^{\mathrm{un}}_{\nn}$ is not potentially semi-stable. 
\end{remar}
\begin{proof} Let $\Pi=\Pi(\kappa(\nn))$ and let $\Pi_1$ and $\Pi_2$ be as in part (ii) (c) of the Proposition \ref{specialize}. Lemma 12.5 in \cite{cmf} implies that $\Pi_1^{\alg}=0$. Thus we have an injection 
$\Pi^{\alg}\hookrightarrow \Pi_2^{\alg}$ and an exact sequence $0\rightarrow (\widehat{\St}\otimes \delta_2\circ \det)^{\alg}\rightarrow \Pi_2^{\alg}\rightarrow ((\delta_2\circ \det)^{\alg})^{\oplus m}$. If $\delta_2$
is not locally algebraic then  $(\delta_2\circ \det)^{\alg}=0$ and \cite[Lem.12.6]{cmf} implies that $(\widehat{\St}\otimes \delta_2\circ \det)^{\alg}=0$. Thus we may assume that $\delta_2$ is locally algebraic and, 
since in this case $(\Pi\otimes \delta^{-1}_2 \circ \det)^{\alg}\cong \Pi^{\alg}\otimes \delta_2^{-1} \circ \det$, we may further assume that $\delta_2=\Eins$ and so $\delta_1=\varepsilon$. Lemma \ref{derive} 
does not allow  $\Eins$ to be a subrepresentation of $\Pi$. We conclude that if $\Pi^{\alg}\neq 0$ then it contains $\St$ as a subrepresentation. However, then the universal unitary completion 
$\widehat{\St}$ of $\St$ would be a subrepresentation of $\Pi$, and so $\delta_2=\cV(\widehat{\St})$ would be a quotient of $\cV(\Pi)$ contradicting Lemma \ref{V_Pi} (ii).
\end{proof}

\begin{prop}\label{nearly_done}If $V=\sigma(\mathbf w, \tau)$ or $\sigma^{\mathrm{cr}}(\mathbf w, \tau)$ then  for all $\nn\in \MaxSpec R^{\psi}_{\rho}[1/p]$
 $$\dim_{\kappa(\nn)} \Hom_K(V, \Pi(\kappa(\nn)))\le 1.$$
 
 If $V=\sigma(\mathbf w, \tau)$ then $\dim_{\kappa(\nn)} \Hom_K(V, \Pi(\kappa(\nn)))=1$ if and only if $\rho^{\mathrm{un}}_{\nn}$
is potentially semi-stable of type $(\mathbf w, \tau, \psi)$.

If $V=\sigma^{\mathrm{cr}}(\mathbf w, \tau)$ and $\rho^{\mathrm{un}}_{\nn}\not\cong \bigl( \begin{smallmatrix}  \delta\varepsilon & \ast\\ 0 & \delta\end{smallmatrix} \bigr)$ then $\dim_{\kappa(\nn)} \Hom_K(V, \Pi(\kappa(\nn)))=1$ if and only if $\rho^{\mathrm{un}}_{\nn}$
is potentially crystalline of type $(\mathbf w, \tau, \psi)$.

If $V=\sigma^{\mathrm{cr}}(\mathbf w, \tau)$, $\rho^{\mathrm{un}}_{\nn}\cong \bigl( \begin{smallmatrix}  \delta\varepsilon & \ast\\ 0 & \delta\end{smallmatrix} \bigr)$ and $\dim_{\kappa(\nn)} \Hom_K(V, \Pi(\kappa(\nn)))=1$ then $\rho^{\mathrm{un}}_{\nn}$ is potentially semi-stable of type $(\mathbf w, \tau, \psi)$.
\end{prop}

\begin{proof} Since $V$ is a locally algebraic representation of $K$, we have
\begin{equation}\label{target_loc_alg}
\Hom_K(V, \Pi(\kappa(\nn)))\cong \Hom_K(V, \Pi(\kappa(\nn))^{\alg}),
\end{equation} 
where $\Pi(\kappa(\nn))^{\alg}$ denotes the subspace of locally algebraic vectors in $\Pi(\kappa(\nn))$. Thus, 
 if $\Hom_K(V, \Pi(\kappa(\nn)))\neq 0$ then  $\Pi(\kappa(\nn))^{\alg}$  is non-zero. 
 
 If $\Pi(\kappa(\nn))$ is an absolutely irreducible $\kappa(\nn)$-Banach space representation of $G$, 
 Colmez in \cite[VI.6.50]{colmez} and Emerton in \cite[3.3.22]{emfm} have shown that $\Pi(\kappa(\nn))^{\alg}\cong \pi\otimes W$, where 
 $\pi$ is either supercuspidal, smooth principal series or smooth special series and $W$ is an irreducible algebraic representation of $G$. Lemma \ref{homalg}
  and \eqref{target_loc_alg} imply that $W\cong \Sym^{b-a-1} L^2\otimes \det^a$ and  $\Hom_K(V, \Pi(\kappa(\nn)))\cong \Hom_K(\sigma(\tau), \pi)$ if $V=\sigma(\mathbf w, \tau)$, and $\Hom_K(V, \Pi(\kappa(\nn)))\cong\Hom_K(\sigma^{\mathrm{cr}}(\tau), \pi)$ if  $V=\sigma^{\mathrm{cr}}(\mathbf w, \tau)$.
Henniart has shown in \cite{henniart} that  these spaces are at most one dimensional. Moreover, $\Hom_K(\sigma(\tau), \pi)\neq 0$ if and only if
$\WD(\pi)|_{I_{\Qp}}\cong \tau$, where $\WD(\pi)$ is the Weil-Deligne representation  associated to $\pi$ via the classical local Langlands correspondence.
The only case, when $\sigma^{\mathrm{cr}}(\tau)\not\cong \sigma(\tau)$ is when $\tau\cong \chi\oplus \chi$, which is also the only case, 
when the monodromy operator $N$ of $\WD(\pi)$ can be non-trivial. In this case $\sigma^{\mathrm {cr}}(\tau)\cong \chi\circ \det$, and 
$\Hom_K(\sigma^{\mathrm{cr}}(\tau), \pi)\neq 0$ if and only if $\pi$ is a twist of an unramified principal series representation by $\chi\circ \det$, which
is equivalent to the monodromy operator $N=0$. Since $\VV(\Pi(\kappa(\nn)))\cong \cV(\Pi(\kappa(\nn)))$ by \cite[Lem.5.48]{cmf}, and $\cV(\Pi(\kappa(\nn)))\cong \rho^{\mathrm{un}}_{\nn}$
by Lemma \ref{vpi}, the assertion of the Proposition for irreducible $\Pi(\kappa(\nn))$ follows from \cite[VI.6.50]{colmez} and \cite[3.3.22]{emfm}.

Let us assume that part (ii) of Proposition \ref{specialize} holds, so that there is a non-split extension $0\rightarrow \Pi_1\rightarrow \Pi(\kappa(\nn))\rightarrow \Pi_2\rightarrow 0$ as in \eqref{reducible_case_1}. Then we have a non-split extension 
$0\rightarrow \delta_2\rightarrow \rho^{\un}_{\nn}\rightarrow \delta_1\rightarrow 0$
with $\delta_1\neq \delta_2$ and $\delta_1\delta_2=\varepsilon \zeta$. 

If part (ii) (c) of Proposition \ref{specialize} holds then $\Pi(\kappa(\nn))^{\alg}=0$ and $\rho^{\un}_{\nn}$ is not potentially semi-stable, see Lemma \ref{exclude3}, Remark \ref{exclude4}. 

Let us assume that  (ii) (b) of Proposition \ref{specialize} holds. 
Lemma 12.5 of \cite{cmf} implies that $\Pi_2^{\alg}=0$ and so $\Pi(\kappa(\nn))^{\alg}\cong \Pi_1^{\alg}$. There is an exact sequence of locally algebraic
representations 
$0\rightarrow (\widehat{\St}\otimes \delta_1\circ \det)^{\alg}\rightarrow \Pi_1^{\alg}\rightarrow ((\delta_1\circ \det)^{\alg})^{\oplus m}$. If $\delta_1$ is not
 locally algebraic then $\Pi(\kappa(\nn))^{\alg}=0$ and $\rho^{\un}_{\nn}$ is not potentially semistable. If $\delta_1=\chi\varepsilon^a$ for a smooth 
 character $\chi$ then $\rho^{\un}_{\nn}$ is potentially semi-stable  of type $((a, a+1), \chi\oplus \chi, \psi)$.  \cite[Lem.12.6]{cmf} implies 
 $(\widehat{\St}\otimes \delta_1\circ \det)^{\alg}= \St\otimes  \chi\varepsilon^a\circ \det$. Now $\Ext^1_{G/Z}(\Eins, \St)$ computed in the category of 
 smooth representations is $1$-dimensional, with the non-split extension realized by $(\Indu{P}{G}{|\centerdot|\otimes |\centerdot|^{-1}})_{\sm}$. 
 Lemma \ref{extalg} implies that the same holds in the category of locally algebraic representations.
 Since $\Pi(\kappa(\nn))^{\SL_2(\Qp)}=0$ by Lemma \ref{derive}, there is an exact sequence 
 $0\rightarrow \St \otimes \chi\varepsilon^a\circ \det \rightarrow \Pi(\kappa(\nn))^{\alg}\rightarrow  \chi\varepsilon^a\circ \det$. The exact sequence and 
Henniart's result imply  that $\dim_{\kappa(\nn)} \Hom_K(V, \Pi(\kappa(\nn)))\le1$, and is equal to $1$ if 
$V=\sigma((a, a+1), \chi\oplus \chi)$. If $\Hom_K(\sigma^{\mathrm{cr}}(\mathbf{w}, \tau), \Pi(\kappa(\nn)))\neq 0$
then $\Hom_K(\sigma(\mathbf{w}, \tau), \Pi(\kappa(\nn)))\neq 0$ and so $\rho^{\mathrm{un}}_{\nn}$
is potentially semi-stable of type $(\mathbf w, \tau, \psi)$.

Let us assume that (ii) (a) of Proposition \ref{specialize} holds, so that $\delta_1\delta_2^{-1}\neq \varepsilon^{\pm 1}$. It is shown in \cite[Lem.11.5]{cmf} that 
$\Pi_1^{\alg}\neq 0$, if and only if $\delta_1=\chi_1\varepsilon^a$, $\delta_2=\chi_2\varepsilon^b$, with $b>a$ and $\chi_1, \chi_2:\Qp^{\times}\rightarrow L^{\times}$ 
are smooth characters, in which case $\Pi_1^{\alg}\cong (\Indu{P}{G}{\chi_1|\centerdot|^a\otimes\chi_2|\centerdot|^{b-1}})_{\mathrm{sm}}\otimes \Sym^{b-a-1} L^2\otimes \det^a$. Since $\Pi_1$ has central 
character $\zeta$, we have $\chi_1\chi_2 \varepsilon^{b+a}=\zeta \varepsilon$. In particular, if $\Pi_1^{\alg}\neq 0$ then $\Pi_2^{\alg}=0$ and hence 
$\Pi(\kappa(\nn))^{\alg}=\Pi_1^{\alg}$.  As $\rho^{\mathrm{un}}_{\nn}\cong \bigl ( \begin{smallmatrix} \chi_2 \varepsilon^b & \ast \\ 0 & \chi_1 \varepsilon^a\end{smallmatrix}\bigr )$, it is of type $((a,b), \chi_1\oplus \chi_2, \psi)$. Hence, if $\Pi_1^{\alg}\neq 0$ we are done. 

To finish  the proof we need to show that if $\Pi_1^{\alg}=0$ then $\Pi(\kappa(\nn))^{\alg}=0$, since if 
$\Hom_K(V, \Pi(\kappa(\nn)))\cong \Hom_K(V, \Pi_2)\neq 0$, then $\delta_1=\chi_1\varepsilon^b$, $\delta_2= \chi_2\varepsilon^a$, for some smooth unitary characters
$\chi_1, \chi_2:\Qp^{\times}\rightarrow L^{\times}$ and $\rho^{\mathrm{un}}_{\nn}\cong \bigl ( \begin{smallmatrix} \chi_2 \varepsilon^a & \ast \\ 0 & \chi_1 \varepsilon^b\end{smallmatrix}\bigr )$, which is not potentially semi-stable as $b>a$.
If $\Pi(\kappa(\nn))^{\alg}\neq 0$ and $\Pi_1^{\alg}=0$ then $\Pi_2^{\alg}\neq 0$. 
We claim that the map $\Pi(\kappa(\nn))^{\alg}\rightarrow \Pi_2^{\alg}$ is surjective. 
Since $\Pi_1^{\alg}=0$ the claim implies that the map  is an isomorphism.  However, \cite[2.2.1]{be} implies that the universal unitary completion 
of $\Pi_2^{\alg}$ is isomorphic to $\Pi_2$,  which implies that the sequence splits, leading to a contradiction. To prove the claim it is enough to consider the case, 
when $\Pi_2^{\alg}$ is reducible. Since $\Pi_2^{\alg}\cong (\Indu{P}{G}{\chi_2|\centerdot|^a\otimes\chi_1|\centerdot|^{b-1}})_{\mathrm{sm}}\otimes \Sym^{b-a-1} L^2\otimes \det^a\cong (\Indu{P}{G}{\chi_2\otimes\chi_1|\centerdot|^{b-a-1}})_{\mathrm{sm}}\otimes \Sym^{b-a-1} L^2\otimes \varepsilon^a\circ\det$, either
 1) $\chi_1=\chi_2$ and $a=b+1$, which implies $\delta_1\delta_2^{-1}=\varepsilon$, 
or 2) $\chi_2=\chi |\centerdot|$ and  $\chi_1=\chi |\centerdot|^{b-a}$, for some smooth $\chi: \Qp^{\times}\rightarrow L^{\times}$. 
Since $\chi_1$ and $\chi_2$ are unitary we deduce that $b=a+1$ and hence $\delta_1\delta_2^{-1}=\varepsilon$. 
 \end{proof}

\begin{thm}\label{exist_pst_def} We assume the existence of $N$ satisfying (N0), (N1), (N2) in \S\ref{points}. For each type $(\mathbf w, \tau, \psi)$ there exists 
a reduced, $\OO$-torsion free quotient $R^{\psi}_{\rho}(\mathbf w, \tau)$ (resp. $R^{\psi, \mathrm{cr}}_{\rho}(\mathbf w, \tau)$) of $R^{\psi}_{\rho}$ such that for each $\nn\in \MaxSpec R^{\psi}_{\rho}[1/p]$, 
$\rho_{\nn}^{\mathrm{un}}$ is potentially semi-stable (resp. potentially crystalline) of type $(\mathbf w, \tau, \psi)$ if and only if $\nn\in \MaxSpec  R^{\psi}_{\rho}(\mathbf w, \tau)[1/p]$ (resp.  $\nn\in\MaxSpec  R^{\psi, \mathrm{cr}}_{\rho}(\mathbf w, \tau)[1/p]$).

Moreover, if $\Theta$ is a $K$-invariant $\OO$-lattice in $\sigma(\mathbf w, \tau)$  and 
$\mathfrak a$ is the $R^{\psi}_{\rho}$-annihilator of $M(\Theta)$, defined in \ref{M_Th}, then $R^{\psi}_{\rho}(\mathbf w, \tau)=R^{\psi}_{\rho}/\sqrt{\mathfrak a}$. If $\Theta$ is a $K$-invariant $\OO$-lattice in $\sigma^{\mathrm{cr}}(\mathbf w, \tau)$ and either $\rho\not\cong \bigl(\begin{smallmatrix} \bar{\delta} \omega & \ast\\ 0 & \bar{\delta}\end{smallmatrix}\bigr)$, $b-a>1$  or $\tau\neq \chi\oplus \chi$
then $R^{\psi, \mathrm{cr}}_{\rho}(\mathbf w, \tau)=R^{\psi}_{\rho}/\sqrt{\mathfrak a}$.
\end{thm} 

\begin{proof} If $\rho\cong \bigl(\begin{smallmatrix} \bar{\delta} \omega & \ast\\ 0 & \bar{\delta}\end{smallmatrix}\bigr)$, $b-a=1$ and $\tau=\chi\oplus \chi$ then we deduce the existence of $R^{\psi, \mathrm{cr}}_{\rho}(\mathbf w, \tau)$ from Fontaine-Laffaille theory. 
In all other cases, Propositions \ref{dim_eq} and  \ref{nearly_done} imply that $\nn\in \MaxSpec R^{\psi}_{\rho}[1/p]$ 
lies in the support of $M(\Theta)$ if and only if $\rho^{\mathrm{un}}_{\nn}$ is potentially semi-stable (resp. potentially crystalline) of type 
$(\mathbf w, \tau, \psi)$. This implies that $R^{\psi}_{\rho}/\sqrt{\mathfrak a}$ is the required ring.
\end{proof}

\begin{lem}\label{cor_to_proof} We assume that either $\rho\not\cong \bigl(\begin{smallmatrix} \bar{\delta} \omega & \ast\\ 0 & \bar{\delta}\end{smallmatrix}\bigr)$, $b-a>1$  or $\tau\neq \chi\oplus \chi$. Let $\Theta$ be a $K$-invariant $\OO$-lattice in either $\sigma(\mathbf w, \tau)$ or  $\sigma^{\mathrm{cr}}(\mathbf w, \tau)$ and 
 let $M(\Theta)$ be the $R^{\psi}_{\rho}$-module defined in \ref{M_Th}. Then for 
 almost all\footnote{This means `all except for finitely many' throughout the text.}  $\nn\in \MaxSpec R^{\psi}_{\rho}[1/p]$ lying in $\supp M(\Theta)$,  
 $\Pi(\kappa(\nn))^{\alg}$ is an irreducible representation 
of $G$.

 Moreover for such $\nn$,  if $\Pi(\kappa(\nn))$ is reducible in $\Ban^{\adm}_{G, \zeta}(\kappa(\nn))$  then there is a closed absolutely irreducible 
$\kappa(\nn)$-Banach space $G$-subrepresentation $\Pi_1$ of $\Pi(\kappa(\nn))$ such that 
$\Pi_1^{\alg}\cong \Pi(\kappa(\nn))^{\alg}$ and $\Pi_1$ is isomorphic to the universal unitary completion of $\Pi_1^{\alg}$.
\end{lem}  
\begin{proof} Let  $\nn$ be a maximal ideal of $R^{\psi}_{\rho}[1/p]$ lying in $\supp M(\Theta)$, such that $\Pi(\kappa(\nn))$ is absolutely irreducible. 
 As already explained in the proof of Proposition \ref{nearly_done}, $\Pi(\kappa(\nn))^{\alg}\cong \pi\otimes W$, where $W$ is an algebraic 
representation of $G$ completely determined by the Hodge-Tate weight $\mathbf{w}$, and $\pi$ is a smooth representation of 
$G$, which is is either supercuspidal, principal series or special series. 
If $\Pi(\kappa(\nn))^{\alg}$ is reducible then $\pi\cong (\Indu{P}{G}{|\centerdot|\otimes|\centerdot|^{-1}})_{\sm}\otimes \chi\circ\det$. 
Since $\zeta$ is equal to the product of central characters of $\pi$ and $W$, $\chi^2=\zeta \varepsilon^{1-a-b}$. There are only finitely many characters 
$\chi$ satisfying this equality. In each case, the universal unitary completion of $\pi\otimes W$ is irreducible by \cite[5.3.4]{bb}, and hence is isomorphic to 
$\Pi(\kappa(\nn))$. Thus there are only finitely many $\nn\in \supp M(\Theta)$ such that $\Pi(\kappa(\nn))$ is irreducible and $\Pi(\kappa(\nn))^{\alg}$ is 
reducible. If $\Pi(\kappa(\nn))$ is reducible  then Proposition \ref{specialize} (ii) (a) holds, since Lemma \ref{exclude3} rules out (ii) (c), and (ii) (b) would 
imply that  $\rho\cong \bigl(\begin{smallmatrix} \bar{\delta} \omega & \ast\\ 0 & \bar{\delta}\end{smallmatrix}\bigr)$, $b-a=1$ and $\tau= \chi\oplus \chi$.
It follows from the proof of Proposition \ref{nearly_done} 
that $\Pi(\kappa(\nn))^{\alg}\cong \pi\otimes W$, where $\pi$ is irreducible principal series. The universal unitary completion of $\Pi(\kappa(\nn))^{\alg}$ is absolutely irreducible, \cite[2.2.1]{be}. 
\end{proof}

\subsection{Locally algebraic vectors in extensions} We will reformulate the condition (c)(ii) in Theorem \ref{A} in terms of 
locally algebraic vectors in extensions of Banach space representations. We will verify the condition in the crystabeline, 
and in the semi-stable, non-crystalline cases. The condition is verified by Dospinescu in \cite{dospinescu} in general. 

\begin{defi}
We will say that an absolutely irreducible $\Pi$ in $\Ban^{\adm}_{G, \zeta}(L)$ satisfies the hypothesis 
(RED), if $\Pi^{\alg}\neq 0$ and the following statement holds:
\begin{itemize}
\item[(RED)] The subspace $\mathcal E$ of $\Ext^1_{G, \zeta}(\Pi, \Pi)$, generated by the extensions 
$0\rightarrow \Pi \rightarrow E\rightarrow \Pi\rightarrow 0$, such that the corresponding sequence of 
locally algebraic vectors $0\rightarrow \Pi^{\alg} \rightarrow E^{\alg}\rightarrow \Pi^{\alg}\rightarrow 0$
is exact,  is at most one dimensional.
\end{itemize}
\end{defi}

\begin{lem}\label{easy} Let $\Pi$ be in $\Ban^{\adm}_{G, \zeta}(L)$ be absolutely irreducible, satisfying (RED) and such that $\Pi^{\alg}$ 
is irreducible. Let 
\begin{equation}\label{seq1}
0\rightarrow \Pi \rightarrow E\rightarrow \Pi^{\oplus n}\rightarrow 0
\end{equation} 
be an exact sequence in $\Ban^{\adm}_{G, \zeta}(L)$, such that $\Hom_G(\Pi, E)$ is one dimensional. 
If $V$ is either  $\sigma(\mathbf w, \tau)$ or $\sigma^{\mathrm{cr}}(\mathbf w, \tau)$ then $\Hom_{K}(V, E)$ is at most $2$-dimensional.
\end{lem}
\begin{proof} We may assume that $\Hom_K(V, \Pi)$ is non-zero. Since $\Pi^{\alg}$ is irreducible by assumption, 
Lemma \ref{homalg} and \cite{henniart} imply that $\Hom_K(V, \Pi)$ is one dimensional.
 Since $\Pi$ is absolutely irreducible  $\Hom_G(\Pi, \Pi)$ is also $1$-dimensional. 
 In particular, the natural map $\Hom_G(\Pi, \Pi)\rightarrow \Hom_K(V, \Pi)$ 
is an isomorphism. Applying $\Hom_G(\Pi, \ast)$ and $\Hom_K(V, \ast)$ to \eqref{seq1} we obtain 
a commutative diagram: 
 \begin{displaymath}
\xymatrix@1{\,\Hom_{G}(\Pi, \Pi^{\oplus n})\,\ar@{^{(}->}[r] \ar[d]^{\cong}& \Ext^1_{G, \zeta}(\Pi, \Pi)\ar[d]^{\alpha}\\
\Hom_{K}(V, \Pi^{\oplus n})\ar[r]^{\beta} &\Ext^1_{K, \zeta}(V ,\Pi),}
\end{displaymath}
where $\Ext^1$ is the Yoneda $\Ext$ computed in $\Ban^{\adm}_{G, \zeta}(L)$ and $\Ban^{\adm}_{K, \zeta}(L)$ respectively. 
It follows from the diagram that $\dim \Ker \beta\le \dim \Ker \alpha$.
Since $\Pi^{\alg}$ is irreducible, the subspace $\mathcal E$, defined in (RED), is equal to $\Ker \alpha$.  
Since (RED) holds we deduce that $\dim \Ker \beta \le 1$, which implies the assertion.
\end{proof}

\begin{thm}\label{RED_is_ok} Let $\nn\in \MaxSpec R^{\psi}_{\rho}[1/p]$ be such that $\Pi(\kappa(\nn))^{\alg}$ is non-zero and is an irreducible representation of $G$ and let $\Pi$ be the unique non-zero closed irreducible $\kappa(\nn)$-Banach space $G$-sub\-rep\-re\-sen\-ta\-tion of $\Pi(\kappa(\nn))$.  
Then $\Pi$ satisfies (RED).
\end{thm} 
\begin{proof} Proposition \ref{specialize} implies that the $G$-socle $\Pi$ of $\Pi(\kappa(\nn))$ is irreducible. It follows from the proof of Proposition \ref{nearly_done}
that $\Pi^{\alg}\cong \Pi(\kappa(\nn))^{\alg}\cong \pi \otimes W$, where $\pi$ is an irreducible smooth representation of $G$ and 
$W$ an irreducible algebraic representation of $G$. 

If $\pi$ is principal series then $\Pi$ is the universal unitary completion of $\Pi^{\alg}$ by 
 \cite[5.3.4]{bb}, \cite{except} and \cite[2.2.1]{be}. Hence, taking locally algebraic vectors induces an injection 
 $\mathcal E\hookrightarrow \Ext^1_{G, \zeta}(\pi\otimes W, \pi\otimes W)$, where the target is computed in the category of locally algebraic 
 representation with central character $\zeta$. Lemma \ref{extalg}  implies that tensoring with $W$ induces an isomorphism 
 $\Ext^1_{G, \zeta_{\pi}}(\pi, \pi)\cong \Ext^1_{G,\zeta}(\pi\otimes W, \pi \otimes W)$, 
 where $\zeta_{\pi}$ is the central character of $\pi$, and the source is computed in the category of smooth representations with central character $\zeta_{\pi}$. 
 Using the projective resolution of $\pi$, constructed in \cite{schst}, one may show that this space is one dimensional.

If $\pi$ is special series and $W$ is one dimensional then $\Pi\cong \widehat{\St}\otimes \delta\circ \det$, for some unitary character $\delta: \Qp^{\times}\rightarrow L^{\times}$, and $\Ext^1_{G, \zeta}(\Pi, \Pi)=0$. If $\pi$ is special series and $W$ is not one dimensional then 
this is shown in \S\ref{semi_non_crys} below, by analyzing the self extensions of the locally analytic vectors in $\Pi$.

If $\pi$ is supercuspidal, which is the hardest case, this is proved by Dospinescu in \cite[Thm.1.4, Prop.1.3]{dospinescu}.
 \end{proof}

\begin{lem}\label{hom_def} Let $\nn$ be a maximal ideal of $R^{\psi}_{\rho}[1/p]$. Then 
$$\dim_{\kappa(\nn)}\Hom_{G_{\Qp}}(\rho^{\un}[1/p]/\nn^2, \rho_{\nn}^{\un})=1.$$
If there is an extension 
$0\rightarrow \delta_2\rightarrow \rho^{\un}_{\nn}\rightarrow \delta_1\rightarrow 0,$ then 
$$\dim_{\kappa(\nn)} \Hom_{G_{\Qp}}(\rho^{\un}[1/p]/\nn^2, \delta_1)= \dim_{\kappa(\nn)} \Ext^1_{G_{\Qp}}(\delta_1, \delta_2).$$
\end{lem}
\begin{proof} After extending scalars we may assume that $\kappa(\nn)=L$ and let $A:=L[\epsilon]/\epsilon^2$ be the dual numbers over $L$. The $\nn$-adic completion of $ R_{\rho}^{\psi}[1/p]$, which we denote by $\widehat{R}$,  pro-represents the deformation problem 
of $\rho_{\nn}^{\un}$ with determinant $\psi$ to local artinian $L$-algebras,  with universal object isomorphic to the $\nn$-adic completion of
$\rho^{\un}[1/p]$,  \cite[(2.3.3), (2.3.5)]{kisin_moduli}. Let $\Hom(\widehat{R}, A)$ be the set of homomorphisms of local $L$-algebras from $\widehat{R}$ to $A$. Then 
we have natural isomorphisms
\begin{equation} 
(\nn/\nn^2)^*\cong \Hom(\widehat{R}, A)\cong \Ext^1_{G_{\Qp}, \psi}(\rho_{\nn}^{\un}, \rho_{\nn}^{\un}),
\end{equation}
where the last term is  the subgroup of $\Ext^1_{G_{\Qp}}(\rho_{\nn}^{\un}, \rho_{\nn}^{\un})$ consisting of the extensions, 
which, when considered as deformations of $\rho_{\nn}^{\un}$ to $A$, have determinant $\psi$. The second isomorphism sends $\varphi$ to 
$r_{\varphi}:=  A\otimes_{\varphi, R_{\rho}^{\psi}} \rho^{\un}$. The first assertion can be proved in the same way as \cite[Lem.11.5]{cmf}.

Let us assume that $\rho_{\nn}^{\un}$ is reducible. Since $\Hom_{G_{\Qp}}(\rho_{\nn}^{\un}, \delta_1)$ is 
one dimensional, there is a natural isomorphism 
\begin{equation} 
\Hom_{G_{\Qp}}(\nn/\nn^2 \otimes_{R^{\psi}_{\rho}} \rho^{\un}, \delta_1)\cong (\nn/\nn^2)^*
\end{equation}

Let $C$ be the image of the natural map   
\begin{equation}
\Hom_{G_{\Qp}}(\rho^{\un}[1/p]/\nn^2, \delta_1)\rightarrow \Hom_{G_{\Qp}}(\nn/\nn^2 \otimes_{R^{\psi}_{\rho}} \rho^{\un}, \delta_1).
\end{equation}
We claim that $C$ is naturally isomorphic to $\Ext^1_{G_{\Qp}}(\delta_1, \delta_2)/\mathcal L$, where $\mathcal L$ is a one dimensional 
subspace generated by the extension class corresponding to $\rho_{\nn}^{\un}$. The claim implies the second assertion of the Lemma. 

It is immediate that $\varphi$ lies in the image of $C$ in $\Hom(\widehat{R}, A)$ if and only if $\dim_L \Hom_{G_{\Qp}}(r_{\varphi}, \delta_1)=2$.
Since $\delta_1\neq \delta_2$ the surjection $r_{\varphi}\twoheadrightarrow \delta_1^{\oplus 2}$ induces a surjection of $A[ G_{\Qp}]$-modules $r_{\varphi}\twoheadrightarrow A\otimes_{L}\delta_1$.  Since the determinant of $r_{\varphi}$ induces the trivial deformation on $\psi$, the kernel of this map is isomorphic to 
$A\otimes_{L}\delta_2$. Thus we may choose a basis of $r_{\varphi}$ as an $A$-module, such that the action of $G_{\Qp}$ is given by a matrix 
$g\mapsto \bigl( \begin{smallmatrix} \delta_2(g) & a(g)\\ 0 & \delta_1(g)\end{smallmatrix}\bigr)$, where $a:G_{\Qp}\rightarrow A$ is a continuous 
function, satisfying $a(g h)= \delta_2(g) a(h) + a(g) \delta_1(h)$, for all $g, h\in G_{\Qp}$. We may write $a=a_1+ \epsilon a_2$, where
$a_1, a_2: G_{\Qp}\rightarrow L$ are continuous functions satisfying the same transformation identity. This implies that the maps 
$\rho_1: G_{\Qp}\rightarrow \GL_2(L)$,  $g\mapsto  \bigl( \begin{smallmatrix} \delta_2(g) & a_1(g)\\ 0 & \delta_1(g)\end{smallmatrix}\bigr)$ and 
$\rho_2: G_{\Qp}\rightarrow \GL_2(L)$,  $g\mapsto  \bigl( \begin{smallmatrix} \delta_2(g) & a_2(g)\\ 0 & \delta_1(g)\end{smallmatrix}\bigr)$ are continuous 
group homomorphisms. There is an isomorphism  $\rho_1\cong r_{\varphi}/\epsilon \cong \rho_{\nn}^{\un}$. Let $c\in \Ext^1_{G_{\Qp}}(\delta_1, \delta_2)$ 
be the extension class corresponding to the isomorphism class of $\rho_2$. We leave it as an exercise for the 
reader to check that  the image of $c$ in $\Ext^1_{G_{\Qp}}(\delta_1, \delta_2)/\mathcal L$ does not depend
on the choice of basis for $\tau_{\varphi}$, used to define $\rho_2$,  and to construct the inverse $\Ext^1_{G_{\Qp}}(\delta_1, \delta_2)/\mathcal L\rightarrow \Ext^1_{G_{\Qp}, \psi}(\rho_{\nn}^{\un}, \rho_{\nn}^{\un})$.
\end{proof} 

\begin{cor}\label{c_ii_ok} We assume the existence of $N$ satisfying (N0), (N1) and (N2) in \S\ref{points}. We further 
assume that either $\rho\not\cong \bigl(\begin{smallmatrix} \bar{\delta} \omega & \ast\\ 0 & \bar{\delta}\end{smallmatrix}\bigr)$, $b-a>1$  or $\tau\neq \chi\oplus \chi$.
If $V$ is either $\sigma(\mathbf w, \tau)$ or $\sigma^{\mathrm{cr}}(\mathbf w, \tau)$ and $\Theta$ is a $K$-invariant $\OO$-lattice in $V$ then for almost all maximal ideals $\nn$ of $R^{\psi}_{\rho}[1/p]$
lying in $\supp M(\Theta)$, $$\dim_{\kappa(\nn)} \Hom_K(V, \Pi(R^{\psi}_{\rho, \nn}/\nn^2))\le 2.$$
\end{cor}
\begin{proof}  Let $\nn\in \MaxSpec R^{\psi}_{\rho}[1/p]$ be such that Lemma \ref{cor_to_proof} applies. Let $\Pi:=\Pi(\kappa(\nn))$ and let $\Pi_1$ be the 
closure of locally algebraic vectors in $\Pi$. Then $\Pi_1$ is absolutely irreducible in $\Ban^{\adm}_{G, \zeta}(\kappa(\nn))$, and is equal to the $G$-socle of $\Pi$.
Moreover, if $\Pi$ is reducible then $\Pi_1$ is the universal unitary completion of $\Pi^{\alg}$.

We claim that $\dim_{\kappa(\nn)} \Hom_G^{\cont}(\Pi_1,  \Pi(R^{\psi}_{\rho, \nn}/\nn^2))=1$. The functor $\cV$ induces an injection
\begin{equation}\label{inject_on_homs}
\Hom_G^{\cont}(\Pi_1,  \Pi(R^{\psi}_{\rho, \nn}/\nn^2))\hookrightarrow \Hom_{G_{\Qp}}(\cV(  \Pi(R^{\psi}_{\rho, \nn}/\nn^2)), \cV(\Pi_1)).
\end{equation}
Lemma \ref{vpi} implies that $\cV(\Pi(R^{\psi}_{\rho, \nn}/\nn^2))\cong \rho^{\un}[1/p]/\nn^2$.  If $\Pi$ is irreducible then 
$\Pi_1\cong \Pi$ and so $\cV(\Pi_1)\cong \rho_{\nn}^{\un}$. The claim follows from Lemma \ref{hom_def}  and \eqref{inject_on_homs}. If $\Pi$ is reducible then $\cV(\Pi_1)=\delta_1$ and $\rho_{\nn}^{\un}\cong \bigl(\begin{smallmatrix} \delta_2 & \ast \\ 0 & \delta_1\end{smallmatrix} \bigr)$, 
 where $\delta_2\delta_1^{-1}\neq \Eins, \varepsilon$. Since  $\dim_{\kappa(\nn)} \Ext^1_{G_{\Qp}}(\delta_1, \delta_2)=1$, the claim follows from Lemma \ref{hom_def}  and \eqref{inject_on_homs}.

Let $E$ be the closure of locally algebraic vectors in $\Pi(R^{\psi}_{\rho, \nn}/\nn^2)$. Since $V$ is locally algebraic we have 
$ \Hom_K(V, E)\cong \Hom_K(V, \Pi(R^{\psi}_{\rho, \nn}/\nn^2)).$
The exact sequence $0\rightarrow \nn/\nn^2\rightarrow R^{\psi}_{\rho, \nn}/\nn^2 \rightarrow \kappa(\nn)\rightarrow 0$ leads to  exact sequences
of unitary Banach space representations of $G$: 
$0\rightarrow \Pi\rightarrow \Pi(R^{\psi}_{\rho, \nn}/\nn^2)\rightarrow \Pi^{\oplus d}$, 
$0\rightarrow \Pi_1\rightarrow E\rightarrow \Pi_1^{\oplus m}\rightarrow 0$,
where $d:=\dim_{\kappa(\nn)} \nn/\nn^2$ and $m\le d$. Since $\Hom^{\cont}_G(\Pi_1, E)$ is one dimensional and $\Pi_1$ satisfies (RED) by Theorem \ref{RED_is_ok}, Lemma \ref{easy} implies that $\dim_{\kappa(\nn)} \Hom_K(V, E)\le 2$.
\end{proof}

\section{Representation theoretic input} 
In this section we will construct an exact sequence, which will be used in \S\ref{generic_case}. We make no assumption on the prime $p$. Let $\pi\in \Mod^{\sm}_{G, \zeta}(k)$ be  absolutely irreducible,  which is neither special series nor a character. Let $\wP\twoheadrightarrow \pi^{\vee}$ be a projective envelope of $\pi^{\vee}$ 
in $\dualcat(\OO)$, and let $\wE:=\End_{\dualcat(\OO)}(\wP)$.

\begin{prop}\label{Omega_nice} Let $\Omega$ be in $\Mod^{\mathrm{l.fin}}_{G,\zeta}(k)$ with $\soc_G \Omega\cong \pi$. If 
$\Omega$ is admissible and $\Omega|_K$ is injective in $\Mod^{\mathrm{sm}}_{K,\zeta}(k)$ then the following hold:
\begin{itemize} 
\item[(i)] $\Ext^1_{G, \zeta}(\pi, \Omega)$ is one dimensional; 
\item[(ii)] $\Ext^i_{G, \zeta}(\pi,\Omega)=0$ for $i\ge 2$; 
\item[(iii)] $\Ext^i_{G, \zeta}(\pi', \Omega)=0$ for all $i\ge 0$ and all $\pi'\in \Irr_{G, \zeta}(k)$, $\pi'\not\cong \pi$. 
\end{itemize}
\end{prop}
\begin{proof} Let $\xi$ be irreducible in $\Mod^{\mathrm{sm}}_{K, \zeta}(k)$ then 
$$\Ext^i_{G, \zeta}(\cIndu{KZ}{G}{\xi}, \Omega)\cong \Ext^i_{K, \zeta}(\xi, \Omega), \quad \forall i\ge 0.$$
 Since $\Omega$ is injective in $\Mod^{\mathrm{sm}}_{K, \zeta}(k)$ we deduce that $\Ext^i_{G, \zeta}(\cIndu{KZ}{G}{\xi}, \Omega)$ is zero for $i\ge 1$, 
and, since $\Omega$ is admissible, $\Hom_G(\cIndu{KZ}{G}{\xi}, \Omega)$ is finite dimensional.
Let $\phi\in \HH_{\xi}:=\End_G(\cIndu{KZ}{G}{\xi})$ be non-zero. Since $\cIndu{KZ}{G}{\xi}$ is a free 
$\HH_{\xi}$-mo\-du\-le by \cite{bl}, we have an exact sequence 
\begin{equation}\label{phi}
0\rightarrow \cIndu{KZ}{G}{\xi}\overset{\phi}{\rightarrow}
\cIndu{KZ}{G}{\xi}\rightarrow \frac{\cIndu{KZ}{G}{\xi}}{\phi}\rightarrow 0.
\end{equation} 
 Applying $\Hom_G(\ast, \Omega)$ to \eqref{phi} we obtain
\begin{itemize} 
\item[(a)] $\Ext^i_{G, \zeta}( \frac{\cIndu{KZ}{G}{\xi}}{\phi}, \Omega)=0$, $\forall i\ge 2$;
\item[(b)] $\dim_k \Ext^1_{G, \zeta}(\frac{\cIndu{KZ}{G}{\xi}}{\phi}, \Omega)=\dim_k \Hom_G(\frac{\cIndu{KZ}{G}{\xi}}{\phi}, \Omega)$.
\end{itemize}
Since every supersingular and principal series representation is of the form $\frac{\cIndu{KZ}{G}{\xi}}{\phi}$ for a suitable $\xi$ and $\phi$ by \cite{bl}, \cite{breuil1}, 
(a) implies part (ii) of the Proposition, and since $\soc_G \Omega\cong \pi$ by assumption, we have $\Hom_G(\pi, \Omega)\cong \End_G(\pi)=k$, which together 
with (b) implies (i).

Let $\pi'\in \Irr_{G, \zeta}(k)$ such that $\pi'\not\cong \pi$. Since $\soc_G\Omega\cong \pi$ this implies 
$\Hom_G(\pi', \Omega)=0$. If $\pi'$ is either supersingular or principal series then  (a) and (b) imply that (iii) holds. If
$\pi'$ is either special or a character then there exist exact non-split sequences:
\begin{equation}\label{seqa}
0\rightarrow \pi'\rightarrow \frac{\cIndu{KZ}{G}{\xi}}{T_{\xi}-\lambda}\rightarrow \pi'' \rightarrow 0, \quad 
0\rightarrow \pi''\rightarrow \frac{\cIndu{KZ}{G}{\check{\xi}}}{T_{\check{\xi}}-\mu}\rightarrow \pi' \rightarrow 0,
\end{equation} 
for some $\lambda, \mu\in k$, where $\{\xi, \check{\xi}\}=\{ \chi\circ \det,  \Sym^{p-1}k^2 \otimes \chi\circ \det\}$, $\{\pi', \pi''\}= \{\chi\circ \det, \Sp\otimes \chi\circ \det\}$,  for some character $\chi$. Since $\pi$ is neither 
special series nor a character by assumption, using \eqref{seqa} we deduce that (iii) also holds when $\pi'$ is either special or a character. 
\end{proof}
 
 \begin{thm}\label{iks}There exists $x\in \wE$, such that $x: \wP\rightarrow \wP$ is injective and the quotient 
 is a projective envelope of $(\soc_K \pi)^{\vee}$ in $\Mod^{\pro}_{K,\zeta}(\OO)$. 
 \end{thm}

 \begin{proof} It follows from \cite[Cor.9.11]{bp} that there exists an injection $\pi\hookrightarrow \Omega$, where 
 $\Omega$ is a smooth $G$-representation with central character $\zeta$, such that $\Omega|_K$ is an injective  envelope of $\soc_K \pi$ in 
 $\Mod^{\sm}_{K, \zeta}(k)$. If $\pi'$ is a non-zero $G$-subrepresentation of $\Omega$, then $\soc_K \pi'$ is non-zero, and hence $\pi'\cap \pi$ is non-zero. 
 Since $\pi$ is irreducible, we deduce that $\pi'$ contains $\pi$, and hence $\soc_G \Omega\cong \pi$. Since $\pi$ is admissible, $\soc_K \pi$ is a finite 
 direct sum of irreducible representations, \cite[Lem.3.6]{comp} implies that $\Omega$ is admissible.
 
 Let $\pi\hookrightarrow J$ be an injective envelope of $\pi$ in $\Mod^{\mathrm{l.fin}}_{G, \zeta}(k)$. We now proceed as in \cite[Rem.10.11]{cmf} 
 to deduce from Proposition \ref{Omega_nice} the existence of an exact sequence 
 \begin{equation}\label{JOmega}
 0\rightarrow \Omega \rightarrow J\rightarrow J\rightarrow 0.
 \end{equation} 
 Since the monomorphism $\pi\hookrightarrow \Omega$ is essential, there exists an embedding $\Omega\hookrightarrow J$. Since 
 $\Omega$ and $J$ have the same $G$-socle and $J$ is injective, for all $\pi'\in \Irr_{G, \zeta}(k)$ and all $i\ge 0$, we have 
 $\Ext^i_{G, \zeta}(\pi', J/\Omega)\cong \Ext^{i+1}_{G, \zeta}(\pi', \Omega)$.  Parts (ii) and (iii) of Proposition \ref{Omega_nice} imply that 
 $J/\Omega$ is injective in  $\Mod^{\mathrm{l.fin}}_{G, \zeta}(k)$, and part (i) of Proposition \ref{Omega_nice} implies that $\soc_G J/\Omega\cong \pi$. 
 This implies that $\pi\hookrightarrow J/\Omega$ is an injective envelope of $\pi$, and since these are unique up to isomorphism, we have $J/\Omega\cong J$.
 
Dualizing \eqref{JOmega} we obtain an exact sequence $0\rightarrow J^{\vee}\overset{\bar{x}}{\rightarrow} J^{\vee} \rightarrow \Omega^{\vee}\rightarrow 0$. 
Since Pontryagin duality induces an anti-equivalence of categories between $\Mod^{\mathrm{l.fin}}_{G, \zeta}(k)$ and $\dualcat(k)$,  $J^{\vee}$ is a projective envelope of $\pi^{\vee}$ in $\dualcat(k)$. Since these are unique up to isomorphism, $J^{\vee}$ is isomorphic to $\wP/\varpi\wP$. The same argument shows that 
 $\Omega^{\vee}$ is a projective envelope of $(\soc_K \pi)^{\vee}$ in $\Mod^{\pro}_{K, \zeta}(k)$.  We identify $\wP/\varpi \wP$ with $J^{\vee}$. Since $\wP$ is projective there exists $x\in \wE$ lifting $\bar{x}\in \End_{\dualcat(k)}(J^{\vee})$. We apply the snake lemma 
to the diagram, in which the exact sequence $0\rightarrow \wP\overset{\varpi}{\rightarrow} \wP\rightarrow J^{\vee}\rightarrow 0$ 
is mapped to itself by $x\in \wE$. Since $\Ker \bar{x}=0$ we obtain exact sequences 
$0\rightarrow \Ker x \overset{\varpi}{\rightarrow} \Ker x \rightarrow 0$ and 
$0\rightarrow \Coker x \overset{\varpi}{\rightarrow} \Coker x \rightarrow \Coker \bar{x}\rightarrow 0$. 
Since $\Ker x$ is a compact $\OO$-module, Nakayama's lemma implies that $x$ is injective. The second exact sequence implies that 
$\Coker x$ is $\OO$-torsion free. Since the reduction modulo $\varpi$ of $\Coker x$ is a projective envelope
of $(\soc_K \pi)^{\vee}$ in $\Mod^{\pro}_{K, \zeta}(k)$, $\Coker x$ is a projective envelope
of $(\soc_K \pi)^{\vee}$ in $\Mod^{\pro}_{K, \zeta}(\OO)$ by \cite[Prop.4.6]{comp}.
\end{proof}

 \begin{cor}\label{aok} $\wP$ is projective in $\Mod^{\pro}_{K,\zeta}(\OO)$.
 \end{cor} 
 \begin{proof} It follows from Corollary \ref{proj} and Remark \ref{pk}  that $\wP$ is projective in $\Mod^{\pro}_{K,\zeta}(\OO)$. 
 Alternatively, one may deduce the statement from \cite{eff}.
 \end{proof} 
 
  \section{Proof of the conjecture} \label{proof_of_conj}
  Let $\rho$ and $\psi$ be as in \S\ref{points}. In this section we will construct $N$ satisfying the hypotheses (N0), (N1) and (N2) of \S\ref{points}, and 
  parts (a) and (b) of Theorem \ref{A}. This will allow us to use the formalism developed in \S\ref{general} and prove the Breuil-M\'ezard conjecture. 
  The proof splits into two cases: we will say that $\rho$ is generic, if it is either irreducible, or $\rho\cong \bigl( \begin{smallmatrix}\chi_1 & * \\ 0 & \chi_2\end{smallmatrix}\bigr )$, 
  with $\chi_1\chi_2^{-1}\neq \omega, \Eins$. The assumption implies that $\Ext^1_{G_{\Qp}}(\chi_2, \chi_1)$ is one dimensional, hence $\rho$ is uniquely determined by the ordered pair 
  $(\chi_1, \chi_2)$. Since we assume that $\End_{G_{\Qp}}(\rho)=k$, the only case not covered by the above  is when  $\rho\cong \bigl( \begin{smallmatrix}\chi\omega & * \\ 0 & \chi\end{smallmatrix}\bigr )$. We will refer to this case as non-generic. This division is somewhat artificial, but there are some shortcuts in the generic cases, which we would like to use. We assume that $p\ge 5$.
  
  \subsection{Generic case}\label{generic_case} We assume that $\rho$ is either irreducible, in which case we let $\pi$ be the unique smooth irreducible $k$-representation of $G$ with central character 
  $\zeta$, such that $\VV(\pi)\cong \rho$, or equivalently $\cV(\pi^{\vee})\cong \rho$, or $\rho\cong  \bigl( \begin{smallmatrix}\chi_1 & * \\ 0 & \chi_2\end{smallmatrix}\bigr )$ non-split, with 
  $\chi_1\chi_2^{-1}\neq \Eins, \omega$, in which case we let $\pi:= \Indu{P}{G}{\chi_1\otimes\chi_2\omega^{-1}}$, so that $\cV(\pi^{\vee})=\chi_1$. Let $\wP\twoheadrightarrow \pi^{\vee}$ 
  be a projective envelope of $\pi^{\vee}$ in $\dualcat(\OO)$, and let $\wE:=\End_{\dualcat(\OO)}(\wP)$.
  
  \begin{prop}\label{Nok} $\wP$ satisfies the hypotheses (N0), (N1) and (N2) of \S\ref{points}. Moreover, the functor $\cV$ induces an isomorphism of local rings between 
  $\wE$ and $R^{\psi}_{\rho}$.
  \end{prop}
  \begin{proof} This follows from \cite{cmf}. The representation $k\wtimes_{\wE} \wP$ is a maximal quotient of $\wP$, which contains $\pi^{\vee}$ with multiplicity one.   
  Such representations were computed in \cite{cmf} and denoted by $Q$, see \cite[Rem.1.13]{cmf}. If $\rho$ is irreducible then $k\wtimes_{\wE}\wP \cong \pi^{\vee}$ by \cite[Prop. 6.1]{cmf}.
  If $\rho\cong \bigl( \begin{smallmatrix}\chi_1 & * \\ 0 & \chi_2\end{smallmatrix}\bigr )$, with $\chi_1\chi_2^{-1}\neq \Eins, \omega^{\pm 1}$ then we have an exact non-split sequence: 
  \begin{equation} \label{atome1}
  0\rightarrow (\Indu{P}{G}{\chi_2\otimes\chi_1\omega^{-1}})^{\vee}\rightarrow k\wtimes_{\wE} \wP\rightarrow (\Indu{P}{G}{\chi_1\otimes\chi_2\omega^{-1}})^{\vee}\rightarrow 0
 \end{equation} 
 by \cite[Prop.8.3]{cmf}. If  $\rho\cong \bigl( \begin{smallmatrix}\chi& * \\ 0 & \chi\omega\end{smallmatrix}\bigr )$ then by \cite[Lem.10.16]{cmf} we have an exact non-split sequence
  \begin{equation} \label{atome2}
  0\rightarrow (\tau_1\otimes\chi\circ \det)^{\vee}\rightarrow k\wtimes_{\wE} \wP\rightarrow (\Indu{P}{G}{\chi\omega \otimes\chi\omega^{-1}})^{\vee}\rightarrow 0,
 \end{equation}
 where $\tau_1$ is a smooth representation of $G$ constructed in \cite[\S 10.2, (180)]{cmf}. In particular, we have a non-split sequence 
 $0\rightarrow \Sp\rightarrow \tau_1\rightarrow \Eins^{\oplus 2}\rightarrow 0$. We see that in all three cases (N0) is satisfied.
 
 Since $\pi\hookrightarrow \wP^{\vee}$ is essential, and $\pi^{\SL_2(\Qp)}=0$, we deduce that $(\wP^{\vee})^{\SL_2(\Qp)}$ is zero, as it is stable under $G$. Hence, (N1) is satisfied. 
The assertion that (N2) is satisfied and that $\wE\cong R^{\psi}_{\rho}$ follows from Prop.6.3, Cor.8.7 and Thm.10.72, Cor.10.71 of \cite{cmf}.
 \end{proof}
 
 \begin{remar}\label{socle} If $\rho$ is absolutely irreducible and $\rho|_{I_{\Qp}}\cong (\omega_2^{r+1}\oplus \omega_2^{p(r+1)})\otimes \omega^m$ then 
 $\soc_K \pi\cong (\Sym^r k^2 \oplus \Sym^{p-1-r} k^2\otimes \det^{r})\otimes \det^m$, where $0\le r\le p-1$, $0\le m\le p-2$ and $\omega_2$ is the fundamental 
 character of Serre of niveau $2$, see \cite{breuil1}, \cite{breuil2}. If $\rho\cong \bigl (\begin{smallmatrix} \chi_1& *\\ 0 & \chi_2 \omega^{r+1} \end{smallmatrix}\bigr ) \otimes \omega^m$, 
 where $\chi_1, \chi_2$ are unramified and $\chi_1\neq\chi_2\omega^{r+1}$ then $\pi\cong  (\Indu{P}{G}{\chi_1\otimes \chi_2\omega^r})\otimes\omega^m \circ \det$. Hence, 
 $\soc_K \pi\cong \Sym^r k^2\otimes \det^m$ if $0<r<p-1$ and $\det^m \oplus \Sym^{p-1}k^2\otimes \det^m$, otherwise. In particular, $\soc_K \pi$ is multiplicity free.
  \end{remar}

\begin{remar} It follows from  \eqref{atome1}, \eqref{atome2} and \cite[3.2]{beegee} that $\soc_K \pi\cong \soc_K (k\wtimes_{\wE} \wP)^{\vee}$.
\end{remar}

 \begin{cor}\label{bok} Let $\Theta$ be a $K$-invariant $\OO$-lattice in $\sigma(\mathbf{w}, \tau)$ or $\sigma^{\mathrm{cr}}(\mathbf w, \tau)$. Then 
 $M(\Theta)$, defined in Definition \ref{M_Th} with $N=\wP$, is a Cohen-Macaulay $R^{\psi}_{\rho}$-module of dimension $2$. In particular, 
 $R^{\psi}_{\rho}/\ann M(\Theta)$ is equidimensional and all its associated primes are minimal. 
 \end{cor}
 \begin{proof} Since (N0) holds by Proposition \ref{Nok}, $k\wtimes_{R^{\psi}_{\rho}} \wP$ is a finitely generated $\OO\br{K}$-module.
 The assertions follow from Lemma \ref{MCM} and Proposition \ref{HS3}.  We note that in the definition of a $p$-adic Hodge type $(\mathbf w, \tau, \psi)$ we require 
 that $\psi\varepsilon^{-(a+b)}|_{I_{\Qp}}=\det \tau$.  Using our conventions regarding local class field theory this is equivalent to the central character of $\Theta$ equal to $\zeta$. 
 In particular, $\Theta^d$ is an object of $\Mod^{\pro}_{K, \zeta}(\OO)$. 
  \end{proof}
 
 \begin{cor}\label{Ra} Let $\Theta$ be a $K$-invariant $\OO$-lattice in $\sigma(\mathbf{w}, \tau)$ (resp. $\sigma^{\mathrm{cr}}(\mathbf w, \tau)$) and 
 let $\mathfrak a$ be the 
  $R^{\psi}_{\rho}$-annihilator of $M(\Theta)$. Then $R^{\psi}_{\rho}/\mathfrak  a$  is equal to $R^{\psi}_{\rho}(\mathbf w, \tau)$ (resp. $R^{\psi, \mathrm{cr}}_{\rho}(\mathbf w, \tau)$).
  \end{cor} 
  \begin{proof} Corollary \ref{bok}, Proposition \ref{nearly_done}, Corollary \ref{c_ii_ok} imply that the conditions of Proposition \ref{RHS4} are satisfied.
  It follows from Proposition \ref{RHS4} that $R^{\psi}_{\rho}/\mathfrak a$ is reduced.  The assertion 
  follows from Corollary \ref{exist_pst_def}.
   \end{proof} 
  
  \begin{thm}\label{bm_gen} The ring $R^{\psi}_{\rho}(\mathbf w, \tau)$ is $\OO$-torsion free and of relative dimension $1$ over 
  $\OO$. Moreover, we have an equality of one dimensional cycles:
$$z_1(R^{\psi}_{\rho}(\mathbf w, \tau)/(\varpi))= \sum_{\sigma} m_{\sigma} z_{1}(M(\sigma)),$$
where the sum is taken over the set of isomorphism classes of smooth irreducible $k$-representations of $K$ with central character $\zeta$, 
$m_{\sigma}$ is the multiplicity with which $\sigma$ occurs as a subquotient of $\overline{\sigma(\mathbf w, \tau)}$ and $M(\sigma)$ 
is an $R^{\psi}_{\rho}$-module defined in Definition \ref{M_sigma} with $N=\wP$.  

Further, $M(\sigma)\neq 0$ if and only if $\sigma$ occurs in the $K$-socle of $\pi$. 
If $\Hom_K(\sigma, \pi)\neq 0$ then the Hilbert-Samuel multiplicity of $M(\sigma)$ is equal to $1$.

The same statement holds if we replace $R^{\psi}_{\rho}(\mathbf w, \tau)$  with $R^{\psi, \mathrm{cr}}_{\rho}(\mathbf w, \tau)$ and 
 $\sigma(\mathbf w, \tau)$ with $\sigma^{\mathrm{cr}}(\mathbf w, \tau)$.
 \end{thm}
 
 \begin{proof} The first two assertions follow from Theorem \ref{A}, and Corollary \ref{Ra}. We note that Corollaries \ref{aok}, \ref{bok} 
 imply that parts (a) and (b) of Theorem \ref{A} are satisfied. Proposition \ref{nearly_done},  Corollary \ref{c_ii_ok} imply that part  (c) of Theorem \ref{A} is satisfied, 
 with $\Sigma$ equal to $\MaxSpec R^{\psi}_{\rho}(\mathbf w, \tau)[1/p]$ (resp. $\MaxSpec R^{\psi, \mathrm{cr}}_{\rho}(\mathbf w, \tau)[1/p]$) minus finitely many points. 
 Let $\sigma$ be an irreducible smooth $k$-representation of $K$ with central character $\zeta$.  We apply $\Hom_{\OO\br{K}}^{\cont}(\ast, \sigma^{\vee})^{\vee}$ 
 to the exact sequence $0\rightarrow \wP\overset{x}{\rightarrow} \wP \rightarrow \wP/x\wP\rightarrow 0$, given by Theorem \ref{iks}, to obtain an exact 
 sequence of $R^{\psi}_{\rho}$-modules: 
 \begin{equation}\label{x_reg0}
 0\rightarrow M(\sigma)\overset{x}{\rightarrow} M(\sigma)\rightarrow \Hom_K(\wP/x \wP, \sigma^{\vee})^{\vee}\rightarrow 0
 \end{equation} 
 Since $\wP/x \wP$ is a projective envelope of $(\soc_K \pi)^{\vee}$ we have an isomorphism of $k$-vector spaces: 
 \begin{equation}\label{homs}
 \Hom_{\OO\br{K}}^{\cont}(\wP/x\wP, \sigma^{\vee})^{\vee}\cong \Hom_K(\sigma, \soc_K \pi)\cong \Hom_K(\sigma, \pi).
 \end{equation}
 It follows from \eqref{homs} that  $\Hom_K(\wP/x\wP, \sigma^{\vee})^{\vee}=0$ if and only if 
 $\Hom_K(\sigma, \pi)=0$, in which case Nakayama's lemma implies $M(\sigma)=0$.  
 Since $\soc_K \pi$ is multiplicity free by Remark \ref{socle}, if  $\Hom_K(\sigma, \pi)\neq 0$ then it is a one dimensional $k$-vector space. Then \eqref{x_reg0} and \eqref{homs} give an exact sequence:
 \begin{equation}\label{x_reg}
 0\rightarrow M(\sigma)\overset{x}{\rightarrow} M(\sigma)\rightarrow k \rightarrow 0. 
 \end{equation} 
 Hence, if we let $\mathfrak a$ be the  $R^{\psi}_{\rho}$-annihilator of $M(\sigma)$, then $M(\sigma)\cong R^{\psi}_{\rho}/\mathfrak a$, and
 \eqref{x_reg} implies that $R^{\psi}_{\rho}/\mathfrak a\cong k\br{x}$, and so the Hilbert-Samuel multiplicity of $M(\sigma)$ is $1$.
\end{proof}

\subsection{Non-generic case}\label{non_generic_case} We deal with the case $\rho\cong \bigl( \begin{smallmatrix}\chi\omega & * \\ 0 & \chi\end{smallmatrix}\bigr )$. We will recall 
a construction due to Colmez of $\beta\in \Mod^{\mathrm{l.fin}}_{G, \zeta}(k)$, such that the semi-simplification of $\beta$ is isomorphic to 
$ (\Sp\oplus \Eins\oplus \pi_{\alpha})\otimes \chi\circ \det,$ where $ \pi_{\alpha}:= \Indu{B}{G}{\omega\otimes\omega^{-1}}$,
and $\cV(\beta^{\vee})\cong \rho$.  We will then proceed as in \cite[\S5.7]{cmf} by modifying an argument of Kisin \cite{kisin} to show that $\cV$ induces an
isomorphism between the deformation problems of $\beta^{\vee}$ in $\dualcat(\OO)$ and of $\rho$ with a fixed determinant equal to $\zeta\varepsilon$. 
We will then show that the conditions of Theorem \ref{B} are satisfied with $N$ equal to the universal deformation of $\beta^{\vee}$. This will then allow us to proceed as in \S \ref{generic_case} to deduce an analog of Theorem \ref{bm_gen} in this case. Since the functor $\cV$ is compatible with twisting by characters, we
 may assume that $\chi$ is trivial so that $\rho\cong \bigl( \begin{smallmatrix}\omega & * \\ 0 & 1\end{smallmatrix}\bigr )$ and $\zeta$ is trivial modulo $\varpi$, 
 so that $\Mod^{\mathrm{l.fin}}_{G,\zeta}(k)=\Mod^{\mathrm{l.fin}}_{G/Z}(k)$. 
 
 For $\pi_1, \pi_2 \in  \Mod^{\mathrm{l.fin}}_{G/Z}(k)$ we let $e^1(\pi_1, \pi_2):=\dim_k \Ext^1_{G/Z}(\pi_1, \pi_2)$, 
 where the extensions are computed in $\Mod^{\mathrm{l.fin}}_{G/Z}(k)$. The following 
 dimensions are computed in Theorems 11.4 and 11.5 (ii) in \cite{ext2}, or alternatively see \cite[\S VII.4]{colmez} or \cite{ord2}:
\begin{equation}\label{modp1}
e^1(\Eins, \Eins)=0, \quad e^1(\Sp, \Eins)=1, \quad e^1(\pi_{\alpha}, \Eins)=1,
\end{equation}
\begin{equation}\label{modp2}
e^1(\Eins, \Sp)=2, \quad e^1(\Sp, \Sp)=0, \quad e^1(\pi_{\alpha}, \Sp)=0,
\end{equation}
\begin{equation}\label{modp3}
e^1(\Eins, \pi_{\alpha})=0, \quad e^1(\Sp, \pi_{\alpha})=1, \quad e^1(\pi_{\alpha}, \pi_{\alpha})=2.
\end{equation}
The dimensions of higher $\Ext$-groups are computed in \cite[\S 10.1]{cmf}.  Since $e^1(\pi_{\alpha}, \Eins)=1$, there exists a unique up to
 isomorphism non-split sequence:
\begin{equation}\label{define_kappa}
0\rightarrow \Eins\rightarrow \kappa\rightarrow \pi_{\alpha}\rightarrow 0.
\end{equation} 
 It is shown in \cite[VII.4.18]{colmez} that $\Ext^1_{G/Z}(\Eins, \Sp)\cong \Hom(\Qp^{\times}, k)$. We denote the extension corresponding to 
 $\phi\in\Hom(\Qp^{\times}, k)$ by $E_{\phi}$.

\begin{lem}\label{all_beta} Let $\beta\in \Mod^{\mathrm{l.fin}}_{G, \zeta}(k)$ be such that $\soc_G \beta\cong \Sp$ and the semi-simp\-li\-fi\-ca\-tion of $\beta$ as 
a $G$-representation is isomorphic to $\Sp\oplus \Eins\oplus \pi_{\alpha}$.  Then there exist a non-zero $\phi\in \Hom(\Qp^{\times}, k)$ and non-split exact sequences: 
\begin{equation}\label{beta_phi}
0\rightarrow E_{\phi}\rightarrow \beta\rightarrow \pi_{\alpha}\rightarrow 0,
\end{equation}
\begin{equation}\label{beta_Sp} 
0\rightarrow \Sp \rightarrow \beta\rightarrow \kappa\rightarrow 0.
\end{equation} 
Moreover, $\beta$ determines uniquely the one dimensional $k$-subspace spanned by $\phi$ and every non-zero $\phi$ arises in this way.
\end{lem} 
\begin{proof}  Since $e^1(\pi_{\alpha}, \Sp)=0$ and $\soc_G \beta\cong \Sp$, we deduce that 
$\soc_G (\beta/\Sp)\cong \Eins$, and hence $\beta$ contains a non-split extension of $\Sp$ by $\Eins$.
Since $\Ext^1_{G/Z}(\Eins, \Sp)\cong \Hom(\Qp^{\times}, k)$, every such extension is isomorphic to $E_{\phi}$ for 
a  non-zero $\phi$, which is uniquely determined up to a scalar. The assumption on the semi-simplification of $\beta$ implies that $\beta/E_{\phi}\cong \pi_{\alpha}$, 
which gives us \eqref{beta_phi}. Since $\soc_G( \beta/\Sp)\cong \Eins$ and the semi-simplification of $\beta/\Sp$ is equal 
to $\Eins\oplus \pi_{\alpha}$, we deduce that $\beta/\Sp$ is a non-split extension of $\Eins$ by $\pi_{\alpha}$. Since 
$e^1(\pi_{\alpha}, \Eins)=1$, it must be isomorphic to $\kappa$, which gives \eqref{beta_Sp}. 

Let $\phi\in \Hom(\Qp^{\times}, k)$ be non-zero and let $E_{\phi}$ be the corresponding extension. Since 
$e^1(\pi_{\alpha}, \Sp)=e^2(\pi_{\alpha}, \Sp)=0$ by \cite[\S10.1]{cmf}, we have $e^1(\pi_{\alpha}, E_{\phi})=e^1(\pi_{\alpha}, \Eins)=1$. 
This shows that for every non-zero $\phi$ there exists $\beta$, with $\soc_G \beta\cong \Sp$, and semi-simplification 
$\Sp\oplus\Eins\oplus \pi_{\alpha}$, such that $E_{\phi}$ is a subrepresentation of $\beta$.
\end{proof}

\begin{lem}\label{pi_alpha_beta} $e^1(\pi_{\alpha}, \beta)=0$.
\end{lem}
\begin{proof} Since $e^1(\pi_{\alpha}, \Sp)=0$ and $e^1(\pi_{\alpha}, \kappa)=0$ by \cite[\S10.1, (194)]{cmf}, we deduce the assertion
by applying $\Hom_G(\pi_{\alpha}, \ast)$ to \eqref{beta_Sp}.
\end{proof}

\begin{lem}\label{kappa_beta} $e^1(\kappa, \beta)=1$.
\end{lem} 
\begin{proof} Since $e^1(\kappa, \kappa)=0$ and $e^1(\kappa, \Sp)=2$ by \cite[Lem.10.18]{cmf}, the assertion follows from 
applying $\Hom_{G/Z}(\kappa, \ast)$ to \eqref{beta_Sp}.
\end{proof}

 Let $\beta$ be as in Lemma \ref{all_beta}, we will recall Colmez's description of $\cV(\beta^{\vee})$. Using  local class field theory and some basic group cohomology  we may identify $\Hom(\Qp^{\times}, k)\cong \Hom(G_{\Qp}, k)\cong H^1(G_{\Qp}, k)$. The cup product defines a non-de\-ge\-ne\-rate 
 pairing $H^1(G_{\Qp}, k)\times H^1(G_{\Qp}, k(\omega))\rightarrow H^2(G_{\Qp}, k(\omega))\cong k$. Since $\Hom(\Qp^{\times}, k)$ is a two dimensional $k$-vector 
 space and the pairing is non-degenerate, the subspace orthogonal to $\phi$ is one dimensional, and we denote it by $k\phi^{\perp}$. Since 
 $H^1(G_{\Qp}, k(\omega))\cong \Ext^1_{G_{\Qp}}(\Eins, \omega)$, $k\phi^{\perp}$ defines a non-split extension $0\rightarrow \omega\rightarrow \rho\rightarrow \Eins\rightarrow 0$. It is shown in  \cite[VII.4.24]{colmez} that $\VV(\beta)\cong \rho$, which implies that $\cV(\beta^{\vee})\cong \rho$. The space
 $\Hom(\Qp^{\times}, k)$ has a distinguished line of \textit{unramified} homomorphisms, i.e. homomorphisms, which are trivial on $\Zp^{\times}$, 
 $H^1(G_{\Qp}, k(\omega))$ contains the so called \textit{peu ramifi\'e} line, see \cite[2.4]{serre_duke}. One may show that the two lines are orthogonal to each other. 
 We will say that $\rho$ is \textit{peu ramifi\'e}, when the corresponding element in $H^1(G_{\Qp}, k(\omega))$ lies on the peu ramifi\'e line, otherwise we will say that $\rho$ is \textit{tr\`es ramifi\'e}.
 
From now on $\beta$ will always denote a representation satisfying the conditions of Lemma \ref{all_beta}, and $\rho=\cV(\beta^{\vee})$.  Since one may recover $\phi$ from $\rho$, Lemma \ref{all_beta} implies that $\rho$ determines $\beta$ uniquely. 

\begin{thm}\label{deform_beta} The universal deformation problem of $\beta^{\vee}$ in $\dualcat(\OO)$ is representable by a complete local noetherian ring $R_{\beta^{\vee}}$. The functor 
$\cV$ induces an isomorphism  between the deformation functor of $\beta^{\vee}$ and the deformation functor of $\rho$ with determinant $\varepsilon \zeta$, and hence an isomorphism $R_{\beta^{\vee}}\cong R^{\psi}_{\rho}$.
\end{thm}
\begin{proof} Since $\End_{\dualcat(\OO)}(\beta^{\vee})=k$ and $\Ext^1_{\dualcat(\OO)}(\beta^{\vee}, \beta^{\vee})\cong \Ext^1_{G/Z}(\beta, \beta)$, 
which has dimension at most $3$ as a $k$-vector space by \cite[Lem.10.19]{cmf}, one may show that the deformation functor of $\beta^{\vee}$ is representable by checking Schlessinger's criterion. Arguing as in the proof of \cite[Prop.5.55]{cmf}, we obtain a surjection $R_{\beta^{\vee}}\twoheadrightarrow R^{\psi}_{\rho}$. Let $\mm$ be the maximal ideal of $R_{\beta^{\vee}}$. Since $R^{\psi}_{\rho}\cong \OO\br{x_1, x_2, x_3}$ and $\dim_k \mm/(\mm^2+(\varpi))=\dim_k \Ext^1_{\dualcat(\OO)}(\beta^{\vee}, \beta^{\vee})\le 3$, the surjection is an isomorphism.
\end{proof}

\begin{defi} We let $N$ be the universal deformation of $\beta^{\vee}$ in $\dualcat(\OO)$.
\end{defi}
\begin{remar} Theorem \ref{deform_beta} allows to view $N$ as a compact $R^{\psi}_{\rho}$-module.
\end{remar}

\begin{lem}\label{N_is_ok} The conditions (N0), (N1) and (N2)  hold for $N$ equal to the universal deformation of $\beta^{\vee}$ in $\dualcat(\OO)$.
\end{lem} 
\begin{proof} Since $k\otimes_{R^{\psi}_{\rho}} N\cong \beta^{\vee}$, (N0) is satisfied. We have $\soc_G \beta\cong \Sp$, and hence 
$\Hom_{\SL_2(\Qp)}(\Eins, \beta)=0$. Since $N$ is a deformation of $\beta^{\vee}$, $\{ (N/\mm^i N)^{\vee}\}_i$ is an increasing exhaustive filtration of 
$N^{\vee}$ with graded pieces isomorphic to  finite direct sums of copies of $\beta$, which implies that $\Hom_{\SL_2(\Qp)}(\Eins, N^{\vee})=0$, so that (N1) is satisfied. Theorem \ref{deform_beta} implies that (N2) is satisfied. 
\end{proof}

\begin{defi} We let $\gamma:=(N/(\mm^2,\varpi) N)^{\vee}$, where $\mm$ is the maximal ideal of $R^{\psi}_{\rho}$.
\end{defi}

 \begin{lem}\label{formal_deform1} There exists an exact sequence in $\Mod^{\mathrm{l.fin}}_{G/Z}(k)$: 
 \begin{equation}\label{gamma}
 0\rightarrow \beta\rightarrow \gamma \rightarrow \beta^{\oplus 3}\rightarrow 0.
 \end{equation}
 Moreover, by applying $\Hom_{G/Z}(\beta, \ast)$ to \eqref{gamma} we obtain 
 isomorphisms: 
 \begin{equation}\label{iso_hom}
 \Hom_{G}(\beta, \beta)\cong \Hom_{G/Z}(\beta,\gamma), \quad \Hom_{G}(\beta, \gamma/\beta)\cong \Ext^1_{G/Z}(\beta, \beta).
 \end{equation}
 \end{lem}
 \begin{proof} This is a formal consequence of deformation theory. Let $\overline{R}$, $\overline{\mm}$ and $\overline{N}$ denote 
 the images of $R^{\psi}_{\rho}$, the maximal ideal of $R^{\psi}_{\rho}$ and $N$ modulo $\varpi$. Since $N$ is a deformation of $\beta^{\vee}$ to $R^{\psi}_{\rho}$, 
 it is $R^{\psi}_{\rho}$-flat and $k\otimes_{R^{\psi}_{\rho}} N\cong \beta^{\vee}$. Thus tensoring the exact sequence 
 $0\rightarrow \overline{\mm}/\overline{\mm}^2\rightarrow \overline{R}/\overline{\mm}^2\rightarrow k \rightarrow 0$  with $N$ over $R^{\psi}_{\rho}$, we obtain an 
 exact sequence in $\dualcat(\OO)$: $ 0\rightarrow (\beta^{\vee})^{\oplus d} \rightarrow \overline{N}/\mm^2 \overline{N} \rightarrow \beta^{\vee}\rightarrow 0$, 
 where $d=\dim_k  \overline{\mm}/\overline{\mm}^2$. Since $R^{\psi}_{\rho}\cong \OO\br{x_1, x_2, x_3}$, $d=3$. We obtain \eqref{gamma} by dualizing this exact sequence.

The assertion of \eqref{iso_hom} follows by unraveling the isomorphism $\Ext^1_{\dualcat(k)}(\beta^{\vee}, \beta^{\vee})\cong (\overline{\mm}/\overline{\mm}^2)^*$, where $*$ denotes the  $k$-linear dual, \cite{mazur}.  Namely, any extension $0\rightarrow \beta^{\vee}\rightarrow E\rightarrow \beta^{\vee}\rightarrow 0$  in $\dualcat(k)$ can be viewed as a deformation of $\beta^{\vee}$ to $k[\epsilon]$, where $\epsilon^2=0$. The isomorphism classes of such extensions are in bijection with 
$\Hom(R^{\psi}_{\rho}, k[\epsilon])$, since $R^{\psi}_{\rho}$ is the universal deformation ring of $\beta^{\vee}$. 
If $\varphi\in \Hom(R^{\psi}_{\rho}, k[\epsilon])$  corresponds to $0\rightarrow \beta^{\vee}\rightarrow E\rightarrow \beta^{\vee}\rightarrow 0$, then 
$E\cong k[\epsilon]\otimes_{R^{\psi}_{\rho}, \varphi} N$. In particular, the extension is non-split if and only if $\varphi$ is surjective, and any non-split 
extension is a quotient of $\overline{N}/\mm^2\overline{N}$, as $\mm^2+(\varpi)$ is contained in the kernel of $\varphi$. Hence, the map $$\Hom_{\dualcat(k)}(\overline{\mm}/\overline{\mm}^2\otimes_{R^{\psi}_{\rho}} N, \beta^{\vee})\rightarrow \Ext^1_{\dualcat(k)}(\beta^{\vee}, \beta^{\vee})$$ is surjective, and hence an isomorphism for dimension reasons. Dually we obtain \eqref{iso_hom}.
\end{proof}
 
 \begin{lem}\label{formal_deform} For all  non-zero $G$-subrepresentations $\tau$ of $\beta$ there exists an exact sequence:
 \begin{equation}\label{dimension}
 0\rightarrow \Hom_{G}(\beta, \gamma)\rightarrow \Hom_{G}(\tau, \gamma)\rightarrow \Ext^1_{G/Z}(\beta/\tau, \beta)\rightarrow 0.
 \end{equation}
 \end{lem}
 \begin{proof} Since $\beta$ is multiplicity free and its $G$-socle is irreducible, the natural map  $\Hom_G(\beta, \beta)\rightarrow \Hom_G(\tau, \beta)$ is an isomorphism and $\Hom_{G}(\beta/\tau, \beta)=0$. It follows from \eqref{gamma} that $\Hom_G(\beta/\tau, \gamma)=0$.
 We apply $\Hom_{G/Z}(\beta, \ast)\rightarrow\Hom_{G/Z}(\tau, \ast)$ to \eqref{gamma} to obtain a commutative diagram with exact rows:
 \begin{displaymath}
\xymatrix@1{ \, \Hom_{G}(\beta, \beta)\,\ar[d]^{\cong}\ar@{^(->}[r] & \Hom_{G}(\beta, \gamma )\ar[r]\ar@{^(->}[d] &\Hom_{G}(\beta, \beta^{\oplus 3})\ar@{->>}[r]^{\eqref{iso_hom}}\ar[d]^{\cong}&\Ext^1_{G/Z}(\beta,\beta)\ar[d]\\  \,\Hom_{G}(\tau, \beta)\,\ar@{^(->}[r] & \Hom_{G}(\tau, \gamma )\ar[r] & \Hom_{G}(\tau, \beta^{\oplus 3})\ar[r] & \Ext^1_{G/Z}(\tau,\beta)}
\end{displaymath}
It follows from a diagram chase that the cokernel of the second vertical arrow is equal to the kernel of the fourth vertical arrow.  Since $\Hom_G(\beta, \beta)\rightarrow \Hom_G(\tau, \beta)$ is an isomorphism, the latter is isomorphic to $\Ext^1_{G/Z}(\beta/\tau, \beta)$.
\end{proof} 

For every $\tau\in \Mod^{\mathrm{l.fin}}_{G/Z}(k)$ we let $\II(\tau):=\Hom_{I_1}(\Eins, \tau)\cong \Hom_G(\cIndu{I_1Z}{G}{\Eins}, \tau)$. Then $\II(\tau)$ is  naturally a right  $\HH:=\End_G(\cIndu{I_1Z}{G}{\Eins})$-module. In the following we will make extensive use of the following Theorem due to  Rachel Ollivier.

\begin{thm}[\cite{ollivier}]\label{rachel} If  $\tau\in \Mod^{\mathrm{l.fin}}_{G/Z}(k)$ is generated by its $I_1$-invariants then the natural 
map $\II(\tau)\otimes_{\HH} \cIndu{ZI_1}{G}{\Eins}\rightarrow \tau$ is an isomorphism and 
we have an isomorphism of functors $\Hom_G(\tau, \ast)\cong \Hom_{\HH}(\II(\tau), \II(\ast))$. 
\end{thm}

We let $\pi(0,1): = \frac{\cIndu{KZ}{G}{\Eins}}{T-1}$, where if we identify $\cIndu{KZ}{G}{\Eins}$ with the space of finitely supported functions from the set of vertices of the Bruhat-Tits tree to $k$, 
then $T$ is the endomorphism induced by sending a vertex to the sum of its neighbours.

\begin{lem}\label{I_beta} If $\rho$ is \textit{peu ramifi\'e} then $E_{\phi}\cong \pi(0,1)$ and $\II(\beta)\cong \II(\pi(0,1))$. If $\rho$ is \textit{tr\`es ramifi\'e} then $\II(\beta)\cong \II(\Sp)$.
\end{lem}
\begin{proof} We will use \eqref{define_kappa}, \eqref{beta_phi} and \eqref{beta_phi} extensively in this proof. It is shown in \cite[Lem.10.22]{cmf} that $\II(\kappa)\cong \II(\Eins)$. This implies that $\II(\beta)\cong \II(E_{\phi})$. We have an exact 
sequence $0\rightarrow \II(\Sp)\rightarrow \II(E_{\phi})\rightarrow \II(\Eins)$. It follows by examining the action of $\bigl ( \begin{smallmatrix} \Zp^{\times} & 0 \\ 0 & 1\end{smallmatrix} \bigr)$ on the $I_1$-coinvariants of $E_{\phi}^{\vee}$, see the formula before Theorem VII.4.18 in \cite{colmez}, that if $\phi$ is ramified then the 
map $\II(E_{\phi})\rightarrow \II(\Eins)$ cannot be surjective. Hence, if $\rho$ is \textit{tr\`es ramifi\'e} then $\II(\beta)\cong \II(\Sp)$. 

Since there exists a non-split extension $0\rightarrow \Sp \rightarrow \pi(0,1)\rightarrow \Eins\rightarrow 0$ by \cite[Thm.30(2)(b)]{bl}, there exists $\phi\in \Hom(\Qp^{\times}, k)$ such that $\pi(0,1)\cong E_{\phi}$. Since 
$\pi(0,1)$ is a quotient of $\cIndu{KZ}{G}{\Eins}$, we deduce that $\II(\pi(0,1))\rightarrow \II(\Eins)$ is surjective. Moreover, it follows 
from Theorem \ref{rachel} that the sequence of $\HH$-modules $0\rightarrow \II(\Sp)\rightarrow \II(\pi(0,1))\rightarrow \II(\Eins)\rightarrow 0$ is non-split.
By the previous part, we deduce that $\phi$ cannot be ramified. Since the space of unramified homomorphisms is one dimensional we are done. 
\end{proof}

 \begin{prop}\label{invariants_gamma} If $\rho$ is \textit{tr\`es ramifi\'e} then $\II(\gamma)\cong \II(\Sp)^{\oplus 2}$. If $\rho$ is \textit{peu ramifi\'e} then we have an exact sequence of 
 $\HH$-modules:
 $0\rightarrow \II(\pi(0,1))\oplus \II(\Sp) \rightarrow \II(\gamma)\rightarrow  \II(\pi(0,1)).$
\end{prop} 
\begin{proof} It follows from \eqref{dimension} with $\tau=\Sp$ and Lemma \ref{kappa_beta} that $\Hom_G(\Sp, \gamma)$ is 
$2$-dimensional. This together with Theorem \ref{rachel} implies that $\Hom_{\HH}(\II(\Sp), \II(\gamma))$ is $2$-dimensional.
Moreover, from \eqref{gamma} we obtain an exact sequence:
\begin{equation}\label{I_gamma}
0\rightarrow \II(\beta)\rightarrow \II(\gamma)\rightarrow \II(\beta)^{\oplus 3}.
\end{equation}
If $\rho$ is \textit{tr\`es ramifi\'e} then $\II(\beta)=\II(\Sp)$ by Lemma \ref{I_beta}. Since $\Ext^1_{\HH}(\II(\Sp), \II(\Sp))=0$ by \cite[Lem.5.26]{cmf}, we deduce from 
\eqref{I_gamma} that $\II(\gamma)$ is isomorphic to  a direct sum of copies of $\II(\Sp)$. Since $\Hom_{\HH}(\II(\Sp), \II(\gamma))$ is 
$2$-dimensional, we get $\II(\gamma)\cong \II(\Sp)^{\oplus 2}$.

If $\rho$ is \textit{peu ramifi\'e} then $\pi(0,1)$ is a subrepresentation of $\beta$. It follows from \eqref{dimension} 
with $\tau=\pi(0,1)$ and Lemma \ref{pi_alpha_beta} that $\Hom_{\HH}(\II(\pi(0,1)), \II(\gamma))$ is $1$-dimensional, and hence 
that $\II(\pi(0,1))$ is a submodule of $\II(\gamma)$. Since $\Hom_{\HH}(\II(\Sp), \II(\gamma))$ is $2$-dimensional, we deduce that
$\II(\pi(0,1))\oplus \II(\Sp)$ is a submodule of $\II(\gamma)$ and denote the quotient by $Q$.  Since $\II(\beta)=\II(\pi(0,1))$, it follows
from \eqref{I_gamma} that $Q$ is a submodule of $\II(\pi(0,1))^{\oplus 2} \oplus \II(\Eins)$. It is enough to show that 
$\Hom_{\HH}(\II(\Eins), Q)=0$ and $\Hom_{\HH}(\II(\Sp), Q)$ is at most one dimensional. 

Lemma 5.26 of \cite{cmf} implies that 
$\Ext^1_{\HH}(\II(\Eins), \II(\Eins))=0$ and $\Ext^1_{\HH}(\II(\Eins), \II(\Sp))$ is one dimensional. Since 
the extension $0\rightarrow \II(\Sp)\rightarrow \II(\pi(0,1))\rightarrow \II(\Eins)\rightarrow 0$ is non-split, its class is a basis vector of 
$\Ext^1_{\HH}(\II(\Eins), \II(\Sp))$.  By applying $\Hom_{\HH}(\II(\Eins), \ast)$ to this exact sequence we get that $\Ext^1_{\HH}(\II(\Eins), \II(\pi(0,1)))=0$.
Since $\Hom_G(\Eins, \beta)=0$, we have $\Hom_{\HH}(\II(\Eins), \II(\gamma))=0$. By applying $\Hom_{\HH}(\II(\Eins), \ast)$ to the exact sequence 
$0\rightarrow \II(\pi(0,1))\oplus \II(\Sp)\rightarrow \II(\gamma)\rightarrow Q\rightarrow 0$, we deduce that if $\Hom_{\HH}(\II(\Eins), Q)$ is non-zero, 
then $\II(\gamma)$ contains $\II(\pi(0,1))^{\oplus 2}$, which implies that $\Hom_{\HH}(\II(\pi(0,1)), \II(\gamma))$ is at least two dimensional, leading to 
a contradiction. Lemmas 5.23, 5.26 of \cite{cmf} imply that $\Ext^i_{\HH}(\II(\Sp),\II( \Sp))=0$ for all $i\ge 1$, and $\Ext^1_{\HH}(\II(\Sp), \II(\Eins))$ 
is one dimensional. This implies that $\Ext^1_{\HH}(\II(\Sp), \II(\pi(0,1))$ is one dimensional. By applying $\Hom_{\HH}(\II(\Sp), \ast)$ to 
$0\rightarrow \II(\pi(0,1))\oplus \II(\Sp)\rightarrow \II(\gamma)\rightarrow Q\rightarrow 0$, we deduce that 
that $\Hom_{\HH}(\II(\Sp), Q)$ is at most one dimensional.
\end{proof}

For all $\sigma\in \Irr_{K,\zeta}(k)$ we let $M(\sigma)$ be the $R^{\psi}_{\rho}$-module  defined in Definition \ref{M_sigma} with $N$ equal to the universal 
deformation of $\beta^{\vee}$.

\begin{prop}\label{dim_M_sigma} If  $\rho$ is \textit{tr\`es ramifi\'e}  then $M(\sigma)\neq 0$ if and only if $\sigma\cong \st\cong \Sym^{p-1} k^2$. If $\rho$ is \textit{peu ramifi\'e} then 
$M(\sigma)\neq 0$ if and only if either $\sigma=\Eins$ or $\sigma=\st$. In all cases the Krull dimension of $M(\sigma)$ is at most one.   
\end{prop}
\begin{proof} Let us recall that $\beta\cong (k\otimes_{R^{\psi}_{\rho}} N)^{\vee}$ and $\gamma\cong (N/(\mm^2,\varpi) N)^{\vee}$. If $\rho$ is \textit{tr\`es ramifi\'e} 
then $\II(\beta)\cong \II(\Sp)$ and hence the $G$-representation generated by $I_1$-invariants  of $\beta$ is isomorphic to $\Sp$. This implies that 
$\soc_K \beta\cong \soc_K \Sp\cong \st$. It follows from Corollary \ref{zero_lambda} that $M(\sigma)=0$ if $\sigma\not\cong \st$, and from Proposition 
\ref{disdual} that $k\otimes_{R^{\psi}_{\rho}} M(\st)$ is a one dimensional $k$-vector space, which implies that $M(\st)\neq 0$. Since $\II(\gamma)\cong 
\II(\Sp)^{\oplus 2}$, we deduce by the same argument that $\Hom_K(\st, \gamma)\cong \Hom_G(\cIndu{K}{G}{\sigma}, \gamma)$ is a two dimensional $k$-vector space. Corollary \ref{fgG1} implies that $M(\st)$ is at most one dimensional.

If $\rho$ is \textit{peu ramifi\'e} then $\II(\beta)\cong \II(\pi(0,1))$ and so $\soc_K \beta\cong \soc_K \pi(0,1)\cong \Eins\oplus \st$. The same argument gives that
$M(\sigma)$ is zero if $\sigma\not\cong \Eins$ or $\sigma\not\cong \st$, and non-zero otherwise. Moreover, it follows from Proposition \ref{invariants_gamma}
that $\dim_k \Hom_K(\Eins, \gamma)\le 2$, which allows us to deduce that $M(\Eins)$ is at most one dimensional. Let $\mathfrak K$ be the $G$-normalizer 
of $I$. Then $\mathfrak K$ acts on $\II(\Eins)$ trivially and on $\II(\Sp)$ by the character $\delta: G\rightarrow k^{\times}$, $\delta(g)=(-1)^{v_p(\det g)}$, which is nontrivial as $p>2$. Now $\Hom_G(\cIndu{\mathfrak K}{G}{\Eins}, \beta)\cong \Hom_{\mathfrak K}(\Eins, \beta)\cong \Hom_{\mathfrak K}(\Eins, \pi(0,1))$ is a one dimensional $k$-vector space, and 
it follows from Proposition \ref{invariants_gamma} that $\Hom_G(\cIndu{\mathfrak K}{G}{\Eins}, \gamma)$ is at most two dimensional $k$-vector space. 
We deduce from Corollary \ref{fgG1} that $\Hom_G(\cIndu{\mathfrak K}{G}{\Eins}, N^{\vee})^{\vee}$ is a finitely generated $R^{\psi}_{\rho}$-module of Krull dimension at most one. Similarly, 
$\dim_k \Hom_G(\Sp, \beta)=1$, and $\dim_k \Hom_G(\Sp, \gamma)=2$ as explained in the proof of  Proposition \ref{invariants_gamma}. Hence, 
$\Hom_G(\Sp, N^{\vee})^{\vee}$ is a finitely generated $R^{\psi}_{\rho}$-module of Krull dimension at most one. There exists an exact sequence of $G$-representations: 
\begin{equation}\label{seq_Sp} 
0\rightarrow \cIndu{\mathfrak K}{G}{\Sp^{I_1}\otimes \delta}\rightarrow \cIndu{KZ}{G}{\Sp^{K_1}}\rightarrow \Sp \rightarrow 0
\end{equation} 
Since $\Sp^{K_1}\cong \st$ and $\mathfrak K$ acts on  $\Sp^{I_1}\otimes \delta$ by $\delta^2=\Eins$, by applying $\Hom_G(\ast, N^{\vee})^{\vee}$ to 
\eqref{seq_Sp}, we obtain an exact sequence of $R^{\psi}_{\rho}$-modules: 
\begin{equation}\label{seq_Sp1}
\Hom_G(\cIndu{\mathfrak K}{G}{\Eins}, N^{\vee})^{\vee}\rightarrow M(\st) \rightarrow \Hom_G(\Sp, N^{\vee})^{\vee} \rightarrow 0.
\end{equation} 
This allows us to deduce that $M(\st)$ is at most one dimensional.
 \end{proof}

\begin{prop}\label{proj_dim} $\pdim_{\OO\br{I_1}, \zeta} \beta^{\vee}=3$.
\end{prop}
\begin{proof} By $\pdim_{\OO\br{I_1}, \zeta}$ we mean the length of a minimal projective resolution in $\Mod^{\pro}_{I_1, \zeta}(\OO)$, with the understanding 
that projective objects have projective dimension $0$. Let $Z_1:=I_1\cap Z$ so that $Z_1\cong 1+p \Zp$. Since $p>2$ there exists a continuous character 
$\sqrt{\zeta}: Z_1\rightarrow \OO^{\times}$,  whose square is equal to $\zeta$. Twisting by $\sqrt{\zeta}\circ \det$ induces an equivalence of categories between  
$\Mod^{\pro}_{I_1, \zeta}(\OO)$ and $\Mod^{\pro}_{I_1/Z_1}(\OO)$.
Hence, we may assume that $\zeta$ is trivial, and have to show that $\pdim_{\OO\br{I_1/Z_1}} \beta^{\vee}=3$. Since $\varpi$ kills $\beta^{\vee}$ and 
is $\OO\br{I_1/Z_1}$-regular, we have $\pdim_{\OO\br{I_1/Z_1}} \beta^{\vee}= \pdim_{k\br{I_1/Z_1}} \beta^{\vee} +1$. It is enough to show that $\wTor^i_{k\br{I_1/Z_1}}(\beta^{\vee}, k)=0$ for $i\ge 3$ and 
$\wTor^2_{k\br{I_1/Z_1}}(\beta^{\vee}, k)\neq 0$, see the proof of \cite[\S 19, Lem.1]{matsumura}.  A Pontryagin dual of a minimal projective resolution of $\beta^{\vee}$ in $\Mod^{\pro}_{I_1/Z_1}(k)$ is a minimal injective resolution of $\beta$ in $\Mod^{\sm}_{I_1/Z_1}(k)$. 
Hence, it is enough to show that $H^i(I_1/Z_1, \beta)=0$ for $i\ge 3$ and $H^2(I_1/Z_1, \beta)\neq 0$. This maybe deduced from \eqref{beta_Sp}, since $H^i(I_1/Z_1,\Sp)=0$ for $i\ge 3$ by \cite[Prop.10.5]{cmf} and $H^i(I_1/Z_1, \kappa)=0$ for $i\ge 3$ and $H^2(I_1/Z_1, \kappa)\neq 0$,  by \cite[Lem.10.23]{cmf} via \cite[Lem.5.21]{cmf}.  
\end{proof}

\begin{cor}\label{abok} The conditions of Theorem \ref{B} are satisfied. In particular, the following hold:
\begin{itemize}
\item[(a)] $N$ is projective in $\Mod^{\pro}_{K,\zeta}(\OO)$;
\item[(b)] If $\Theta$ is  a $K$-invariant $\OO$-lattice in $\sigma(\mathbf{w}, \tau)$ or $\sigma^{\mathrm{cr}}(\mathbf w, \tau)$,  then 
 $R^{\psi}_{\rho}/\ann M(\Theta)$ is equidimensional and all its associated primes are minimal, where $M(\Theta)$ is  defined in Definition \ref{M_Th};
\item[(c)] $\max_{\sigma} \{\dim M(\sigma)\}=1$.
 \end{itemize}
 \end{cor} 
 \begin{proof} Let $\chi:\Zp^{\times}\rightarrow \OO^{\times}$ be any character. The statement of the Corollary holds for $N$ if and only if it holds for
 $N\otimes\chi\circ \det$ with $\zeta$ replaced by $\zeta\chi^{-2}$. Hence, we may assume that $\zeta$ is trivial. Propositions \ref{dim_M_sigma} and 
 \ref{proj_dim} imply that $\pdim_{\OO\br{I_1/Z_1}} k\otimes_{R^{\psi}_{\rho}} N + \max_{\sigma} \{\dim M(\sigma)\}\le 4$. Since $R^{\psi}_{\rho}\cong 
 \OO\br{x_1, x_2, x_3}$, it is Cohen-Macaulay of dimension $4$. Moreover, $N$ is $R^{\psi}_{\rho}$-flat, since it is a deformation of $\beta^{\vee}$ to 
 $R^{\psi}_{\rho}$.
 \end{proof} 
 
 \begin{prop}\label{endlich} Let us assume that $\tau=\chi\oplus\chi$ and $\mathbf{w}=(a, a+1)$. Let $\Theta$ be a $K$-invariant $\OO$-lattice in $\sigma(\mathbf{w}, \tau)$ (resp. $\sigma^{\mathrm{cr}}(\mathbf w, \tau)$) and 
 let $\mathfrak a$ be the 
  $R^{\psi}_{\rho}$-annihilator of $M(\Theta)$. Then $R^{\psi}_{\rho}/\mathfrak  a$  is equal to $R^{\psi}_{\rho}(\mathbf w, \tau)$ (resp. $R^{\psi, \mathrm{cr}}_{\rho}(\mathbf w, \tau)$).
 \end{prop}
 \begin{proof} After twisting by a  character, we may assume that $\rho\cong  \bigl( \begin{smallmatrix}\omega & \ast\\ 0 & 1\end{smallmatrix}\bigr )$, 
 $\zeta$ is trivial, $\tau=\Eins\oplus \Eins$, $\mathbf{w}=(0,1)$. In this case $\sigma(\mathbf{w}, \tau)\cong \tilde{\st}$, where $\tilde{\st}$ is the Steinberg representation of $\GL_2(\Fp)$, and $\Theta/\varpi\cong \st$, and $\sigma^{\mathrm{cr}}(\mathbf{w}, \tau)$ is the trivial representation, in which case $\Theta/\varpi\cong \Eins$. Since $N$ is projective in $\Mod^{\pro}_{K,\zeta}(\OO)$ by Corollary \ref{abok}, Lemma \ref{reduce_M} implies that $M(\Theta)/\varpi\cong M(\Theta/\varpi)$ and hence $k\otimes_{R^{\psi}_{\rho}} M(\Theta)\cong  k\otimes_{R^{\psi}_{\rho}} M(\Theta/\varpi)$.
 
 Let us consider the semi-stable case first, so that $\Theta/\varpi\cong \st$. It is shown in the proof of Proposition \ref{dim_M_sigma} that 
 $k\otimes_{R^{\psi}_{\rho}} M(\st)$ is a one dimensional $k$-vector space. Nakayama's lemma implies that $M(\Theta)$ is a cyclic $R^{\psi}_{\rho}$-module, which yields isomorhisms of $R^{\psi}_{\rho}$-modules:
 $$M(\Theta)\cong R^{\psi}_{\rho}/\mathfrak a, \quad M(\st)\cong R^{\psi}_{\rho}/(\varpi, \mathfrak a).$$
Let $\mm$ be the maximal ideal of $R^{\psi}_{\rho}$. If $\rho$ is \textit{tr\`es ramifi\'e} then it is shown in the proof of Proposition \ref{dim_M_sigma} that $\dim_k M(\st)/ \mm^2 M(\st)\le2$. The above isomorphism implies that the embedding dimension of $R^{\psi}_{\rho}/(\varpi, \mathfrak a)$ is at most $1$. 
Since $M(\st)\neq 0$, Corollary \ref{mainproj} and Proposition \ref{dim_M_sigma} imply that the Krull dimension of $R^{\psi}_{\rho}/(\varpi, \mathfrak a)$ is also $1$. Thus $R^{\psi}_{\rho}/(\varpi, \mathfrak a)\cong k\br{x}$. Since $R^{\psi}_{\rho}/\mathfrak a $ is $\OO$-torsion free, we deduce that $R^{\psi}_{\rho}/\mathfrak a\cong \OO\br{x}$. In particular, 
$R^{\psi}_{\rho}/\mathfrak a$ is reduced, and Theorem \ref{exist_pst_def} implies that $R^{\psi}_{\rho}/\mathfrak a= R^{\psi}_{\rho}(\mathbf w, \tau)$.
 
 If $\rho$ is  \textit{peu ramifi\'e} then let $M$ be either the first or the third module in \eqref{seq_Sp1}. It is shown in the proof of of Proposition \ref{dim_M_sigma} that $\dim_k M/ \mm M=1$, so that $M$ is a cyclic $R^{\psi}_{\rho}$-module, and $\dim_k M/\mm^2M\le 2$. Hence, either $M$ is artinian, or the Krull dimension of $M$ is $1$ and 
 the Hilbert-Samuel multiplicity is $1$. We deduce from \eqref{seq_Sp1} that the Hilbert-Samuel multiplicity of $M(\st)$, and hence of $R^{\psi}_{\rho}/(\varpi, \mathfrak a)$, is at most  $2$. Theorem \ref{exist_pst_def} implies that $R^{\psi}_{\rho}(\mathbf w, \tau)$ is a quotient of $R^{\psi}_{\rho}/\mathfrak a$. 
 Corollary \ref{mainproj} implies that $R^{\psi}_{\rho}/(\varpi, \mathfrak a)$ is Cohen-Macaulay, and thus if the surjection $R^{\psi}_{\rho}/(\varpi, \mathfrak a)\twoheadrightarrow R^{\psi}_{\rho}(\mathbf w, \tau)/(\varpi)$ is not an isomorphism then $R^{\psi}_{\rho}(\mathbf w, \tau)/(\varpi)$ has to have a strictly smaller
 Hilbert-Samuel multiplicity. However,   Breuil-M\'ezard computed in \cite[Thm.5.3.1 (i)]{bm} that the Hilbert-Samuel multiplicity of $R^{\psi}_{\rho}(\mathbf w, \tau)/(\varpi)$ is $2$. Thus $R^{\psi}_{\rho}/(\varpi, \mathfrak a)= R^{\psi}_{\rho}(\mathbf w, \tau)/(\varpi)$, and since $R^{\psi}_{\rho}(\mathbf w, \tau)$ is $\OO$-flat, 
 $R^{\psi}_{\rho}/ \mathfrak a= R^{\psi}_{\rho}(\mathbf w, \tau)$. Let us note that as a byproduct we obtain that the Hilbert-Samuel multiplicity of $M(\st)$ is $2$.

We will deal with the crystalline case now, so that $\Theta/\varpi\cong \Eins$. If $\rho$ is \textit{tr\`es ramifi\'e} then it is shown in the proof of Proposition \ref{dim_M_sigma} that $M(\Eins)=0$. Hence, $M(\Theta)/\varpi=0$ and so $M(\Theta)=0$. On the other hand Fontaine-Laffaille theory shows that 
$R^{\psi, \mathrm{cr}}_{\rho}(\mathbf w, \tau)=0$. If $\rho$ is  \textit{peu ramifi\'e} then it is shown in the proof of Proposition \ref{dim_M_sigma} that $\dim_k M(\Eins)/\mm M(\Eins)=1$, and $\dim_k M(\Eins)/\mm^2 M(\Eins)\le 2$. Exactly the same argument as with $M(\st)$ in the \textit{tr\`es ramifi\'e} case shows that
$R^{\psi}_{\rho}/\mathfrak a \cong \OO\br{x}$ and the Hilbert-Samuel multiplicity of $M(\Eins)$ is $1$. Proposition \ref{nearly_done} implies that 
$\Spec R^{\psi}_{\rho}/\mathfrak a$ is a closed subset of 
$\Spec R^{\psi}_{\rho}(\mathbf w, \tau)$, which contains infinitely many potentially crystalline points, corresponding to representations 
$\bigl(\begin{smallmatrix} \delta_2 & \ast \\ 0 & \delta_1\end{smallmatrix}\bigr)$ with $\delta_2\delta_1^{-1}\neq \varepsilon$.
 Hence, $\Spec R^{\psi}_{\rho}/\mathfrak a$ 
is contained in $\Spec R^{\psi, \mathrm{cr}}_{\rho}(\mathbf w, \tau)$. Fontaine-Laffaille theory shows that
$R^{\psi, \mathrm{cr}}_{\rho}(\mathbf w, \tau)\cong \OO\br{x}$. Hence, $R^{\psi}_{\rho}/\mathfrak a=R^{\psi, \mathrm{cr}}_{\rho}(\mathbf w, \tau)$.
\end{proof}

  \begin{thm}\label{bm_non_gen} Let $\rho\cong  \bigl( \begin{smallmatrix}\chi\omega & * \\ 0 & \chi\end{smallmatrix}\bigr )$ and 
 we assume that $\rho$ is  non-split. The ring $R^{\psi}_{\rho}(\mathbf w, \tau)$ is $\OO$-torsion free and of relative dimension $1$ over 
  $\OO$. Moreover, we have an equality of one dimensional cycles:
\begin{equation}\label{sum_cycles_non}
z_1(R^{\psi}_{\rho}(\mathbf w, \tau)/(\varpi))= \sum_{\sigma} m_{\sigma} z_{1}(M(\sigma)),
\end{equation}
where the sum is taken over the set of isomorphism classes of smooth irreducible $k$-representations of $K$ with central character $\zeta$, 
$m_{\sigma}$ is the multiplicity with which $\sigma$ occurs as a subquotient of $\overline{\sigma(\mathbf w, \tau)}$ and $M(\sigma)$ 
is an $R^{\psi}_{\rho}$-module defined in Definition \ref{M_sigma} with $N$ equal to the universal deformation of $\beta^{\vee}$ in $\dualcat(\OO)$.  

The same statement holds if we replace $R^{\psi}_{\rho}(\mathbf w, \tau)$  with $R^{\psi, \mathrm{cr}}_{\rho}(\mathbf w, \tau)$ and 
 $\sigma(\mathbf w, \tau)$ with $\sigma^{\mathrm{cr}}(\mathbf w, \tau)$.
 
Further, $M(\sigma)\neq 0$ if and only if either $\sigma\cong \st\otimes \det^m$ or $\rho$ is \textit{peu ramifi\'e} and $\sigma\cong \det^m$, where $m$ is an integer such that $\chi|_{I_{\Qp}}=\omega^m$. 
If $M(\sigma)\neq 0$ then its Hilbert-Samuel multiplicity is $1$, except if $\rho$ is \textit{peu ramifi\'e} then the Hilbert-Samuel multiplicity of $M(\st\otimes \det^m)$ 
is $2$.
\end{thm}

\begin{proof} The proof of the first two assertions is the same as the proof of Theorem \ref{bm_gen}, using Corollary \ref{abok} instead of Corollaries 
\ref{aok}, \ref{bok}. The statement about the vanishing of $M(\sigma)$ follows from Proposition \ref{dim_M_sigma} taking into account the twist by $\chi$.
The Hilbert-Samuel multiplicities of $M(\sigma)$, when $M(\sigma)\neq 0$, have been computed in the proof of Proposition \ref{endlich}.
\end{proof}

\section{Semi-stable, non-crystalline case}\label{semi_non_crys}
In this subsection we will prove the following result claimed in Theorem \ref{RED_is_ok}. We make no assumption on the prime $p$.

\begin{thm}\label{semi} Let $\Pi\in \Ban^{\adm}_{G, \zeta}(L)$ be absolutely irreducible and non-ordinary. 
If $\Pi^{\alg}\cong \pi\otimes W$, where $\pi$ is a smooth special series 
representation and $W$ is an irreducible algebraic representation then $\Pi$ satisfies (RED).
\end{thm}

To prove the Theorem, we may assume that 
$\Pi^{\alg}\cong \St\otimes\Sm^n L^2$, where $\St$ is the smooth Steinberg representation of $G$ and $\Sm^n L^2=\Sym^n L^2\otimes |\det|^{n/2}$
by twisting by a unitary character.

Let $\Rep^a_G(L)$ be the category of admissible, locally analytic representations of $G$ defined in \cite{adm}. If $\chi: T\rightarrow L^{\times}$ 
is a locally analytic character, we let $I^{\ana}(\chi)=(\Indu{P}{G}{\chi})_{\ana}$  be the space of locally analytic function 
$f:G\rightarrow L$, such that $f(bg)=\chi(b) f(g)$ for all 
$b\in P$ and $g\in G$, where $G$ acts on $I^{\ana}(\chi)$ by right translations. 
The representation $I^{\ana}(\chi)$ equipped with the subspace topology, inherited from $C^{\ana}(G, L)$, is an object of $\Rep^a_G(L)$.

Recall that $v\in \Pi$ is locally analytic if the orbit map $G\rightarrow \Pi$, $g\mapsto gv$ defines a locally analytic function. Let $\Pi^{\ana}$
denote the subset of locally analytic vectors in $\Pi$. 
 Colmez \cite{colmezihp}
and Liu-Xie-Zhang \cite{liu} (for $p>2$) have shown that there exists an exact non-split sequence in $\Rep^a_G(L)$:
\begin{equation}\label{locanvec}
0\rightarrow \Sigma\rightarrow \Pi^{\ana}\rightarrow I^{\ana}(\delta_2^n |\alpha|)\otimes|\det|^{n/2}\rightarrow 0
\end{equation}
where $\delta_1, \delta_2, \alpha: T\rightarrow L^{\times}$ are the characters $\delta_1(\bigl ( \begin{smallmatrix} a & 0 \\ 0 & d \end{smallmatrix}\bigr))= a$, 
$\delta_2(\bigl ( \begin{smallmatrix} a & 0 \\ 0 & d \end{smallmatrix}\bigr))= d$, $\alpha= \delta_1\delta_2^{-1}$ and $\Sigma$ is 
the representation defined by Breuil in \cite{L}. The $G$-socle of $\Sigma$ is isomorphic to $\St\otimes\Sm^n L^2$, the $G$-cosocle is isomorphic to
$\Sm^n L^2$, and we have exact non-split sequences in $\Rep^a_G(L)$:
$$0\rightarrow \Sigma_1\rightarrow \Sigma \rightarrow \Sm^n L^2\rightarrow 0,$$
$$0\rightarrow \St\otimes\Sm^n L^2\rightarrow \Sigma_1\rightarrow I^{\ana}(\delta_1^{n+1}\delta_2^{-1})\otimes|\det|^{n/2}\rightarrow 0.$$
In the notation of \cite{L}, $\Sigma=\Sigma(k, \mathcal L)$ and $\Sigma_1=\Sigma(k)$, where $k=n+2$. We refer to \cite[\S1]{pem} for 
facts on universal unitary completions. The following two propositions are  well known.  

\begin{prop}\label{uni1} The universal unitary completion of $\Sigma$ is isomorphic to $\Pi$.
\end{prop}
\begin{proof} It is shown in \cite[Thm.0.4]{colmez1} that the unitary Banach space representation of $G$, denoted
by $\widetilde{B}(k, \mathcal L)$ in \cite[\S 3.3]{brL}, is non-zero, irreducible and admissible. It follows from 
\cite[Cor.3.3.4]{brL}  that $\widetilde{B}(k, \mathcal L)$ is isomorphic to the Banach space representation of $G$
denoted by $B(k, \mathcal L)$ in \cite{L}. Hence, parts (i), (ii) and (iii) of \cite[Prop.4.4.4]{L} hold and the assertion follows from
\cite[Prop.4.4.4(v)]{L}, the proof of which depends only on parts (i), (ii) and (iii), and not on the Conjecture 4.4.1 in \cite{L}. 
\end{proof}

We denote by $\mathcal L_G$ continuous, $L$-linear, $G$-equivariant homomorphisms.

\begin{prop}\label{uni2} The natural map $\mathcal L_G(\Sigma_1, \Pi_1)\rightarrow \mathcal L_G(\St\otimes\Sm^n L^2, \Pi_1)$ 
is surjective for every unitary Banach space representation $\Pi_1$ of $G$.
\end{prop}
\begin{proof} Propositions 4.6.1 and 4.3.5(i) of \cite{L} imply that the Banach space representation 
denoted by $B(k)$ is the universal unitary completion of $\St\otimes\Sm^{k-2} L^2$. Moreover, it follows from
\cite[Prop.4.3.4]{L} that the natural map $\St\otimes\Sm^{k-2} L^2\rightarrow B(k)$ factors through
$\Sigma(k)\rightarrow B(k)$, where $\Sigma(k)$ is $\Sigma_1$ with $n=k-2$ in our notation. Since every $\phi:
\St\otimes\Sm^{k-2} L^2\rightarrow \Pi_1$ factors through $B(k)$ by the universality, this implies the assertion. 
\end{proof} 

Let $F$ be a finite extension of $\Qp$ contained in $L$ and let $H$ be a finite dimensional locally  $F$-analytic group.
Let $\Rep^a_H(L)$ be the category of admissible locally analytic $H$-representations on $L$-vector spaces in the sense 
of \cite[\S6]{adm}. Recall that sending $V$ to its continuous $L$-linear dual $V'_b$ equipped with the strong topology induces 
an anti-equivalence of categories between $\Rep^a_H(L)$ and $\mathcal C_H$, the category of coadmissible modules over the algebra 
$D(H)=D(H, L)$ of locally analytic distributions on $H$, \cite[Thm. 6.3]{adm}. These categories are abelian.
For $V$ and $W$ in $\Rep^a_H(L)$ we denote by 
$\Ext^i_H(V, W)$ the Yoneda $\Ext$-groups in the category $\Rep^a_H(L)$ and if the centre $Z(H)$ of $H$ acts on $V$ and $W$ by a 
 character $\zeta$, we denote by $\Ext^i_{H, \zeta}(V, W)$ the Yoneda $\Ext$-groups computed in the full subcategory 
of $\Rep^{a}_H(L)$ with objects those representations on which $Z(H)$ acts by $\zeta$. Note that,  if $\mathcal L_H(V, W)=0$
then $\Ext^1_H(V, W)$ and $\Ext^1_{H, \zeta}(V, W)$ coincide. Since $\Rep^a_H(L)$ is anti-equivalent to 
$\mathcal C_H$ passing to duals induces an isomorphism between $\Ext^i_H(V, W)$ and $\Ext^i_{\mathcal C_H}(W'_b, V'_b)$, the  
Yoneda $\Ext$-group in $\mathcal C_H$, \cite[Thm.1, p.121]{alghom}. For each $i\ge 0$ we obtain a map
\begin{equation}\label{lambda}
\lambda^i: \Ext^i_H(V,W)\cong \Ext^i_{\mathcal C_H}(W'_b, V'_b)\rightarrow \Ext^i_{D(H)}(W'_b, V'_b),
\end{equation}
where the last $\Ext$-group is computed in the category of abstract $D(H)$-modules. The discussion after Lemma 3.6 in \cite{adm}
implies that $\lambda^0$ is an isomorphism and hence $\lambda^1$ in an injection. 

\begin{lem}\label{lambda1} Let $V$, $W$ be in $\Rep^a_L(H)$ and suppose that $W'_b$ is a finitely generated $D(H_0)$-module for some compact open
subgroup $H_0$ of $H$ then $\lambda^1$ is an isomorphism. 
\end{lem}
\begin{proof} Let $0\rightarrow V'\rightarrow M\rightarrow W'\rightarrow 0$ be an exact sequence of abstract $D(H)$-modules. 
We have dropped the subscript $b$, since it only indicates the topology on $V'$. By lifting the generators of $W'$ to $M$
we obtain a surjection $D(H_0)^r\twoheadrightarrow W'$, which factors through $D(H_0)^r\rightarrow M \rightarrow W'$. 
We thus obtain a commutative diagram of $D(H_0)$-modules:
\begin{displaymath}
\xymatrix@1{0\ar[r]& N \ar[d]^-{\beta}\ar[r] & V'\oplus D(H_0)^r \ar@{->>}[d]^-{\gamma}\ar[r]& W'\ar[d]^-{=} \ar[r] & 0\\
0 \ar[r] & V'\ar[r] & M\ar[r]& W'\ar[r]& 0 }
\end{displaymath}
Recall that $D(H_0)$ is a Fr\'echet-Stein algebra by \cite[Thm.5.1]{adm}. A direct sum of two coadmissible $D(H_0)$-modules is 
coadmissible and the kernel and  cokernel  of a $D(H_0)$-linear map between coadmissible modules is again coadmissible, \cite[Cor. 3.4 (i), (ii)]{adm}.
Using this we deduce that $N$ is coadmissible, hence $\Ker \beta$ is coadmissible. Since $\Ker \beta\cong \Ker \gamma$ by the snake lemma, and 
$\gamma$ is surjective, we deduce that $M$ is coadmissible. Hence, the extension class lies in the image of $\lambda^1$. Since $\lambda^1$ is 
injective, it is an isomorphism.  
\end{proof}

\begin{remar} (i) If $\Pi$ is an admissible unitary $L$-Banach space representation of $H$ then the continuous $L$-linear dual of the locally 
$F$-analytic vectors in 
$\Pi$ is a finitely generated $D(H_0)$-module for all compact open subgroups $H_0$ of $H$, \cite[Thm.7.1(i)]{adm}, \cite[Prop.3.4]{tosch} or 
\cite[Prop.6.2.4]{em1}.

 (ii) If $W$ is in $\Rep^a_L(H)$ such that $W'$ is finitely generated over $D(H_0)$, then for every admissible subrepresentation 
$U$ of $W$, $U'$ is also finitely generated over $D(H_0)$, since it is a quotient of $W'$.
\end{remar}

Let $G=\GL_2(\Qp)$ and we note that if $\chi: T\rightarrow L^{\times}$ is a locally analytic character, then 
$I^{\ana}(\chi)'$ is finitely generated over $D(H_0)$ for any compact open subgroup $H_0$ of $G$, see \S5 in \cite{la}.
We say  that $\chi$ is locally constant or smooth 
if the kernel is an open subgroup of $T$. In this case, $I^{\sm}(\chi)$ denotes the space of locally constant functions $f:G\rightarrow L$ such that
$f(bg)=\chi(b) f(g)$ for all $b\in P$ and $g\in G$, with $G$-action by the right translations. If $n$ is a non-negative integer the representation 
$I^{\sm}(\chi)\otimes_L \Sym^n L^2$ equipped with the finest locally convex topology is an object of $\Rep^{a}_L(G)$. It follows from 
\cite[\S4]{ugf} that for every smooth character 
 $\chi_{\lc}: P\rightarrow L^{\times}$ and integer $n\ge 0$ we have an exact sequence in $\Rep^a_L(G)$:
\begin{equation}\label{laprinc}
0\rightarrow I^{\sm}(\chi_{\lc})\otimes \Sym^n L^2\rightarrow I^{\ana}(\chi_{\lc} \delta_2^n)\rightarrow I^{\ana}(\chi_{\lc} \delta_1^{n+1}\delta_2^{-1})
\rightarrow 0.
\end{equation}
The last term in \eqref{laprinc} is topologically irreducible by \cite[Thm.6.1]{la}, note that $P$ is the lower triangular 
matrices in \cite{la}.

The referee has pointed out that most of Lemmas \ref{cahara} and \ref{sympr} have been proved by Schraen in \cite[\S 4]{schraen} in a more general setting. 
We include our proof for the convenience of the reader. 

\begin{lem}\label{cahara} Let $\chi, \psi: T\rightarrow L^{\times}$ be locally analytic characters.
If $\psi\neq \chi$ and $\psi\neq \chi \alpha$ then $\Ext^1_P(\chi, \psi)=0$. Moreover, 
\begin{itemize}
\item[(i)] $\Ext^1_P(\chi, \chi)\cong \Hom^{\cont}(T,L)$ is $4$-dimensional;
\item[(ii)] $\Ext^1_{P, \zeta}(\chi, \chi)\cong  \Hom^{\cont}(T/Z,L)$ is $2$-dimensional;
\item[(iii)] $\Ext^1_P(\chi, \chi\alpha)$ is $1$-dimensional.
\end{itemize}
\end{lem}
\begin{proof} We have natural isomorphisms $\Ext^1_P(\chi, \psi)\cong H^1_{\ana}(P, \psi\chi^{-1})$, where $H^1_{\ana}$ is the locally analytic group
cohomology defined in \cite{jan}. If $\chi=\psi$ then this group 
is isomorphic to locally analytic group homomorphisms $\phi: P\rightarrow L$. Since $U$ is contained in the derived subgroup of $P$, every such 
$\phi$ factors through $P/U\cong T$. Since every locally continuous homomorphism $\phi: T\rightarrow L$ is already locally analytic, 
\cite[\S V.9, Thm.2]{lie},  we deduce (i). Part (ii) follows similarly by observing that 
$\Ext^1_{P, \zeta}(\chi, \psi)\cong H^1_{\ana}(P/Z, \psi\chi^{-1})$. Suppose that $\chi\neq \psi$, then the spectral sequence \cite[Thm.6.8]{jan} 
(or an easy calculation with cocycles) allows to identify $H^1_{\ana}(P, \psi\chi^{-1})$ with the group of locally analytic group 
homomorphism $\phi: U\rightarrow L$, such that $\phi(t^{-1}u t)= \chi\psi^{-1}(t) \phi(u)$ for all $t\in T$ and $u\in U$. This implies that 
the restriction of $\chi\psi^{-1}$ to $Z$ is trivial and, if we let $c=\phi(\bigl(\begin{smallmatrix}1 & 1 \\ 0 & 1\end{smallmatrix}\bigr))$, then 
$\phi(\bigl(\begin{smallmatrix}1& \lambda \\ 0 & 1\end{smallmatrix}\bigr))=
\chi^{-1}\psi(\bigl(\begin{smallmatrix}\lambda & 0 \\ 0 & 1\end{smallmatrix}\bigr))c $, for all $\lambda\in \Qp^{\times}$. 
Since $\phi$ is a continuous character, we have $\phi(\bigl(\begin{smallmatrix}1 & \lambda \\ 0 & 1\end{smallmatrix}\bigr))=
\lambda c$ for all $\lambda\in \Zp$. Hence, if $\phi\neq 0$ then 
$\chi^{-1}\psi(\bigl(\begin{smallmatrix}\lambda & 0 \\ 0 & 1\end{smallmatrix}\bigr))=\lambda$, for all $\lambda\in \Zp\setminus \{0\}$, 
which implies that $\chi^{-1}\psi=\alpha$, and the subspace is one dimensional. 
\end{proof} 

\begin{lem}\label{sympr} The following hold:
\begin{itemize} 
\item[(i)] $\Ext^1_G(\Sym^n L^2, I^{\ana}(\delta_2^n|\alpha|))=0$;
\item[(ii)]$\Ext^1_{G,\zeta}(\Sym^n L^2, I^{\ana}(\delta_2^n))$ is $2$-dimensional;
\item[(iii)] $\Ext^1_G(\Sym^n L^2, I^{\ana}(\delta_1^{n+1}\delta_2^{-1}))$ is $1$-dimensional.
\end{itemize}
\end{lem}
\begin{proof} Lemma \ref{lambda1} implies $\Ext^1_G(\Sym^n L^2, I^{\ana}(\psi))\cong \Ext^1_{D(G)}(I^{\ana}(\psi)', (\Sym^n L^2)')$. 
Lemma 8.1 in \cite{jan} says $I^{\ana}(\psi)'\cong D(G)\otimes_{D(P)} \psi^{-1}$ and the argument used in the proof of \cite[Thm.8.18]{jan}
shows that 
$$\Ext^1_{D(G)}(I^{\ana}(\psi)', (\Sym^n L^2)')\cong \Ext^1_{D(P)}(\psi^{-1}, (\Sym^n L^2)').$$
 Since both spaces are finite dimensional, 
by dualizing we obtain that  $$\Ext^1_{D(P)}(\psi^{-1}, (\Sym^n L^2)')\cong \Ext^1_P(\Sym^n L^2, \psi).$$   
Now $\Sym^n L^2$ has a $P$-invariant filtration with the graded pieces isomorphic to $\delta_1^{n-i} \delta_2^i$, $0\le i\le n$.
Lemma \ref{cahara} implies that  if $\psi=\delta_2^n |\alpha|$ then $\Ext^1_P(\Sym^n L^2, \psi)=0$; 
if $\psi=\delta_2^n$ then $\Ext^1_{P, \zeta}(\Sym^n L^2, \psi)\cong \Ext^1_{P, \zeta}( (\Sym^n L^2)_U, \psi)\cong \Ext^1_{P, \zeta}(\psi, \psi)$, which implies (ii);
if $\psi=\delta_1^{n+1}\delta_2^{-1}$ then we have an injection $\Ext^1_P(\Sym^n L^2, \psi)\hookrightarrow \Ext^1_P((\Sym^n L^2)^U, \psi)\cong 
\Ext^1_P(\psi\alpha^{-1}, \psi)$. The injection is an isomorphism, since the target is one dimensional and the source is non-zero, as it contains
 the  extension $0\rightarrow \psi\rightarrow \Sym^{n+1}L^2\otimes \delta_2^{-1}\rightarrow \Sym^n L^2\rightarrow 0$. 
\end{proof} 

\begin{lem}\label{st3} $\Ext^1_{G, \zeta}(\Sym^n L^2, \Sym^n L^2)=0$, $\Ext^1_G(\Sym^n L^2, \St\otimes \Sym^n L^2)$ is one dimensional. 
\end{lem}
\begin{proof} Lemma \ref{sympr}(i) and \eqref{laprinc} imply that $\Ext^1_{G}(\Sym^n L^2, I^{\sm}(|\alpha|)\otimes\Sym^n L^2)=0$. 
It follows from \cite{schst} that there exists a two term projective resolution of $\Eins$ in the category of smooth representations of $G/Z$:
\begin{equation}\label{resolution_trivial}
0\rightarrow \cIndu{\mathfrak K}{G}{ \delta}\rightarrow \cIndu{KZ}{G}{\Eins}\rightarrow \Eins \rightarrow 0,
\end{equation}
where $\delta: G\rightarrow L^{\times}$ is the character $\delta(g)=(-1)^{v_p(\det g)}$ and $\mathfrak K$ is the $G$-normalizer of $I$.
Using \eqref{resolution_trivial}  and Lemma \ref{extalg} we deduce that $\Ext^1_G(\Sym^n L^2, \St\otimes \Sym^n L^2)$ is one dimensional. 
Lemma 2.4.2 of \cite{L} implies that the natural map, induced by \eqref{laprinc}, 
$$\Ext^1_{G, \zeta}(\Sym^n L^2, I^{\ana}(\delta_2^n))\rightarrow \Ext^1_{G, \zeta}(\Sym^n L^2, I^{\ana}(\delta_1^{n+1} \delta_2^{-1}))$$ 
is non-zero. Since the target is $1$-dimensional and the source is $2$-dimensional by Lemma \ref{sympr}, we deduce that 
$\Ext^1_{G, \zeta}(\Sym^n L^2, I^{\sm}(\Eins)\otimes \Sym^n L^2)$ is one dimensional. We claim that the natural map
$$\Ext^1_{G, \zeta}(\Sym^n L^2, I^{\sm}(\Eins)\otimes \Sym^n L^2)\rightarrow \Ext^1_{G, \zeta}(\Sym^n L^2, \St\otimes \Sym^n L^2)$$
is non-zero. Since the source and the target are one dimensional, the claim implies that $\Ext^1_{G, \zeta}(\Sym^n L^2, \Sym^n L^2)$ is zero. 
To prove the claim,  it is enough to construct an extension $0\rightarrow I^{\sm}(\Eins)\rightarrow E\rightarrow \Eins\rightarrow 0$ in the category 
of smooth $G/Z$-representations, such that the sequence $0\rightarrow \St\rightarrow E/\Eins\rightarrow \Eins\rightarrow 0$ is non-split. 
It follows from \eqref{resolution_trivial} that $\Ext^2_{G/Z, \sm}(\Eins, \Eins)$ is zero. Hence, the natural map $\Ext^1_{G/Z, \sm}(\Eins, I^{\sm}(\Eins))\rightarrow \Ext^1_{G/Z, \sm}(\Eins, \St)$ 
is surjective, and the target is non-zero, as the sequence $0\rightarrow \St\rightarrow I^{\sm}(|\alpha|)\rightarrow \Eins\rightarrow 0$ is non-split.    
\end{proof}

\begin{proof}[Proof of Theorem \ref{semi}] We may assume that $\Pi^{\alg}\cong \St\otimes \Sm^n L^2$. Passing to locally analytic vectors 
induces an exact functor $\Ban^{\adm}_{G}(L)\rightarrow \Rep^a_G(L)$,  \cite[Thm.7.1]{adm}. This induces a map 
$\Ext^1_{G, \zeta}(\Pi, \Pi)\rightarrow \Ext^1_{G, \zeta}(\Pi^{\ana}, \Pi^{\ana})$. It follows from Proposition \ref{uni1} that the composition 
\begin{equation}\label{compose}
\Ext^1_{G, \zeta}(\Pi, \Pi)\rightarrow \Ext^1_{G, \zeta}(\Pi^{\ana}, \Pi^{\ana})\rightarrow \Ext^1_{G, \zeta}(\Sigma, \Pi^{\ana})
\end{equation}
is injective. Let $\epsilon$ be the class of an extension $0\rightarrow \Pi \rightarrow E\rightarrow \Pi\rightarrow 0$ in $\Ban^{\adm}_{G, \zeta}(L)$. 
If $\epsilon\in \mathcal E$ then we obtain an exact sequence $0\rightarrow \Pi^{\alg}\rightarrow E^{\alg}\rightarrow \Pi^{\alg}\rightarrow 0$. 
It follows from \cite{schst} that there exists a two term projective resolution of $\St$ in the category of smooth $G/Z$-representations:
\begin{equation}\label{resolution_steiberg}
0\rightarrow \cIndu{\mathfrak K}{G}{\St^{I_1}\otimes \delta}\rightarrow \cIndu{KZ}{G}{\St^{K_1}}\rightarrow \St \rightarrow 0,
\end{equation}
Since $\mathfrak K$ acts trivially on $\St^{I_1}\otimes \delta$ we deduce that there are no extensions of $\St$ by itself in the category of smooth $G/Z$-representations.
Lemma \ref{extalg} implies that there are no extensions of $\St\otimes \Sm^n L^2$ by itself in the category of locally algebraic representations 
with central character $\zeta$. Hence, $E^{\alg}\cong \Pi^{\alg}\oplus \Pi^{\alg}$. Proposition \ref{uni2} implies that the image of 
$\epsilon$ under the composition of \eqref{compose} with the natural map $\Ext^1_{G, \zeta}(\Sigma, \Pi^{\ana})\rightarrow \Ext^1_{G, \zeta}(\Sigma_1, \Pi^{\ana})$
is zero. Hence, \eqref{compose} induces an injection $\mathcal E\hookrightarrow \Ext^1_{G, \zeta}(\Sm^n L^2, \Pi^{\ana})$. It follows 
from \eqref{locanvec} and Lemmas \ref{sympr} and \ref{st3} that this group is isomorphic to
$\Ext^1_{G, \zeta}(\Sm^n L^2, \St\otimes \Sm^n L^2)$, which is one dimensional.
\end{proof}

\section{Extensions of Banach space representations}\label{ext_banach}

The aim of this section is to bound the dimensions of some Yoneda $\Ext^1$ groups computed in $\Ban^{\adm}_{G, \zeta}(L)$. The results are used in the  proof of Proposition \ref{specialize}. Let $\Pi_1, \Pi_2\in \{\Eins, \widehat{\St}, \Pi_{\alpha}\}$, where $\widehat{\St}$ is the universal unitary completion of the smooth Steinberg representation and $\Pi_{\alpha}:=(\Indu{P}{G}{\varepsilon \otimes\varepsilon^{-1}})_{\cont}$. Let $\Theta_1$, $\Theta_2$ be open bounded $G$-invariant lattices in $\Pi_1$ and $\Pi_2$, respectively. Then $\Theta_1/\varpi, \Theta_2/\varpi$ are in $\{ \Eins, \Sp, \pi_{\alpha}\}$. Proposition B.2 of \cite{hauseux} implies that 
 \begin{equation}\label{bound_dimension} \dim_L \Ext^1_{G, \zeta}(\Pi_1, \Pi_2)\le \dim_k \Ext^1_{G, \zeta}(\Theta_1/\varpi, \Theta_2/\varpi).
 \end{equation}

Let $e^1(\Pi_1, \Pi_2):=\dim_L \Ext^1_{G, \zeta}(\Pi_1, \Pi_2)$. Then \eqref{bound_dimension}, \eqref{modp1}, \eqref{modp2}, \eqref{modp3} give:
\begin{equation}
e^1(\Eins, \Eins)=0, \quad e^1(\widehat{\St}, \Eins)\le1, \quad e^1(\Pi_{\alpha}, \Eins)\le1,
\end{equation}
\begin{equation}
e^1(\Eins, \widehat{\St})\le 2, \quad e^1(\widehat{\St}, \widehat{\St})=0, \quad e^1(\Pi_{\alpha}, \widehat{\St})=0,
\end{equation}
\begin{equation}
e^1(\Eins, \Pi_{\alpha})=0, \quad e^1(\widehat{\St}, \Pi_{\alpha})\le1, \quad e^1(\Pi_{\alpha}, \Pi_{\alpha})\le2.
\end{equation}

\end{document}